%% file: article.tex
\documentclass[10pt,a4paper]{amsart}
\usepackage[utf8]{inputenc}
 \usepackage[english]{babel}
 \usepackage{amssymb,bbm,amsmath,amsthm}
 \usepackage{todonotes}
\usepackage{xargs}

\usepackage{caption}
\usepackage{subcaption}

 \usepackage{fourier}
 \usepackage[T1]{fontenc}

 \usepackage{a4wide}

 \usepackage{xspace}

\usepackage[naturalnames]{hyperref} 
\hypersetup{colorlinks=true}
\hypersetup{linkcolor=black,citecolor=black,urlcolor=black,filecolor=black,menucolor=black}

\DeclareRobustCommand{\SkipTocEntry}[5]{}

\usepackage{tikz-3dplot}
\usepackage{tikz}
\usetikzlibrary{shapes}
\usetikzlibrary{patterns}
\usetikzlibrary{matrix}
\usetikzlibrary{arrows}
\usetikzlibrary{mindmap,trees}

\usepackage[ruled,vlined,norelsize]{algorithm2e} 

\SetStartEndCondition{ }{}{}%
\SetKwProg{Fn}{def}{\string:}{}
\SetKwFunction{Range}{range}
\SetKw{KwTo}{in}\SetKwFor{For}{for}{\string:}{}%
\SetKwIF{If}{ElseIf}{Else}{if}{:}{elif}{else:}{}%
\SetKwFor{While}{while}{:}{fintq}%

\AlgoDontDisplayBlockMarkers\SetAlgoNoEnd\SetAlgoNoLine%

\DontPrintSemicolon

\SetKwInOut{Input}{Input}\SetKwInOut{Output}{Output}
\SetKw{If}{if}
\SetKw{Else}{else}
\SetKw{For}{for}
\SetKw{Return}{return}
\SetKwFunction{Xor}{xor}
\SetKwFunction{Flip}{flip}
\SetKwFunction{MacMahon}{macmahon}
\SetKwFunction{Map}{map}
\SetKwFunction{VectorIsForward}{vectorIsForward}
\SetKwFunction{ElimLC}{elimLastCoord}
\SetKwFunction{prim}{prim}
\SetKwFunction{Elim}{elim}
\SetKwFunction{EnumFundPar}{enumFundPar}
\SetKwFunction{SNF}{SmithNormalForm}
\SetKwFor{ForBlock}{for}{}{}
\SetKwIF{IfBlock}{ElseIfBlock}{ElseBlock}{if}{}{else if}{else}{}

\SetKwComment{Comment}{}{}

\SetFuncSty{textsf}
\SetCommentSty{textrm}

\renewcommand{\v}[1]{\ensuremath{{\bf #1}}}

\newcommandx*\fundparreal[1][1=]{\ensuremath{\Pi_{\RR}^{#1}\xspace}}
\newcommandx*\fundparlattice[2][1={\ZZ^d},2=]{\ensuremath{\Pi_{#1}^{#2}\xspace}}

\newcommand{\gfl}[2][]{\ensuremath{\Phi_{#2}^{#1}}}

\newcommand{\omegatwo}{\texttt{\sc Omega2}\xspace}
\newcommand{\omegaone}{\texttt{\sc Omega}\xspace}

\DeclareMathOperator{\omeg}{\Omega_{\geq}}

\DeclareMathOperator{\inv}{inv}

\newcommand{\CCC}{\mathcal{C}}
\newcommand{\OOO}{\mathcal{O}}
\newcommand{\lcone}{\ensuremath{\lambda}-cone\xspace}

\renewcommand{\geq}{\geqslant}
\renewcommand{\leq}{\leqslant}

\newcommand{\QQ}{\ensuremath{\mathbb{Q}}\xspace}
\newcommand{\RR}{\ensuremath{\mathbb{R}}\xspace}

\newcommand{\ZZ}{\ensuremath{\mathbb{Z}}\xspace}
\newcommand{\KK}{\ensuremath{\mathbb{K}}\xspace}
\newcommand{\OO}{\ensuremath{\mathbb{O}}\xspace}
\newcommand{\ZZnn}{\ensuremath{\mathbb{Z}_{\geqslant 0}}\xspace}

\newcommand{\Id}{\operatorname{Id}}
\newcommand{\sm}[5]{#5_{(#1 \rar #2)\times(#3 \rar #4)}}
\newcommand{\sv}[3]{#3_{(#1 \rar #2)}}

\newcommand{\aff}{\operatorname{aff}}
\newcommand{\sgn}{\operatorname{sgn}}
\newcommand{\sg}{\operatorname{sg}}

\newcommand{\bwd}{\operatorname{bwd}}
\newcommand{\rank}{\operatorname{rank}}
\newcommand{\fract}{\operatorname{fract}}
\newcommand{\ind}{\operatorname{ind}}
\newcommand{\mmod}{\;\operatorname{mod} \;}
\newcommand{\momod}[1]{\; \operatorname{mod}^{#1} \;}

\newcommand{\mats}{\ensuremath{\ZZ^{m\times d}}\xspace}

\newcommand{\Ob}[1]{\ensuremath{\mathbb{O}\left(#1\right)}\xspace}

\newcommand{\lgfl}[1]{\ensuremath{\Phi_{#1}^{\lambda}}}
\newcommand{\lgfr}[1]{\ensuremath{\rho_{#1}^{\lambda}}}
\newcommand{\rar}{\rightarrow}

\newenvironment{myquote}{\list{}{\leftmargin=0.3in\rightmargin=0in}\item[]}{\endlist}

 \newcommand{\ldss}{linear Diophantine systems\xspace}

 \newcommandx*\realcone[3][2=, 3=]{\ensuremath{\mathcal{C}_{\RR}^{#3}\left(#1\ifthenelse{\equal{#2}{}}{}{;#2}\right)\xspace}}
\newcommandx*\discretecone[3][2=, 3=]{\ensuremath{\mathcal{C}_{\ZZ}^{#3}\left(#1\ifthenelse{\equal{#2}{}}{}{;#2}\right)\xspace}}
\newcommandx*\saturatedcone[4][2=0,3=, 4={\ZZ^d}]{\ensuremath{\bar{\mathcal{C}}_{\RR,#4}^{#3}\left(#1;#2\right)\xspace}}

\newcommand{\flip}{\operatorname{flip}}
\newcommand{\tcone}{\operatorname{tcone}}
\newcommand{\fcone}{\operatorname{fcone}}

\newcommand{\floor}[1]{\left\lfloor #1 \right\rfloor}
\newcommand{\ceil}[1]{\left\lceil #1 \right\rceil}

\newcommand{\mset}[2]{\left\{ #1 \;\middle|\; #2 \right\}}
\newcommand{\mmatrix}[1]{\begin{pmatrix} #1 \end{pmatrix}}
\newcommand{\msmat}[1]{\left(\begin{smallmatrix} #1 \end{smallmatrix}\right)}
\newcommand{\mvec}[1]{\begin{pmatrix} #1 \end{pmatrix}}
\newcommand{\choice}[1]{\left\{ \begin{array}{ll} #1 \end{array} \right.}

  \newtheorem*{theorem*}{Theorem}
  \newtheorem{theorem}{Theorem}[section]
  \newtheorem{lemma}[theorem]{Lemma}
  \newtheorem{corollary}[theorem]{Corollary}
  \newtheorem{problem}[theorem]{Problem}
  
\input{colors.tex}

\input{graphics.tex}

\newcommand{\polyomega}{{Polyhedral Omega}\xspace}

\newcommand{\sage}{{ Sage}\xspace}
\newcommand{\mathematica}{{ Mathematica}\xspace}
\newcommand{\latte}{{ LattE}\xspace}
\newcommand{\cteuclid}{{CTEuclid}\xspace}

\begin{document}

\title[Polyhedral Omega: Solving Linear Diophantine Systems]{Polyhedral Omega: A New Algorithm for Solving Linear Diophantine
Systems}

\author{Felix Breuer}
\address{Research Institute for Symbolic Computation (RISC)\\
Johannes Kepler University\\
Altenbergerstra{\ss}e 69\\
A-4040 Linz, Austria}
\email{felix@felixbreuer.net}
\urladdr{http://www.felixbreuer.net}
\thanks{Felix Breuer was partially supported by German Research Foundation (DFG) grant BR 4251/1-1 and by Austrian Science Fund (FWF) special research group \emph{Algorithmic and Enumerative Combinatorics} SFB F50-06.}

\author{Zafeirakis Zafeirakopoulos}
\address{Lab of Geometric \& Algebraic Algorithms \\
Department of Informatics \& Telecommunications \\
University of Athens  \\
Panepistimiopolis, 15784, Greece}
\email{zafeirakopoulos@gmail.com}
\urladdr{http://www.zafeirakopoulos.info}
\thanks{Zafeirakis Zafeirakopoulos was partially supported by 
  the Austrian Science Fund (FWF) special research group Algorithmic and Enumerative Combinatorics SFB F50-06
  and W1214-N15 project DK06,  the Austrian Marshall Plan Foundation  
  and by the European Union (European Social Fund – ESF) and Greek national funds through 
  the Operational Program "Education and Lifelong Learning" of the 
  National Strategic Reference Framework (NSRF) - Research Funding Program: THALIS –UOA (MIS 375891).
 }

\begin{abstract}
  Polyhedral Omega is a new algorithm for solving linear Diophantine systems (LDS), i.e., for computing a multivariate rational function representation of the set of all non-negative integer solutions to a system of linear equations and inequalities. \polyomega combines methods from partition analysis with methods from polyhedral geometry. In particular, we combine MacMahon's iterative approach based on the Omega operator and explicit formulas for its evaluation with geometric tools such as Brion decompositions and Barvinok's short rational function representations. In this way, we connect two recent branches of research that have so far remained separate, unified by the concept of symbolic cones which we introduce. The resulting LDS solver \polyomega is significantly faster than previous solvers based on partition analysis and it is competitive with state-of-the-art LDS solvers based on geometric methods. Most importantly, this synthesis of ideas makes \polyomega the simplest algorithm for solving linear Diophantine systems available to date. Moreover, we provide an illustrated geometric interpretation of partition analysis, with the aim of making ideas from both areas accessible to readers from a wide range of backgrounds.
\end{abstract}


{\maketitle}

\vspace{-1cm}

\tableofcontents

\section{Introduction}

In this article we present \polyomega, a new algorithm for computing a multivariate rational function expression for the set of all solutions to a linear Diophantine system. This algorithm connects two branches of research -- partition analysis and polyhedral geometry -- between which there has been little interaction in the past. To make this paper accessible to researchers from either field, as well as to readers with other backgrounds, we give an elementary presentation of the algorithm itself and we take care to motivate the key ideas behind it, in particular their geometric content. To begin, we use this introduction to define the problem of computing rational function solutions to linear Diophantine systems and to give an overview of the algorithms developed in partition analysis and polyhedral geometry to solve it, before pointing out the benefits of \polyomega.

\addtocontents{toc}{\SkipTocEntry}
\subsection*{Linear Diophantine Systems and Rational Functions}

Let $A \in \ZZ^{m \times n}$ be an integer matrix and let $b \in \ZZ^m$ be an integer vector. In this article, we are interested in finding the set of non-negative integer vectors $x \in \ZZ^n_{\geqslant 0}$ such that $Ax \geq b$. Since we are restricting our attention to non-negative solutions $x \in \ZZ^n_{\geqslant 0}$, it is equivalent to consider systems consisting of any combination of equations and inequalities. However, to streamline this article, we are going to focus on the case $Ax\geq b$. We call such a system of constraints, given by $A$ and $b$, a {\emph{linear Diophantine system}} (LDS).

Linear Diophantine systems are of great importance, both in practice and in theory. For example, the Integer Programming (IP) problem -- which is about computing a solution to an LDS that maximizes a given linear functional -- plays a pivotal role in operations research and combinatorial optimization \cite{Schrijver2003}. In this article, we are not interested in finding one optimal solution, however. Instead, we wish to compute a rational function representation of the set of \emph{all} solutions to an LDS.

Of course, the set of solutions to a linear Diophantine system may be infinite. For example, the equation $x_1 - x_2 = 0$, which is equivalent to the system $x_1-x_2 \leq 0$ and $x_1-x_2\geq 0$, has the solution set $\{ (0, 0), (1, 1), (2, 2), \ldots \}$. One way to represent such infinite sets of solutions is via multivariate rational functions. We identify a vector $x \in \ZZ^n$ with the monomial $z^x := z_1^{x_1} \cdot z_2^{x_2} \cdot \ldots \cdot z_n^{x_n}$ in $n$ variables. Then, using the geometric series expansion formula, we can represent the above set of solutions by the rational function
\[ 
  \frac{1}{1 - z^{(1, 1)}} = \frac{1}{1 - z_1z_2} = \sum_{i = 0}^{\infty} z^{(i, i)} = z^{(0, 0)} +
   z^{(1, 1)} + z^{(2, 2)} + \ldots . 
\]
In general, we represent a set $S \subset \ZZ^n_{\geqslant 0}$ of non-negative integer vectors by the multivariate generating function
\[ 
  \phi_S (z) = \sum_{x \in S} z^x, 
\]
i.e., the coefficient $z^x$ in $\phi_S$ is $1$ if $x \in S$ and $0$ if $x \not\in S$. It is a well-known fact that when $S$ is the set of solutions to a linear Diophantine system, then $\phi_S$ is always a rational function $\rho_S$. It is this rational function $\rho_S$ that we seek to compute when solving a given linear Diophantine system. We are not going to be interested in the normal form of $\rho_S$, though, since the normal form may be unnecessarily large. Take for example the system $x_1 + x_2 = 100$. Here, the set of solutions is
\[ 
  S = \{ (100, 0), (99, 1), \ldots, (0, 100) \}
\]
which can be represented by the rational function
\[ 
  \frac{z^{(100, 0)} - z^{(- 1, 101)}}{1 - z^{(- 1, 1)}} = z^{(100, 0)} + z^{
   (99, 1)} + \ldots + z^{(0, 100)} .
\]
Note that the polynomial on the right-hand side is the normal form of this rational function, which has 101 terms. (In fact, the normal form of a rational function representing a finite set will always be a polynomial.) The {\emph{rational function expression}} on the left-hand side is not in normal form since the denominator divides the numerator, however it is much shorter, having only 4 terms. We are therefore interested in computing multivariate rational function expressions, which are not uniquely determined, as opposed to the unique rational function in normal form. Given this terminology and notation, we can now state the computational problem of solving linear Diophantine systems that this article is about.

\newpage

\begin{problem}
\label{prb:rfsLDS}
  Rational Function Solution of Linear Diophantine System (rfsLDS)
  
  {\textbf{Input:}} $A \in \ZZ^{m \times n}$, $b \in \ZZ^m$
  
  {\textbf{Output:}} An expression for the rational function $\rho \in \mathbbm{Q} (z_1, \ldots, z_n)$ representing the set of all non-negative integer vectors $x \in \ZZ^n_{\geqslant 0}$ such that $Ax \geqslant b$.
\end{problem}

Such rational function solutions to LDS are of great importance in many applications. For example, they can be used to prove theorems in number theory and combinatorics \cite{PA3,PA6,MacMahon}, compute volumes \cite{Barvinok1993}, count integer points in polyhedra \cite{Barvinok1994}, to maximize non-linear functions over lattice points in polyhedral \cite{DeLoera2006}, to compute Pareto optima in multi-criteria optimization \cite{DeLoera2009}, to integrate and sum functions over polyhedra \cite{Baldoni2011}, to compute Gr\"obner bases of toric ideals \cite{DeLoera2004}, to perform various operations on rational functions and quasipolynomials \cite{BarvinokWoods}, and to sample objects from polyhedra \cite{Pak2002}, from combinatorial families \cite{Duchon2004} and from statistical distributions \cite{Chen2005}. We recommend the textbooks \cite{BarvinokIntegerPoints,BeckRobins,DeLoera2012} for an introduction.

Note that, while above rfsLDS is stated in terms of pure inequality systems, this restriction is not essential. The methods and algorithm we present in this article can be easily extended to handle the case of mixed systems directly, as does our implementation \cite{PolyhedralOmegaCode}.

\addtocontents{toc}{\SkipTocEntry}
\subsection*{Partition Analysis}

One of the major landmarks in the long history of Problem~\ref{prb:rfsLDS} is MacMahon's seminal work {\emph{Combinatory Analysis}} \cite{MacMahon}, published in 1915, where he presents {\emph{partition analysis}} as a general framework for solving linear Diophantine systems in the above sense, particularly in the context of partition theory. The general approach of partition analysis is to employ a set of explicit analytic formulas for manipulating rational function expressions that can be iteratively applied to obtain a solution to a given linear Diophantine system. We will examine partition analysis in detail in Section~\ref{sec:partition-analysis}; for now, we will confine ourselves to a quick preview to convey the general flavor.

Consider the linear Diophantine inequality 
\begin{eqnarray}
\label{eqn:diophantine-inequality-example}
2 x_1 + 3 x_2 - 5 x_3 &\geqslant& 4
\end{eqnarray}
To solve this system, MacMahon introduces the {\emph{Omega operator}} $\Omega_{\geq}$ defined on rational functions in $\KK (x_1, x_2, x_3) (\lambda)$ in terms of their series expansion by
\[ 
  \Omega_{\geq} \sum_{s = - \infty}^{+ \infty} c_{s_{}} \lambda^s = \sum_{s =
   0}^{+ \infty} c_s 
\]
where $c_s \in \KK (x_1, x_2, x_3)$ for all $s$. In other words, $\Omega_{\geqslant}$ acts by truncating all terms in the series expansion with negative $\lambda$ exponent and dropping the variable $\lambda$. Using the Omega operator, the generating function $\phi_S$ of the set of solutions $S$ of (\ref{eqn:diophantine-inequality-example}) can be written as
\[ 
  \phi_S (z) = \Omega_{\geqslant} \frac{\lambda^{- 4}}{(1 - z_1 \lambda^2) (1
   - z_2 \lambda^3) (1 - z_3 \lambda^{- 5})} . 
\]
The expression on the right-hand side of this equation is known as the {\emph{crude generating function}}.

Partition analysis is about developing calculi of explicit formulas for evaluating the Omega operator when applied to rational functions of a certain form. MacMahon's 1st Rule provides a typical example:
\[ 
  \Omega_{\geqslant} \frac{1}{(1 - \lambda x) \left( 1 - \frac{y}{\lambda^s}
   \right)} = \frac{1}{(1 - x) (1 - x^s y)} . 
\]
Applied iteratively, such formulas allow us to symbolically transform the crude generating function into a rational function expression for $\phi_S$, as desired.

At the turn of the 20th century, Elliott was the first to develop such a calculus that yielded an algorithmic method for solving a linear Diophantine equation. At the same time, MacMahon was pursuing the ambitious and visionary project of applying such methods to partition theory. While algorithmic, Elliott's method was not practical for MacMahon's purposes, as it was too computationally intensive to carry out by hand. Therefore, MacMahon developed and continually extended his own calculus which enabled him to solve many partition theoretic problems, even though it did not constitute a general algorithm. Despite many successes, MacMahon did not achieve the goal he had originally set himself of using his partition analysis calculus to prove his famous formula for plane partitions \cite{PA6}. It is perhaps for this reason that partition analysis fell out of interest for much of the 20th century, before it started to attract renewed attention in recent decades, beginning with Stanley \cite{Stanley1973}.

It was Andrews who observed the computational potential of MacMahon's partition analysis and waited for the right problem to apply it. In the late 1990s, the seminal paper of Bousquet--M{\'e}lou and Eriksson \cite{BME} on lecture-hall partitions appeared. Andrews suggested to Paule to explore the capacity of partition analysis combined with symbolic computation. This collaboration gave rise to an ongoing series of over 10 papers and includes \textsc{Omega} \cite{PA3} and \textsc{Omega2} \cite{PA6}, two fully algorithmic versions of partition analysis powered by symbolic computation. This line of research has also attracted renewed interest in this field, including for example, a sequence of papers by Xin \cite{XinThesis,Xin2004,Xin2012} who develops alternative partition analysis calculi and corresponding software packages. An illustration of the connections between these different solutions to the problem is given in Figure~\ref{fig:partition-analysis-mind-map}.

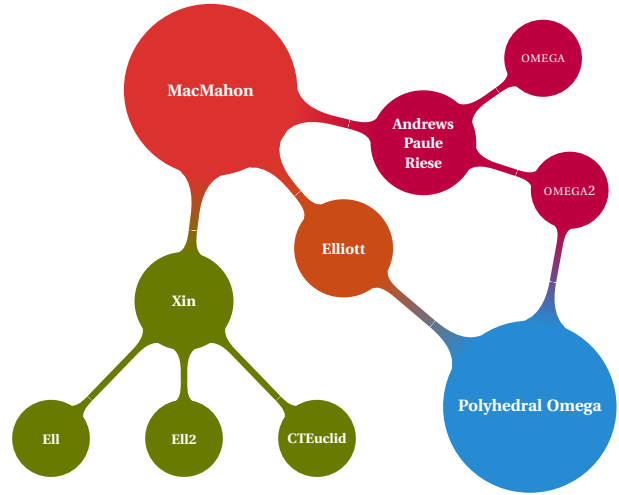
\begin{figure}[t]
\centering
\begin{tikzpicture}[scale=0.7,transform shape]
\tikzset{level 1 concept/.append style={font=\large\bf}};
\tikzset{level 2 concept/.append style={font=\bf}};
\tikzset{level 3 concept/.append style={font=\bf}};
\tikzset{every node/.append style={scale=0.8}};
 \path[mindmap,concept color=red,text=white]
   node[concept] at (0,0) { {\Large\bf MacMahon}}
    child[concept color=green!80!black] { 
      node[concept] at (1,1) {Xin} 
      child { node[concept] at (-1,0.3) {Ell} }
      child { node[concept] at (0,0.3) {Ell2} }
      child { node[concept] at (1,0.3) {CTEuclid} }
      }  
    child[concept color=purple] {
      node[concept]  at (4,4)  { Andrews \\ Paule \\ Riese } 
      child { node[concept] at (3,4.5) {\sc omega} }
      child { node[concept] at (2,2) (omega2) {\sc omega2}  }
    }
    child[concept color=orange] { node[concept] at (1,2) (elliott) {Elliott} };
\path[mindmap,concept color=blue,text=white]
   node[concept] at (6,-6) (polyomega) { {\Large\bf \polyomega}}
 	;
\path (elliott) to[circle connection bar switch color=from (orange) to (blue)] (polyomega);
\path (omega2) to[circle connection bar switch color=from (purple) to (blue)] (polyomega);
%
\end{tikzpicture}
\caption{\label{fig:partition-analysis-mind-map}An overview of the field of partition analysis.}
\end{figure}

\addtocontents{toc}{\SkipTocEntry}
\subsection*{Polyhedral Geometry}

Independently from the work on partition analysis, the field of polyhedral geometry has made major contributions to the solution of linear Diophantine systems. {\emph{Integer linear programming}} is a huge and very active field of research concerned with the computation not of a rational function representation of the set of {\emph{all}} non-negative integer vectors satisfying a given LDS, but instead with the computation of a {\emph{single}} such vector which is optimal with respect to a linear functional. The integer programming problem is distinct from the one we are interested in in this article. Nonetheless, two results on integer programming are important to mention here. On the one hand, deciding satisfiability of linear Diophantine systems is NP-complete, as many important and practical decision problems can be modeled easily using linear Diophantine systems \cite{Schrijver1986,Schrijver2003}. On the other hand, Lenstra has shown in 1983 \cite{Lenstra1983} that if the number of variables is fixed, then the satisfiability of a linear Diophantine system can be decided and (if it is satisfiable) an optimal solution can be found in polynomial time.

The problem of computing a rational function solution to a linear Diophantine system (rfsLDS), which we are concerned with in this article, is also at least as hard as any problem in NP, in the following sense. It is easy to see that NP-hard classes of the satisfiability problem for LDS can be reduced to the problem of computing the unique normal form rational function solution of the LDS. Interpolation arguments then show that moreover {\emph{any}} rational function expression of polynomial size will readily yield the solution to the NP-hard problem in polynomial time.

Thus, the question arises, whether rfsLDS also becomes polynomial time computable if the number of variables is fixed. Lenstra's algorithm does not help here. The key obstacle is that {\emph{a priori}} the encoding length of the output may be exponential in the encoding length of the input. As we have already seen in the example above, the normal form of the rational function solution of $x + y = b$ has $b + 1$ terms while the LDS itself has encoding length $\mathcal{O} (\log b)$. While in this particular case, we can find a rational function expression, $\frac{z^{(b,0)}-z^{(-1,b+1)}}{1-z^{(-1,1)}}$, whose encoding length is in $\mathcal{O} (\log b)$, it is far from clear that such short rational function expressions exist in general. They do, however.

In 1993 Barvinok was able to show in a seminal paper \cite{Barvinok1993} that for every LDS there does exist a rational function expression for its set of solutions whose encoding length is bounded polynomially in terms of the encoding length of the input -- provided that the number of variables is fixed. Moreover, Barvinok gave an algorithm for computing such a \emph{short rational function representation}. The key ingredient that makes these short decompositions possible is Barvinok's method for computing a short signed decomposition of a simplicial cone into unimodular simplicial cones. These \emph{Barvinok decompositions} will play an important role in this article as well.

Barvinok's algorithm for solving rfsLDS combines Barvinok decompositions with a number of other sophisticated tools from polyhedral geometry, ranging from the computation of vertices and edge directions of polyhedra \cite{Fukuda1996,Fukuda1994} to the triangulation of polyhedral cones \cite{DeLoera2010,Pfeifle2003}. This combination makes Barvinok's algorithm very complex and much more involved than the relatively straightforward collection of rules based on explicit formulas that MacMahon had envisioned in the partition analysis framework. A testament of this difficulty is the fact that even though Barvinok's algorithm quickly grabbed the attention of scientific community, it took 10 years and a sequence of additional research papers, before Barvinok's algorithm was first implemented in 2004 by De Loera et al.~\cite{DeLoera2004a}. 

The importance of Barvinok's short rational function representations in this area cannot be overstated, as they have given rise to a whole family of algorithms and theoretical complexity results for a large range of problems, many of which we have mentioned above. Of particular note in our context is that Barvinok and Woods proved a theorem in \cite{BarvinokWoods} which implies in particular that MacMahon's Omega operator can be evaluated in polynomial time in fixed dimension -- not only on crude generating functions as they arise in partition analysis but on a wider class of rational functions.

To conclude our overview of prior literature, we briefly mention a couple of further related publications \cite{BrionVergne,Lasserre2003,Lasserre2005}, a detailed discussion of which is out of scope of this article. \cite{Lasserre2003,Lasserre2005} give two different algorithms for solving rfsLDS via an explicit formula, phrased in terms of a summation over all submatrices of the original system. This can be viewed as a brute force iteration over all simplicial cones formed by the hyperplane arrangement given by the system, together with a recursive algorithm for computing the numerators of the corresponding rational functions. In contrast to an explicit formula given in  \cite{BrionVergne}, the formulas from \cite{Lasserre2003,Lasserre2005} are algorithmically tractable, though the running time can be exponential, even if the dimension is fixed. Since the iterative application of $\omeg$ is not the focus of \cite{Lasserre2003,Lasserre2005}, we do not view these algorithms as part of the partition analysis family for the purposes of this article. To our knowledge these algorithms have not been implemented.

\addtocontents{toc}{\SkipTocEntry}
\subsection*{Our Contribution}

In this article, we bring partition analysis and polyhedral geometry together. Somewhat surprisingly, these two branches of research have so far had relatively little interaction. To bridge this gap, we provide a detailed and illustrated study of these two fields from a unified point of view. In particular, we interpret the partition analysis calculi provided by Elliott, MacMahon and Andrews--Paule--Riese from a geometric perspective. For example, MacMahon's ansatz of setting up the crude generating function corresponds in the language of geometry to a clever way of writing an arbitrary polyhedron as the intersection of a simplicial cone with a collection of coordinate half-spaces -- a construction which we call the MacMahon lifting. While the connection between rational functions and polyhedral cones is well-known and a discussion of the Elliott--MacMahon partition analysis algorithm from this point of view can be found in \cite{Stanley1973}, some of this material is new, including our geometric discussion of Andrews--Paule--Riese and Xin. 

The main contribution of this article is \emph{\polyomega}, a new algorithm for solving linear Diophantine systems which is a synthesis of key ideas from both partition analysis and polyhedral geometry. Just like partition analysis algorithms, \polyomega is an iterative algorithm given in terms of a few simple explicit formulas. The key difference is that the formulas are not given in terms of rational functions, but instead in terms of \emph{symbolic cones}. These are symbolic expressions for simplicial polyhedral cones, given in terms of their generators. Applying ideas like Brion's theorem from polyhedral geometry gives rise to an explicit calculus of rules for intersecting polyhedra with coordinate half-spaces, i.e., for applying the Omega operator on the level of symbolic cones.

After the Omega operator has been applied, the final symbolic cone expression can then be converted to a rational function expression using either an explicit formula based on the Smith Normal Form of a matrix, or using Barvinok decompositions. The Smith Normal Form approach has the advantage of being simpler and easier to implement, while the Barvinok decomposition approach is much faster on many classes of LDSs and produces shorter rational function expressions. However, for many applications it can be advantageous to not convert the symbolic cone expression to a rational function expression at all, since this conversion can increase the size of the output dramatically. Especially when the output is intended for inspection by humans, we strongly recommend working with the symbolic cone expression as it is typically much more intelligible.

\polyomega has several advantages. In comparison with partition analysis methods, \polyomega is much faster. In fact, on certain classes of examples it offers exponential speedups over previous partition analysis algorithms -- even if Barvinok decompositions are not used! -- by virtue of working with symbolic cones instead of rational functions and through an improved choice of transformation rules. In comparison with polyhedral geometry methods, \polyomega is much simpler. It does not require sophisticated algorithms for explicitly computing all vertices and edge-directions of a polyhedron or for triangulating non-simplicial cones. These computations happen implicitly through the straightforward application of a few simple rules given by explicit formulas. At the same time, if Barvinok decompositions are used for the conversion of symbolic cones into rational functions, then Polyhedral Omega runs in polynomial time in fixed dimension, i.e., it lies in the same complexity class as the fastest-known algorithms for the rfsLDS problem. Polyhedral Omega is the first algorithm from the partition analysis family that achieves polynomial run time in fixed dimension. Moreover, Polyhedral Omega is the first partition analysis algorithm for which a detailed complexity analysis is available. This analysis is made possible by our geometric point of view.

On the whole, \polyomega is the overall simplest algorithm for solving linear Diophantine systems, while at the same time its performance is competitive with the current state-of-the-art.

In this article, we give a theoretical description and analysis of \polyomega, together with a detailed motivation and illustration of the central ideas that makes the exposition accessible to readers without prior experience in either partition analysis or polyhedral geometry. The definition of the algorithm is explicit enough to make it possible for interested readers to implement the algorithm even without being experts in the field. Moreover, we have implemented Polyhedral Omega on top of the \sage system \cite{Sage} and the code is publicly available \cite{PolyhedralOmegaCode}.

\addtocontents{toc}{\SkipTocEntry}
\subsection*{Organization of this Article}

In Section~\ref{sec:partition-analysis}, we give a geometric interpretation of previous partition analysis calculi by Elliott, MacMahon, Andrews-Paule-Riese and Xin. In Section~\ref{sec:polyhedral-omega-motivation}, we present the \polyomega algorithm and prove its correctness. We take particular care to motivate and illustrate the key ideas in an accessible manner. In Section~\ref{sec:complexity}, we give a detailed complexity analysis of our algorithm. In Section~\ref{sec:related-work}, we then compare \polyomega to other algorithms from both partition analysis and polyhedral geometry. Finally, we conclude in Section~\ref{sec:conclusion}, where we also point out directions for future research.


\section{Partition Analysis and its Geometric Interpretation}
\label{sec:partition-analysis}

In this section we present some of the major landmarks of partition analysis and interpret the results, which are typically phrased in the language of analysis, from the perspective of polyhedral geometry. 

In the following, we will introduce the required concepts and terminology from polyhedral geometry as we go along. However, a comprehensive introduction to polyhedral geometry is out of scope of this article. We therefore refer the reader to the excellent textbooks \cite{BeckRobins,DeLoera2012,Schrijver1986,Ziegler} for further details.

\subsection{MacMahon}

\subsubsection*{The $\omeg$ operator and the crude generating function}

MacMahon presented his investigations concerning integer partition theory in a series of (seven) memoirs,
published between 1895 and 1916.
In the second memoir \cite{MacMahonPT2}, MacMahon observes that the theory of partitions of numbers
develops in parallel to that of linear Diophantine equations.
In particular, he defines the $\omeg$ operator in Article 66: 

\begin{myquote}
``Suppose we have a function
$
F(x,a)
$
which can be expanded in ascending powers of $x$. 
Such expansion being either finite or infinite, the coefficients of the various powers of $x$
are functions of $a$ which in general involve both positive and negative powers of $a$. 
We may reject all terms containing negative powers of $a$ and subsequently put $a$ equal to unity. 
We thus arrive at a function of $x$ only, which may be represented after Cayley (modified by
the association with the symbol $\geq$) by
$\omeg F(x,a)$ the symbol $\geq$ denoting that the terms retained are those in which the power of $a$ is $\geq 0$.''
\end{myquote}
A modern wording of the definition is given in \cite{PA3}.
The $\omeg$ operator is defined on functions with absolutely convergent multisum expansions 
\begin{equation*}
\sum_{s_1=-\infty}^{\infty}\sum_{s_2=-\infty}^{\infty}\cdots \sum_{s_r=-\infty}^{\infty} A_{s_1,s_2,\ldots,s_r}
\lambda_1^{s_1}\lambda_2^{s_2}\cdots\lambda_r^{s_r}
\end{equation*}
in an open neighborhood of the complex circles $|\lambda_i|=1$.
The action of $\omeg$ is given by
\[
  \omeg \sum_{s_1=-\infty}^{\infty}\sum_{s_2=-\infty}^{\infty}\cdots \sum_{s_r=-\infty}^{\infty} A_{s_1,s_2,\ldots,s_r}
  \lambda_1^{s_1}\lambda_2^{s_2}\cdots\lambda_r^{s_r} := 
  \sum_{s_1=0}^{\infty}\sum_{s_2=0}^{\infty}\cdots \sum_{s_r=0}^{\infty} A_{s_1,s_2,\ldots,s_r}.
\]
MacMahon proceeds with defining what he calls the crude generating function.
This is a (multivariate) generalization of an idea appearing in Cayley as well as in Elliott's work (see next section).
Given a linear system of Diophantine inequalities, MacMahon introduces an extra variable for each inequality.
These extra variables are denoted by $\lambda$.
Given $A\in \mats$ and $b\in \ZZ^m$ we consider the generating function
\[
\lgfl{A,b}(z,\lambda)=\sum_{x\in\ZZnn^d}  z^x \prod_{i=1}^m \lambda_i^{A_i x - b_i}
\]
and based on the geometric series expansion formula $\left( 1 - z \right)^{-1} = \sum_{x \geqslant 0} z^x$ we can transform the series into a rational function. 
The rational form of $\lgfl{A,b}(\v{z})$ is denoted by $\lgfr{A,b}(\v{z})$, i.e.,
\[
\lgfr{A,b}(z,\lambda)=\lambda^{-b}\prod_{i=1}^m \frac{1}{\left( 1-z_i\lambda^{\left({A^T}\right)_{i}}\right)}.
\]
We call $\lgfr{A,b}(z,\lambda)$ the \emph{$\lambda$-generating function} of $A$ and $b$. 
The \emph{crude generating function} is the syntactic expression
\[
\omeg \left( \lambda^{-b}\prod_{i=1}^m \frac{1}{\left( 1-z_i\lambda^{\left({A^T}\right)_{i}}\right)} \right)
\]
obtained by prepending $\omeg$ to the $\lambda$-generating function. We present an example that we also use for the geometric interpretation.
Let $A=
[ \begin{array}{cc}
2 & 3 \\
\end{array} ]
$ and $b=5$ and consider the Diophantine system $Ax \geqslant b$. Then
\[
  \lgfl{A,b}(z_1,z_2,\lambda)=\sum_{x_1,x_2\in\ZZnn} \lambda^{2 x_1+3 x_2 -5 } z_1^{x_1} z_2^{x_2}
  \;\;\;\text{ and }\;\;\;
  \lgfr{A,b}(z_1,z_2,\lambda)=\frac{\lambda^{-5}}{(1-z_1\lambda^2)(1-z_2\lambda^3)}.
\]
Thus, the crude generating function for this system is 
\[
\omeg\frac{\lambda^{-5}}{(1-z_1\lambda^2)(1-z_2\lambda^3)}
\]
After defining the crude generating function, MacMahon proceeds to evaluate this expression by the use of formal rules, such as 
\[ \Omega_{\geqslant} \frac{1}{(1 - \lambda x) \left( 1 - \frac{y}{\lambda^s}
   \right)} = \frac{1}{(1 - x) (1 - x^s y)} . \]
In what follows we will explain what this means geometrically.

\subsubsection*{Cones and fundamental parallelepipeds.}

The connection between partition analysis and geometry hinges on the fact that certain rational function expressions correspond directly to generating functions of lattice points in polyhedral cones. A \emph{lattice point} is any integer linear combination of a fixed set of basis vectors. Without loss of generality, in this article, we will always work with the the standard unit vectors as our basis, so that ``lattice point'' is synonymous with ``integer point''.

We already used $\Phi$ to denote the formal power series generating function for \ldss. More generally, for a set $S\subseteq \ZZ^d$ of lattice points we will use $\Phi_S$ to denote the generating function of $S$, i.e., $\Phi_S= \sum_{s\in S} z^s$. For convenience, if $S\subseteq \RR^d$ is a set of real vectors, we will also write $\Phi_S$ to denote the generating function of all lattice points in $S$, namely $\Phi_S= \sum_{s\in S\cap\ZZ^d} z^s$. Of course, if and where $\Phi_S$ converges depends on the set $S$, but these questions will be of no particular concern to us, as we will see below. We are primarily interested in generating functions $\Phi_C$ where $C$ is a simplicial polyhedral cone.

We define the \emph{cone $C$ generated by $v_1,v_2,\ldots,v_d$} as  
\[
\realcone{v_1,\ldots,v_d} = \mset{\sum_{i=1}^d \lambda_i v_i}{0 \leqslant \lambda_i \in\RR}.
\]
The $v_i$ are the \emph{generators} of $C$. If the $v_i$ are linearly independent, the cone $C$ is \emph{simplicial}.  All cones in this article are \emph{rational polyhedra}, which means that they can also be defined as the set of real vectors satisfying a finite system of linear inequalities with integer coefficients. A cone $C$ is \emph{pointed} or \emph{line-free} if it does not contain a line. All simplicial cones are pointed. For pointed cones $C$ the power series $\Phi_C$ always has a non-trivial radius of convergence, see also Section~\ref{sec:xin-convergence}. If a given cone $C$ is pointed and the chosen set of generators $v_i$ is inclusion-minimal, then the sets $\RR_{\geq 0}v_i$ are called the \emph{extreme rays} of the cone. The set of extreme rays $R_i$ of a pointed cone is uniquely determined, but within each ray the choice of generators $0\not=v_i\in R_i$ is arbitrary. A vector $v\in\ZZ^n$ is \emph{primitive}, if the greatest common divisor of its entries is 1. For every vector $0\not=v\in\QQ^n$ there exists a unique primitive vector $\prim(v)$ that is a positive multiple of $v$. Equivalently, $\prim(v)$ is the unique shortest non-zero integer vector on the ray $\RR_{\geq 0}v$. With these normalizations we find that each pointed rational cone has a unique minimal set of primitive generators.

Later, we are also going to need the notion of \emph{affine cones}, which just means cones like those defined above translated by a rational vector $q$. For translates of pointed cones, the translation vector $q$ is uniquely determined, in which case $q$ is called the \emph{apex} of the affine cone. We will frequently drop the adjective ``affine'' and just refer to cones in general. It will be clear from context whether a given cone has its apex at the origin or not. For a given $q\in\QQ^d$ we adopt the notation
\[
  \realcone{v_1,\ldots,v_d}[q] = q + \realcone{v_1,\ldots,v_d}.
\]
Here, $+$ denotes the \emph{Minkowski sum}, i.e., for two sets $A$ and $B$ we have $A+B = \mset{a+b}{a\in A, b\in B}$.

If instead of considering all conic combinations of the generators with real coefficients, 
we consider only non-negative integer combinations, then we obtain the semigroup
\[
\discretecone{v_1,\ldots,v_d} = \mset{\sum_{i=1}^d \lambda_i v_i}{0 \leqslant \lambda_i \in\ZZ}.
\]
which we also call the \emph{discrete cone} generated by the $v_i$. As above we write $\discretecone{v_1,\ldots,v_d}[q]:=q+\discretecone{v_1,\ldots,v_d}$.

\newcommand{\figscale}{0.6}
\begin{figure*}[t]
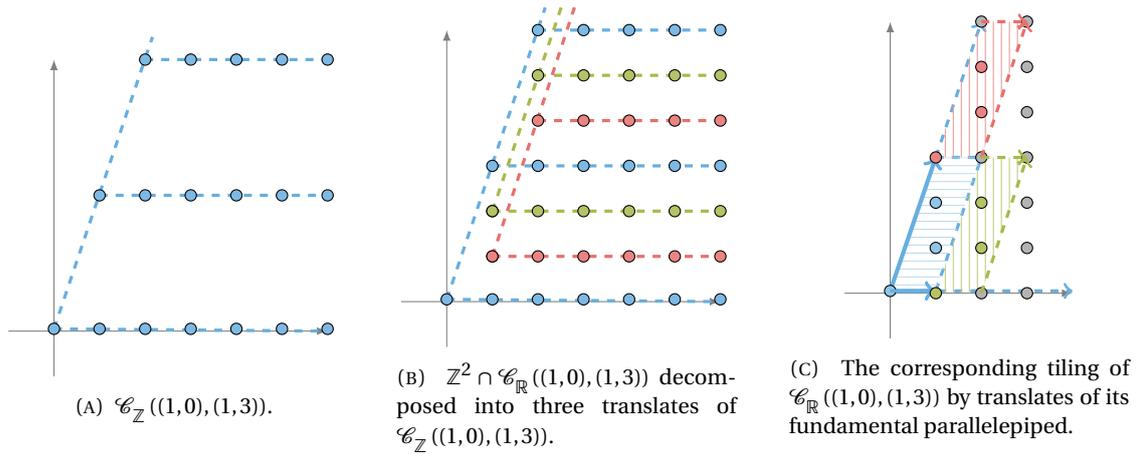

  \begin{subfigure}{0.3\textwidth}
   \centering
   \input{Graphics/genfun_a.tex}
   \caption{\label{fig:genfun:a} $\discretecone{(1,0),(1,3)}$.}
  \end{subfigure}
  \qquad
  \begin{subfigure}{0.3\textwidth}
  \centering
   \input{Graphics/genfun_b.tex}
   \caption{\label{fig:genfun:b} $\ZZ^2\cap\realcone{(1,0),(1,3)}$ decomposed into three translates of $\discretecone{(1,0),(1,3)}$.}
  \end{subfigure}
  \qquad
  \begin{subfigure}{0.3\textwidth}
  \centering
   \input{Graphics/genfun_d.tex}
   \caption{\label{fig:genfun:d} The corresponding tiling of $\realcone{(1,0),(1,3)}$ by translates of its fundamental parallelepiped.}
  \end{subfigure}  
  \caption{Generating functions of cones.}
  \label{fig:genfun}  
\end{figure*}

As an example, consider the semigroup generated by $(1,0)$ and $(1,3)$ as shown in Figure~\ref{fig:genfun:a}. 
If we take the geometric series $\frac{1}{\left(1-z_1\right)}$ and expand it, i.e., $1+z_1+z_1^2+\ldots$, we observe that this series is the generating function of the lattice points on the non-negative $x$-axis. 
In order to take the generating function for the lattice points at height $3$, we need to multiply the series by $\left(z_1 z_2^3\right)$. 
Accordingly, to take the lattice points at height $6$ we multiply by $\left(z_1 z_2^3\right)^2$ and so on. 
It is easy to see, that the generating function of the semigroup is the product of two geometric series, namely $\frac{1}{\left(1-z_1\right)}$ and $\frac{1}{\left(1-z_1 z_2^3\right)}$, so that we obtain
\[
  \gfl{\discretecone{(1,0),(1,3)}} = \frac{1}{\left(1-z_1\right)\left(1-z_1 z_2^3\right)}. 
\]

Following the same idea, we can translate the whole semigroup by multiplying by a monomial. 
If we multiply $\frac{1}{\left(1-z_1\right)\left(1-z_1 z_2^3\right)}$ by $z_1z_2$ we obtain the generating function of the red points in Figure~\ref{fig:genfun:b}. 
Similarly, we multiply by $z_1z_2^2$ to obtain the generating function of the green points. 
Now, since there are no overlaps, we can sum the generating functions in order to obtain the generating function for all lattice points in the cone generated by $(1,0)$ and $(1,3)$ with real coefficients,
\[ 
 \gfl{\realcone{(1,0),(1,3)}} 
 = \left(1 + z_1z_2 + z_1z_2^2\right) \cdot \gfl{\discretecone{(1,0),(1,3)}} =\frac{1+z_1z_2+z_1z_2^2}{\left(1-z_1\right)\left(1-z_1 z_2^3\right)}.
\]
The numerator of this rational function encodes the lattice points $(0,0)$, $(1,1)$ and $(1,2)$ which are precisely the lattice points in the parallelogram $\Pi = \mset{\lambda_1(1,0) + \lambda(1,3)}{0\leqslant\lambda_i < 1}$
as shown in blue in Figure~\ref{fig:genfun:d}. Moreover, the cone $\realcone{(1,0),(1,3)}$ is tiled by translates of $\Pi$. This observation generalizes.

Given a simplicial cone $C=\realcone{v_1,v_2,\ldots,v_d}\in\RR^d$, we define its \emph{fundamental parallelepiped} as
\[
\fundparreal(C)=\fundparreal(v_1,\ldots,v_d)=\mset{\sum_{i=1}^d \lambda_i v_i}{0 \leqslant \lambda_i < 1, \lambda_i\in\RR}.
\]
and the set of lattice points in the fundamental parallelepiped as $\fundparlattice(C)=\fundparreal(C)\cap\ZZ^d$. Given this definition, the general form of the generating function of a simplicial cone $C=\realcone{v_1,\ldots,v_m}\subset\RR^d$ generated by $v_1,\ldots,v_m\subset\ZZ^d$ is
\begin{eqnarray}
\label{eqn:Ehrhart}
 \gfl{C}(z_1,\ldots,z_d) &=& \frac{\gfl{\fundparlattice (C)}(z_1,\ldots,z_d)}{\prod_{i=1}^m \left(1-z^{v_i}\right)} 
\end{eqnarray}
where $\gfl{\fundparlattice (C)}(z_1,\ldots,z_d)$ is the generating function for the lattice points in the fundamental parallelepiped of $C$. This fundamental observation goes back to Ehrhart \cite{ehrhartpolynomial,ehrhart1,ehrhart2}. From this identity it is immediately obvious that cones $C$ with few lattice points in their fundamental parallelepiped $\fundparlattice(C)$ have particularly short representations as rational functions. Cones $C$ where $\fundparlattice(C)$ contains just a single lattice point are called \emph{unimodular}.

An important property of the fundamental parallelepiped is that its lattice points are coset representatives for the lattice points in the cone modulo the semigroup of the generators (as exemplified by the construction of the generating function). To put it in another way: the set of lattice points in the cone is the set of lattice points in the fundamental parallelepiped, translated by all non-negative integral combinations of the generators of the cone, i.e.,
\begin{eqnarray}
\label{eqn:Ehrhart-tiling}
  C\cap\ZZ^d &=& \bigsqcup_{i_1,i_2,\ldots,i_n\in \ZZnn} \left(\sum_{j=1}^{n} i_j v_j + \fundparlattice(C) \right) \;\;\; = \;\;\; \discretecone{v_1,\ldots,v_n} + \fundparlattice(C)
\end{eqnarray}
Here $\bigsqcup$ denotes an ordinary union together with the assertion that the operands are disjoint.

\subsubsection*{Geometry of the $\omeg$ operator and the MacMahon lifting}

Now we can see the connection of MacMahon's construction to polyhedral geometry.
In Figure~\ref{fig:cayley}, we see the geometric steps for creating the \lcone (equivalent of the $\lambda$-generating function) by means of the MacMahon lifting and then applying the $\omeg$ operator.  

\newcommand{\mmtricksize}{0.46}
\begin{figure*}[t]
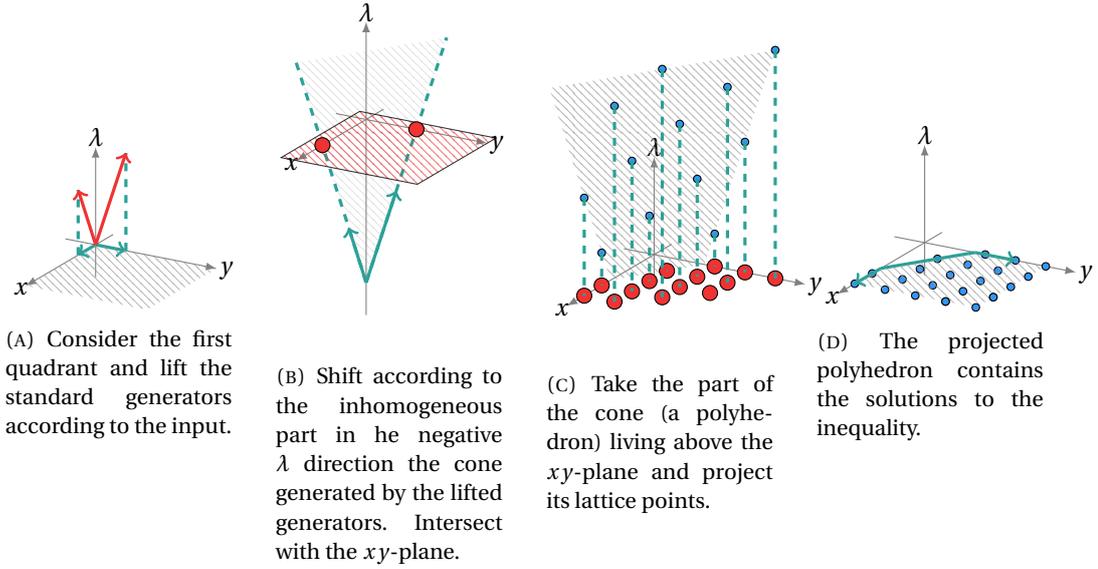

  \begin{subfigure}{0.20\textwidth}
   \centering
   \input{Graphics/mmtricka.tex}
   \caption{\label{fig:cayleyA}Consider the first quadrant and lift the standard generators according to the input.}
  \end{subfigure}
  \;\;\;\;\;
  \begin{subfigure}{0.20\textwidth}
  \centering
   \input{Graphics/mmtrickb.tex}
   \caption{\label{fig:cayleyB}Shift according to the inhomogeneous part in he negative $\lambda$ direction the cone generated by the lifted generators. Intersect with the $xy$-plane.}
  \end{subfigure}
  \;\;\;\;\;
  \begin{subfigure}{0.20\textwidth}
  \centering
   \input{Graphics/mmtrickd.tex}
   \caption{\label{fig:cayleyC}Take the part of the cone (a polyhedron) living above the $xy$-plane and project its lattice points.}
  \end{subfigure}
  \;\;\;\;\;
  \begin{subfigure}{0.20\textwidth}
  \centering
   \input{Graphics/mmtricke.tex}
   \caption{\label{fig:cayleyD}The projected polyhedron contains the solutions to the inequality.}
  \end{subfigure}  
  \caption{\label{fig:cayley} $\omeg$ as lifting and projection.}
\end{figure*}
  
Given an LDS $Ax\geqslant b$ with $m$ inequalities and $d$ variables, we first lift the standard generators of the positive orthant of $\RR^d$ by appending the exponents of the $\lambda$ variables (thus lifting the generators in $\RR^{d+m}$). This lifting idea first appears in an analytic context in Cayley's work \cite{Cayley}.
Consider the cone generated by the MacMahon lifting, i.e., the cone generated by the columns of matrix $V= \msmat{\operatorname{Id}_d \\ A}$.
The generating function of this cone, due to (\ref{eqn:Ehrhart}), is exactly $\frac{1}{\left( 1-z_i\lambda^{\left({A^T}\right)_{i}}\right)}$.
This is true because the matrix $V$ is unimodular by construction, whence the fundamental parallelepiped of the cone contains exactly 
one lattice point, namely the origin.
Next, we translate the cone by $q= \msmat{0 \\ -b}$ and denote the new cone by $C = \CCC(V) + \msmat{0 \\ -b}$. Then the generating function of cone $C$ is exactly the $\lambda$-generating function, i.e.,
\[
\Phi_C(z) = \rho^{\lambda}_{A,b}(z) 
= 
\lambda^{-b}\prod_{i=1}^m 
  \frac{1}{
    \left( 1-z_i\lambda^{\left({A^T}\right)_{i}} \right)
    }.
\]

Applying the Omega operator to $\Phi_C$ now corresponds geometrically to intersecting $C$ with all half-spaces where the $\lambda$-variables are non-negative and then applying a projection that forgets the $\lambda$-variables. To make this precise, let us denote by $H_{\lambda}$ the intersection of coordinate halfspaces where all $\lambda$ variables are non-negative, i.e., 
\[
  H_{\lambda}=\mset{\left(x_1,x_2,\ldots,x_{d},x_{\lambda_1},x_{\lambda_2},\ldots,x_{\lambda_m}\right)\in\RR^{d+m}}{x_{\lambda_i}\geqslant 0 \text{ for } 1\leqslant i\leqslant m}.
\]
Next, let $\pi$ denote the projection from $\RR^{d+m}$ to $\RR^d$ that forgets the last $m$ coordinates which correspond to $\lambda$-variables. Given this notation, the action of $\omeg$ on the $\lambda$-generating function corresponds to first intersecting $C$ with the $H_\lambda$ and then projecting the resulting polyhedron using $\pi$, i.e.,
\[
\omeg\Phi_C(z) = \omeg\rho^{\lambda}_{A,b}(z)= \Phi_{\pi (C\cap H_{\lambda}) }(z).
\]

Continuing with the example we had chosen previously, assume $A=[ 2\ 3 ]$ and $b=5$. 
In Figure~\ref{fig:cayleyA}, we construct the vectors $(1,0,2)$ and $(0,1,3)$, 
lifting the two standard basis vectors $(1,0)$ and $(0,1)$ according to the entries of matrix $A$. 
In Figure~\ref{fig:cayleyB}, we shift the cone generated by $(1,0,2)$ and $(0,1,3)$ by $-b=-5$, and consider the halfspace defined by the $xy$-plane, where $\lambda$ is non-negative. The two intersection points of the cone with the hyperplane are marked. Now, in Figure~\ref{fig:cayleyC}, we consider the intersection of the halfspace defined above with the shifted cone. The lattice points in the intersection are shown in blue marked, and their projections onto the $\lambda=0$ plane are shown in red. These projections are the non-negative solutions to the original inequality, as shown in Figure~\ref{fig:cayleyD}.

\subsubsection*{MacMahon's rules}
  
MacMahon's method was general and in principle algorithmic, based on Elliott's work (see next section).
As we shall see, Elliott's method is simple but computationally very inefficient.
MacMahon introduced a set of rules in order to deal with particular combinatorial problems computationally -- in particular, as he had to carry out his computations by hand. 
We present some of MacMahon's rules from a geometric point of view.

The first two rules MacMahon gave are
\[
 \omeg \frac{1}{\left(1-\lambda x\right)\left(1-\frac{y}{\lambda^s}\right)}
= \frac{1}{\left(1-x\right)\left(1-x^s y\right)}
\qquad \text{for }s\in\ZZnn^*
\]
and 
\[
 \omeg\frac{1}{\left(1-\lambda^s x\right)\left(1- \frac{y}{\lambda}\right)}=
\frac{1+xy\frac{1-y^{s-1}}{1-y}}{\left(1-x\right)\left(1-xy^s\right)}
\qquad \text{for }s\in\ZZnn^*.
\]
There is an apparent symmetry in the input, but elimination results in structurally different numerators, i.e.,
$1$ versus $1+xy\frac{1-y^{s-1}}{1-y}$.
This can be easily explained by recalling that the numerator of the generating function enumerates the lattice points in the fundamental parallelepiped.
In Figures~\ref{fig:MMR1} and \ref{fig:MMR2} the respective fundamental parallelepipeds are shown.

\newcommand{\mmrsize}{0.4}
\begin{figure}[t]
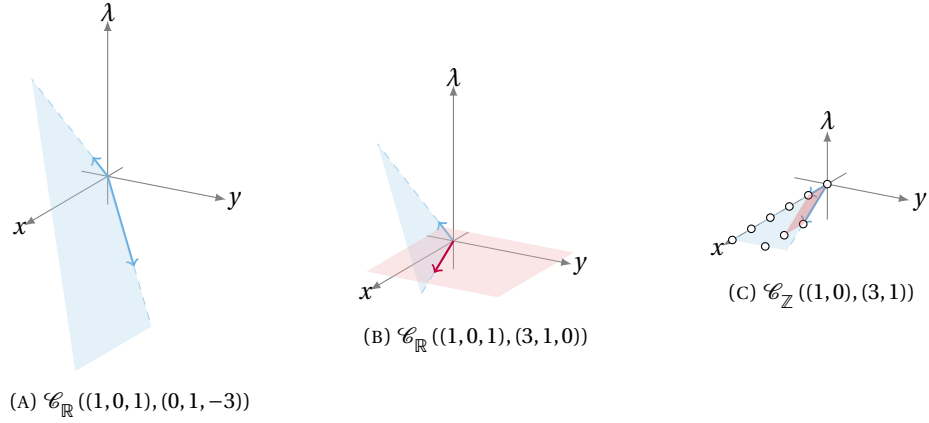

  \begin{subfigure}{0.3\textwidth}
    \centering
    \input{Graphics/mmr1_a}
    \caption{$\realcone{(1,0,1),(0,1,-3)}$}
  \end{subfigure}  
  \begin{subfigure}{0.3\textwidth}
    \centering
    \input{Graphics/mmr1_b}
    \caption{$\realcone{(1,0,1),(3,1,0)}$}
  \end{subfigure}  
  \begin{subfigure}{0.3\textwidth}
    \centering
    \input{Graphics/mmr1_c}
    \caption{ $\discretecone{(1,0),(3,1)}$}
  \end{subfigure}  
  \caption{The geometry of the first MacMahon rule.}
  \label{fig:MMR1}	
\end{figure}

 \begin{figure}[t]
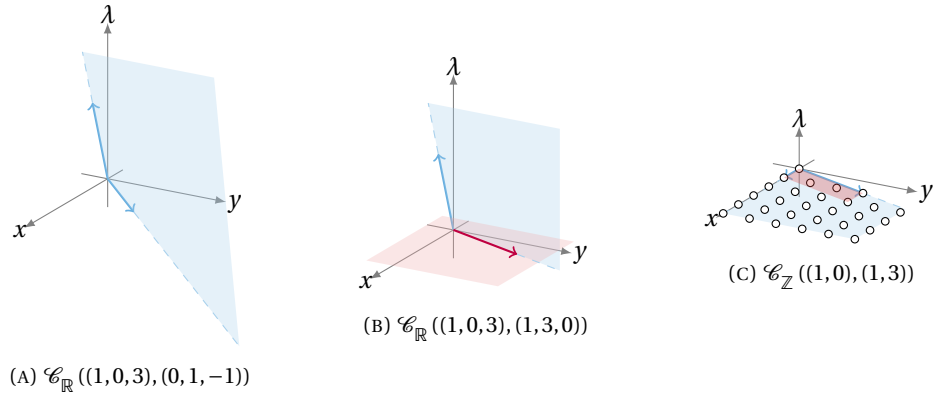

  \begin{subfigure}{0.3\textwidth}
    \centering
    \input{Graphics/mmr2_a}
    \caption{ $\realcone{(1,0,3),(0,1,-1)}$}
  \end{subfigure}  
  \begin{subfigure}{0.3\textwidth}
    \centering
    \input{Graphics/mmr2_b}
    \caption{$\realcone{(1,0,3),(1,3,0)}$}
  \end{subfigure}  
  \begin{subfigure}{0.3\textwidth}
    \centering
    \input{Graphics/mmr2_c}
    \caption{ $\discretecone{(1,0),(1,3)}$}
  \end{subfigure}  
  \caption{The geometry of the second MacMahon rule.}
  \label{fig:MMR2}	
\end{figure}

The fourth rule in MacMahon's list is 
\[
 \omeg \frac{1}{\left(1-\lambda x\right)\left(1-\lambda y\right)\left(1-\frac{z}{\lambda}\right)}
= \frac{1-xyz}{\left(1-x\right)\left(1- y\right)\left(1-x z\right)\left(1-y z\right)}.
\]
\noindent
Although the rational function on which $\omeg$ acts has three factors in the denominator, the resulting rational
generating function has four factors in the denominator
and the numerator has terms with both positive and negative signs.
This is because the initial cone defined by the inequality $x+y-z\geqslant 0$ is not simplicial.
The cone admits a signed decomposition into simplicial  cones.
When summing up the generating functions of the cones we obtain the right hand side of the rule, with a negative term due to the signed decomposition.
The decomposition is shown in Figure~\ref{fig:fourthmm}.

 \begin{figure*}[t]
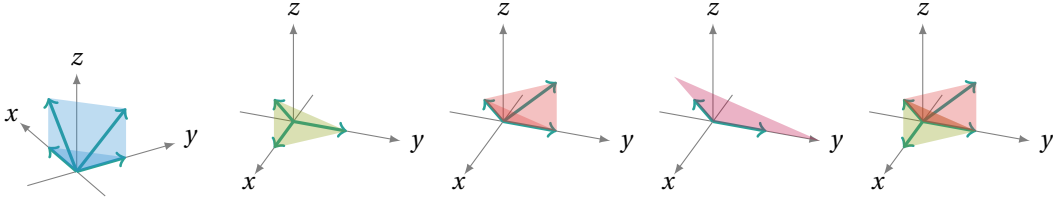

	\renewcommand{\mmrsize}{0.75}
	\input{Graphics/mmr4a1.tex}
	\input{Graphics/mmr4a2.tex}
	\input{Graphics/mmr4a3.tex}
	\input{Graphics/mmr4a4.tex}
	\input{Graphics/mmr4a5.tex}
	\caption{The fourth MacMahon rule. 
	We start with the non-simplicial cone $\realcone{(1,0,0),(0,1,0),(1,0,1),(0,1,1)}$.
	Then we consider the simplicial cones $\realcone{(1,0,0),(0,1,0),(1,0,1)}$ and 
	$\realcone{(0,1,0),(1,0,1),(0,1,1)}$ as well as their intersection 
	$\realcone{(0,1,0),(1,0,1)}$.
	By inclusion-exclusion we have a decomposition of the initial cone. 	
	}
	\label{fig:fourthmm}
 \end{figure*}

\subsection{Elliott}
\label{sec:elliott}

One of the first references relevant to partition analysis is Elliott's article ``On linear homogeneous Diophantine equations'' \cite{Elliott}.
The work of Elliott is exciting, if not for anything else, because it addresses mathematicians that lived a century apart. 
It was of interest to MacMahon, who based his method on Elliott's decomposition, but also to 21st century mathematicians for explicitly giving an important algorithm.
Although the algorithm has very bad computational complexity, it can be considered as an early algorithm for lattice point enumeration (among other things).

The problem Elliott considers is to find all non-negative integer solutions of one homogeneous linear Diophantine equation
\begin{equation}\label{EA:eq}
\sum_{i=1}^{m} a_i x_i - \sum_{i=m+1}^{m+n} b_i x_i =0 \qquad \text{for } a_i,b_i\in\ZZnn .
\end{equation} 

The starting point for Elliott's work is the fact that even if one computes the set of ``simple solutions'', i.e., solutions that are not combinations of others, there are syzygies preventing us from writing down formulas giving each and every solution to the equation exactly once.
He proceeds by explaining that his method computes a generating function whose terms are in one-to-one correspondence with the solutions of the equation.

The basic idea employed by Elliott (and the basis of partition analysis) is the introduction of an extra variable denoted by  $\lambda$, encoding the equation. This idea, as we saw, was employed by MacMahon and can be traced back to Cayley. Elliott then observes that the terms in the series expansion of 
\[
	\frac{1}{\left(1- x_1 \lambda^{a_1} \right)\cdots \left(1- x_m \lambda^{a_m} \right) \left(1- x_{m+1} \lambda^{-b_{m+1}} \right) \cdots \left(1- x_1 \lambda^{-b_{m+n}} \right)} 
\]
correspond to solutions of (\ref{EA:eq}).
The question then is how to decompose this rational function into a sum of rational functions, such that the expansion of each of them contains either only terms with non-negative $\lambda$ exponents or only terms with non-positive $\lambda$ exponents.
Given such a decomposition it is safe to drop the components that give terms containing $\lambda$.
Elliott's algorithm computes such a partial fraction decomposition.
Given a rational function expression $\prod_i^{k} \frac{1}{1-m_i}$, where $m_i$ is a monomial $z_1,z_2,\ldots,z_k,\lambda,\lambda^{-1}$ and $k\in\ZZnn$, Elliott computes a sum of the form $\sum \frac{\pm 1}{\prod (1-p_i)} + \sum \frac{\pm 1}{\prod (1-q_j)}$
with $p_i$ being a monomial in $z_1,z_2,\ldots,z_k,\lambda$ and $q_j$ a monomial in $z_1,z_2,\ldots,z_k,\lambda^{-1}$,
based on the identity
\begin{eqnarray}
\label{eq:elliott-decomposition}
  \frac{1}{(1-z_1\lambda^\alpha)(1-\frac{z_2}{\lambda^\beta})}  &=& 
 \frac{1}{(1-z_1z_2\lambda^{\alpha-\beta})(1-z_1\lambda^\alpha)} 
 + 
\frac{1}{(1-z_1z_2\lambda^{\alpha-\beta})(1-\frac{z_2}{\lambda^\beta})} -      
  \frac{1}{(1-z_1z_2\lambda^{\alpha-\beta})}.
\end{eqnarray}
To achieve the goal of only having summands where the exponents of $\lambda$ in the denominator all have the same sign, the above identity has to be applied iteratively. The left-hand side of (\ref{eq:elliott-decomposition}), with exponents $(\alpha,-\beta)$, is reduced to $(\alpha-\beta,\alpha)$ or $(\alpha-\beta,-\beta)$, depending on whether $\alpha \geq \beta$. This leads to a recurrence corresponding exactly to the Euclidean algorithm, see also \cite{BreuerVonHeymann}.

\newcommand{\elliottsize}{0.3}
\tikzstyle{point} =[draw,circle,inner sep=1pt,fill=white]
\tikzstyle{focuspoint} =[draw,circle,inner sep=1pt,fill=blue]
\tikzstyle{outoffocuspoint} =[draw,circle,inner sep=1pt,fill=gray!20]
\tikzstyle{ray} =[thick,color=blue!70, ->]
\tikzstyle{edge} =[color=blue!40, dashed]
\tikzstyle{vertex} =[draw,circle,inner sep=1pt,fill=white]   
\tdplotsetmaincoords{70}{150}
\begin{figure}[t]
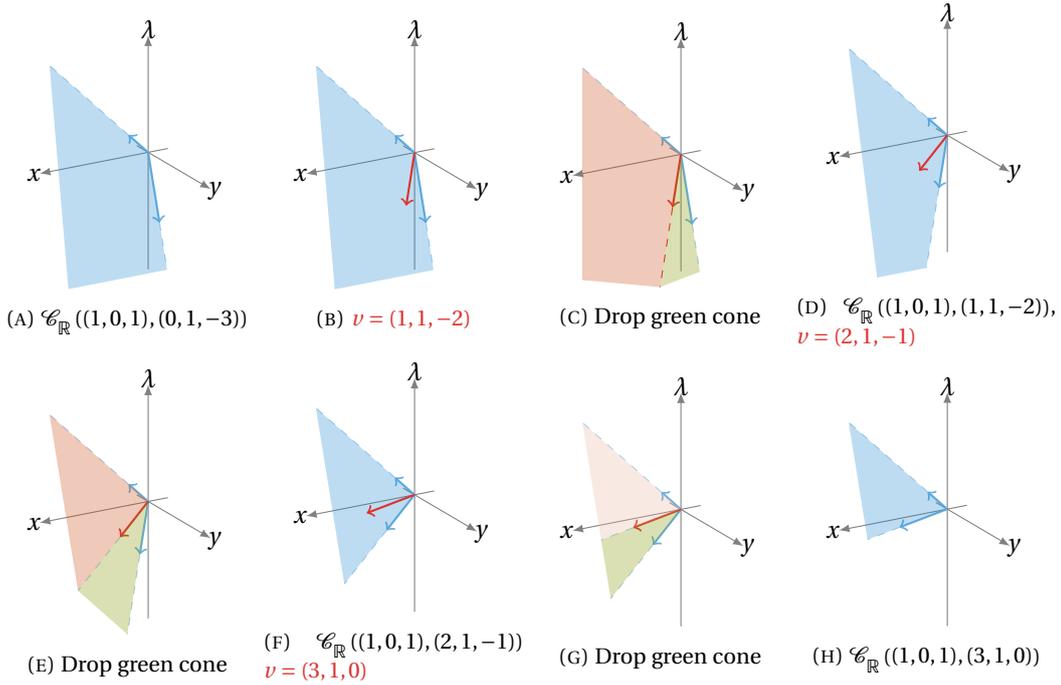

  \begin{subfigure}{0.23\textwidth}
    \centering
    \input{Graphics/elliott2.tex}
    \caption{$\realcone{(1,0,1),(0,1,-3)}$}
  \end{subfigure}  
  \begin{subfigure}{0.23\textwidth}
    \centering
    \input{Graphics/elliott3.tex}
    \caption{{\color{red}$v=(1,1,-2)$}}
  \end{subfigure}    
  \begin{subfigure}{0.23\textwidth}
    \centering
    \input{Graphics/elliott4.tex}
    \caption{Drop green cone}
  \end{subfigure}  
  \begin{subfigure}{0.23\textwidth}
    \centering
    \input{Graphics/elliott6.tex}
    \caption{$\realcone{(1,0,1),(1,1,-2)}$, {\color{red}$v=(2,1,-1)$}}
  \end{subfigure}  
  \begin{subfigure}{0.23\textwidth}
    \centering
    \input{Graphics/elliott7.tex}
    \caption{Drop green cone}
  \end{subfigure}    
  \begin{subfigure}{0.23\textwidth}
    \centering
    \input{Graphics/elliott9.tex}
    \caption{$\realcone{(1,0,1),(2,1,-1)}$ {\color{red}$v=(3,1,0)$}}
  \end{subfigure}  
  \begin{subfigure}{0.23\textwidth}
    \centering
    \input{Graphics/elliott10.tex}
    \caption{Drop green cone}
  \end{subfigure}  
  \begin{subfigure}{0.23\textwidth}
    \centering
    \input{Graphics/elliott11.tex}
    \caption{$\realcone{(1,0,1),(3,1,0)}$}
  \end{subfigure}    
  \caption{Elliott's decomposition method.}
  \label{fig:elliott}	
\end{figure}


Elliott's identity has a very nice geometric interpretation, which is illustrated in Figure~\ref{fig:elliott}. We observe that
$\displaystyle
\frac{1}{(1-x\lambda^{\alpha})(1-\frac{y}{\lambda^{\beta}})} 
$ 
is the generating function of the $2$-dimensional cone $C=\realcone{(1,0,\alpha),(1,0,-\beta)}$, i.e., the fundamental parallelepiped of $C$ contains only the origin.
Since the point $v=(1,1,\alpha-\beta)$ is in the interior of the cone $C$, the cones 
$
A=\realcone{(1,0,\alpha),(1,1,\alpha-\beta)}
$
and
$
B=\realcone{(1,0,-\beta),(1,1,\alpha-\beta)}
$
subdivide cone $C$. 
Their intersection is exactly the ray starting from the origin and passing through  $v=(1,1,\alpha-\beta)$.
By a simple inclusion-exclusion argument we have the signed decomposition $C=A+B-(A\cap B)$.
This decomposition translated back to the generating functions level is exactly the partial fraction 
decomposition employed by Elliott.
One should mention here that Elliott insists on the fact that the numerators are $\pm 1$, 
which geometrically means that he deals with unimodular cones, which also play a key role in Barvinok's algorithm.

\subsection{Andrews-Paule-Riese}

In 1997, Bousquet--M{\'e}lou and Eriksson presented a theorem on 
lecture hall partitions in \cite{BME}. 
This theorem gathered a lot of attention from the community (and it
still does, with many lecture hall type theorems appearing still today \cite{BBKSZ1}).
Andrews, who had already studied MacMahon's method and was aware of its computational potential, figured that lecture hall partitions offered the right problems to attack algorithmically via partition analysis.

At the same time it was planned to spend a semester during his sabbatical at RISC in Austria to work with Paule. 
It is only natural that the result was a fully algorithmic version
of MacMahon's method powered by symbolic computation\index{Symbolic Computation}.
This collaboration gave a series of 10 papers 
\cite{PA3,PA4,PA5,PA6,PA7,PA8,PA9,PA10,PA11,PA12}.
Many interesting theorems and different kinds of partitions are 
defined and explored in this series of papers, but for us the most important 
two are \cite{PA3} and \cite{PA6}, which contain the algorithmic
improvements on partition analysis.
Namely, in \cite{PA3} the authors introduce \omegaone , a \mathematica package based on a fully algorithmic version of partition analysis, 
while in \cite{PA6} they present a more advanced partial fraction
decomposition (and the related \mathematica package \omegatwo) solving some of the problems appearing in \omegaone. 
While presenting the methods, we will see some of their geometric aspects.
 
The main tool in \omegaone is the Fundamental Recurrence \index{Fundamental Recurrence}(Theorem 2.1 in \cite{PA3}) for the $\omeg$ operator.
Following MacMahon, or Elliott for that matter, iterative application of this recurrence is enough for computing the action of $\omeg$.
The interested reader can find the definition of the recurrence in \cite{PA3}. The fundamental recurrence assumes that the exponents of $\lambda$ in the denominator are $\pm 1$. 
This is not a strong assumption, as noted in \cite{PA3}, since we can always employ a decomposition to linear factors.
The obvious drawback of this approach is that we introduce complex coefficients instead of $\pm 1$.
This motivates the desire for a better recurrence. In \cite{PA6} the same authors introduce an improved partial fraction decomposition method (the generalized partial fraction decomposition), 
given by the following recurrence.\footnote{Here, we have slightly rewritten the recurrence from \cite{PA6} to facilitate the geometric interpretation.} Let $\alpha\geqslant\beta\geqslant 1$, $\gcd (\alpha,\beta)=1$ and $\inv_{b}(a)$ denote the
 multiplicative inverse of $a$ modulo $b$. Then
\begin{equation}
\label{eq:gpfd}
 \frac{1}{(1-z_1z_3^{\alpha})(1-z_2z_3^{\beta})}=
\frac{ P_{\alpha,\beta}}{(1-z_1^{-\beta}z_2^{\alpha})(1-z_1z_3^{\alpha})} 
- \frac{ Q_{\alpha,\beta}}{(1-z_1^{-\beta}z_2^{\alpha})(1-z_2z_3^{\beta})}
\end{equation}
where
$P_{\alpha,\beta}:=\displaystyle\sum_{i=0}^{\alpha-1}a_i z_3^i$ and 
$Q_{\alpha,\beta}:=\displaystyle\sum_{i=0}^{\beta-1}b_i z_3^i$
for
\begin{eqnarray}
\label{eqn:omega2-numerator-a}
a_i&=& 
    \begin{cases}
    z_2^{\frac{i}{\beta}} & \text{ if\ } \beta | i \text{ or } i=0 ,\\
    z_1^{( \inv_{\beta}(\alpha) i \bmod \beta)-\beta} z_2^{( \inv_{\alpha}(\beta)i \bmod \alpha)} & \text{ otherwise} ,\\
    \end{cases}\\
\label{eqn:omega2-numerator-b}
b_i&=& 
    \begin{cases}
     z_1^{-\beta}z_2^{\alpha} & \text{ if } i=0 ,\\
     z_1^{ ( \inv_{\beta}(\alpha) i \bmod \beta)-\beta} z_2^{( \inv_{\alpha}(\beta) i \bmod \alpha)} & \text{ otherwise} .\\
    \end{cases}
\end{eqnarray}

Before interpreting the rational functions involved in the Fundamental Recurrence of \cite{PA6} as cones in Figure~\ref{fig:omega}, we introduce some notation and the concept of a half-open cone, which is useful when dealing with cone decompositions. The \emph{dimension} of a polyhedron is the dimension of its affine hull. A \emph{face} of a polyhedron $P$ is a subset of $P$ on which some linear functional is maximized. Faces of polyhedra are polyhedra themselves. Faces capture the intuitive notion of a ``side'' of a polyhedron. 0-dimensional faces are called \emph{vertices}, 1-dimensional faces are called \emph{edges} and the $(d-1)$-dimensional faces of a polyhedron of dimension $d$ are called \emph{facets}. For example, a 3-dimensional cube has 8 vertices, 12 edges and 6 facets. Recall that
\[
  C=\realcone{v_1,v_2,\ldots,v_k}[q]=q+\mset{\sum_{i=1}^k \lambda_i v_i}{0\leqslant\lambda \in\RR}.
\]
Assume $C$ is simplicial and let $[k]=I\subset \ZZnn$ be the index set for its generators. 
Then each face $F$ of $C$ can be identified by a set $I_F\subseteq I$, since 
a face has the form 
\[
  F = q+\mset{\sum_{i=1}^k \lambda_i v_i}{ 0\leqslant \lambda_i \in\RR, \lambda_j=0 \text{ for } j\in I_F }.
\]
Essentially, this means that the face is generated by a subset of the generators.
Note that $I_F$ contains the indices of the generators of the simplicial cone $C$ that are not used to generate $F$ as in Figure~\ref{fig:facesandsets}. 
A facet is identified with the index of the single generator not used to generate it.

\begin{figure*}[t]
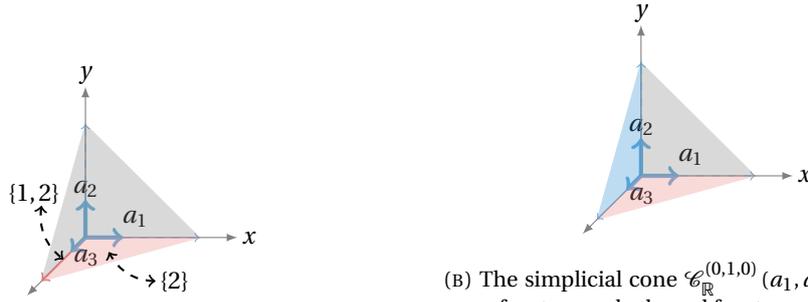

  \begin{subfigure}{0.45\textwidth}
    \centering
    \input{Graphics/facesandsets.tex}
    \caption{Assignment of index subsets $I_F$ to faces $F$. The ray generated by $a_3$ is indexed by $I=\{1,2\}$. The face generated by $a_1,a_2$ is indexed by $I=\{2\}$.}
    \label{fig:facesandsets}
  \end{subfigure}  
  \qquad
  \begin{subfigure}{0.45\textwidth}
    \centering
    \input{Graphics/halfopencones.tex}
    \caption{The simplicial cone $\realcone{a_1,a_2,a_3}[][(0,1,0)]$ has one open facet, namely the red facet on which the coefficient of $a_2$ is zero. The vector $(0,1,0)$ denotes that the facet with index set $I=\{2\}$ is open. Similarly, the cone $\realcone{a_1,a_2,a_3}[][(1,1,0)]$ has both the blue and red facets open, which correspond to the index sets $\{1\}$ and $\{2\}$, respectively.}
    \label{fig:openess}
  \end{subfigure}  
  \caption{Half-open cones.}
\end{figure*}

A \emph{half-open cone} $C$ is a cone from which some facets have been removed.
From the previous discussion, for each facet we need only to mention the generator corresponding to that facet.
We will use a $0-1$ vector $o$ to express the openness of a cone, i.e.,
the openness vector has an entry of $1$ in the position $k$ if the facet corresponding to the $k$-th generator is open. To this end, we define
\[
  \realcone{v_1,\ldots,v_d}[q][o] = q + \mset{\sum_{i=1}^d \lambda_i v_i}{0 \leq \lambda_i \in\RR \text{ if $o_i$=0 and } 0 < \lambda_i \in\RR \text{ if $o_i$=1}}.
\]
The above notation will be used throughout the paper. To improve readability, we will drop the $o$ from the notation when $o=0$, i.e., when all facets are closed. Similarly, we drop $q$ when $q$ is the zero vector. Moreover, if $V$ is a matrix with vectors $v_1,\ldots,v_d$ as columns we may write $V$ instead of listing the vectors explicitly, such that, for example, $\realcone{V} := \realcone{v_1,\ldots,v_d}[0][(0,\ldots,0)]$.

Note that both (\ref{eqn:Ehrhart}) and (\ref{eqn:Ehrhart-tiling}) extend naturally to the case of affine half-open cones. Let $V$ be the matrix with $v_1,\ldots,v_d \in \ZZ^n$ as columns. If we generalize the notion of a fundamental parallelepiped to
\[
  \fundparreal[o](V) = \mset{\sum_{i=1}^d \lambda_i v_i}{0 \leq \lambda_i < 1 \text{ if $o_i$=0 and } 0 < \lambda_i \leqslant 1 \text{ if $o_i$=1}}
\]
and let $\fundparlattice[\ZZ^n][o](V) := \fundparreal[o](V) \cap \ZZ^n$, then we have both
\begin{eqnarray}
\label{eqn:Ehrhart-general}
  \gfl{\realcone{V}[q][o]}(z) = \frac{\gfl{\fundparlattice[\ZZ^n][o](V)}(z)}{\prod_{i=1}^d (1-z^{v_i})} &\;\;\text{ and } &
\label{eqn:Ehrhart-tiling-general}
  \realcone{V}[q][o] \cap \ZZ^n = \discretecone{V}[q] + \fundparlattice[\ZZ^n][o](V).
\end{eqnarray}
  
\newcommand{\omegasize}{0.5}
  \begin{figure}[t]
  \begin{subfigure}{0.45\textwidth}
    \centering
    \input{Graphics/Omega2_a}
    \caption{Cone $C$}
  \end{subfigure}   
  \begin{subfigure}{0.45\textwidth}
    \centering
    \input{Graphics/Omega2_b}
    \caption{The new generator}
  \end{subfigure}   

\renewcommand{\omegasize}{0.4}  
  \begin{subfigure}{0.3\textwidth}
    \centering
    \input{Graphics/Omega2_c}
    \caption{Cone $A$}
  \end{subfigure}   
  \begin{subfigure}{0.3\textwidth}
    \centering
    \input{Graphics/Omega2_d}
    \caption{Cone $B$}
  \end{subfigure}   
  \begin{subfigure}{0.3\textwidth}
    \centering
    \input{Graphics/Omega2_e}
    \caption{The decomposition}
  \end{subfigure}

  \caption{\omegatwo decomposition.}
    \label{fig:omega}
 \end{figure}

Returning to the cone interpretation of \omegatwo (Figure~\ref{fig:omega}), we observe that the 2-dimensional cone  $C=\realcone{(1,0,\alpha),(0,1,\beta)}$ is the left hand side of the generalized partial fraction decomposition (\ref{eq:gpfd}).
Then we intersect the $xy$-plane with the plane the cone lives in. 
Choose a ray on this intersection (conveniently, one such ray is $(-\beta, \alpha,0)$, denoted by red in the figure).
Then consider two cones $A=\realcone{(-\beta,\alpha,0),(1,0,\alpha)}$  and $B=\realcone{(-\beta,\alpha,0),(0,1,\beta)}[][(1,0)]$. We observe that the cone $C$ is decomposed as the cone $A$ minus the half-open cone $B$, as can be seen in Figure~\ref{fig:omega}.
In particular, the generating function for the fundamental parallelepiped of cone $A$ is exactly the numerator $P_{\alpha,\beta}$
and similarly the fundamental parallelepiped of $B$ is precisely the numerator $Q_{\alpha,\beta}$. See Figure~\ref{fig:omegafp} for an example. This connection between polynomials of this form and fundamental parallelepipeds is well-known in the context of Dedekind--Carlitz polynomials \cite{BeckHaaseMatthews}. In particular, these polynomials have a short recursive description \cite{BreuerVonHeymann}, which can be exponentially more compact than (\ref{eqn:omega2-numerator-a}) and (\ref{eqn:omega2-numerator-b}).

This shows that the Andrews-Paule-Riese decomposition adds a layer of complexity to the geometry of partition analysis 
in comparison with previous methods.
Elliott is dealing only with unimodular cones, but has to resort to a recursive procedure for the elimination of a $\lambda$ variable. 
In contrast, the generalized partial fraction decomposition of \cite{PA6}, eliminates a $\lambda$ variable in one step at the expense of having cones with fundamental parallelepipeds containing potentially exponentially many lattice points. Moreover, the Andrews-Paule-Riese decomposition is more complex than previous rules in that it works with half-open cones.

\begin{figure}[t]
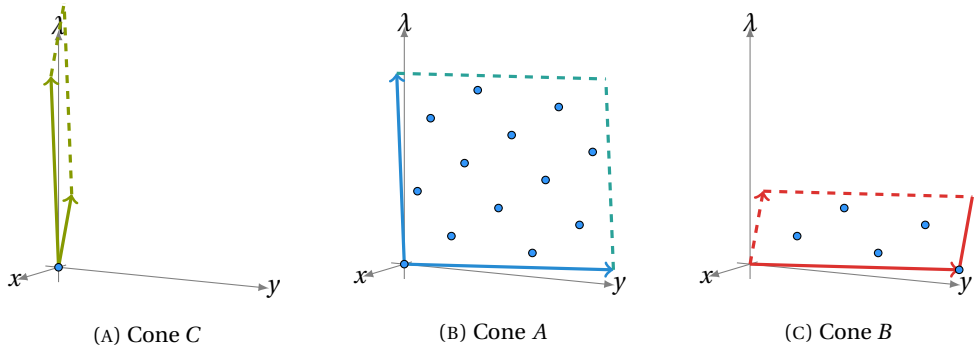

  \begin{subfigure}{0.3\textwidth}
    \centering
    \input{Graphics/omegafundpar_C}
    \caption{Cone $C$}
  \end{subfigure}   
  \begin{subfigure}{0.3\textwidth}
    \centering
    \input{Graphics/omegafundpar_A}
    \caption{Cone $A$}
  \end{subfigure}   
  \begin{subfigure}{0.3\textwidth}
    \centering
    \input{Graphics/omegafundpar_B}
    \caption{Cone $B$}
  \end{subfigure}   
  \caption{Rule (\ref{eq:gpfd}), for $\alpha=13$ and $\beta=5$, decomposes $C=\CCC(g_1,g_2)$ into the difference of $A=\CCC(g,g_1)$ and $B=\CCC^{(1,0)}(g,g_2)$, i.e., $[C] = [A] - [B]$, where $g_1 = (1,0,13), g_2=(0,1,5), g=(-5,13,0)$. While the fundamental parallelepiped of $C$ contains just a single lattice point, the lattice points in the fundamental parallelepipeds of $A,B$ correspond to the numerators $P_{\alpha,\beta}$ and $Q_{\alpha,\beta}$ respectively. The number of lattice points is linear in $\alpha,\beta$, i.e., it is exponential in the encoding length of $\alpha,\beta$.
    }
  \label{fig:omegafp}
\end{figure}

\subsection{Questions of Convergence}
\label{sec:xin-convergence}

In the partition analysis series of papers by Andrews-Paule-Riese and in particular in \cite{PA3},
the authors stress the importance of working with absolutely convergent series.
This is because the Omega operator is not well defined if rational function expansions
do not have a common region of convergence. Quoting from \cite{PA3}:
\begin{quotation}
``While MacMahon did not carefully distinguish whether his Laurent series were
analytic or merely formal, we emphasize that it is essential to treat everything
analytically rather than formally because the method relies on unique Laurent
series representations of rational functions. For instance, if we were to proceed
formally, then
\[
\omeg \sum_{n=0}^{\infty} q^n\lambda^n=
\sum_{n=0}^{\infty} q^n=
\frac{1}{1-q}
\;\;\; \text{ while } \;\;\;
-\omeg \sum_{n=1}^{\infty} q^{-n}\lambda^{-n}=0.
\]
But if we allowed a purely formal application of the geometric series, then both
initial expressions are $\frac{1}{1-\lambda q}$.''
\end{quotation}
While this observation about the Omega operator is indeed crucial, the conclusion that it is therefore ``essential to treat everything analytically rather than formally'' is not valid, as there are several methods to treat these issues from a formal point of view.

Xin, who also gives an improved set of rules for computing the Omega operators in \cite{XinThesis}, notes that by working in the field of iterated Laurent series and by normalizing the rational function expression accordingly, it is possible to disregard questions of convergence. An in-depth study of different ways to define multivariate Laurent series appears in \cite{MonforteKauers}. There, Aparicio Monforte and Kauers introduce multivariate Laurent series as linear combinations of series that are supported in possibly shifted pointed cones compatible with a given ordering. They prove that this construction ensures that the support is well ordered and embed their theory  in the Malcev--Neumann series construction presented by Xin in \cite{XinThesis}. Their construction is a natural and ``large'' field of multivariate Laurent series.

Independently from these developments, researchers in the polyhedral geometry community have been dealing with these analytical phenomena from a geometric perspective for a long time. Recall that, intuitively, we think of $\frac{1}{1-q}$ as corresponding to the 1-dimensional cone $[0,+\infty)$ and of $\frac{q^{-1}}{1-q^{-1}}$ as corresponding to the half-open 1-dimensional cone $(-\infty,0)$. The identity $\frac{1}{1-q}=\frac{-q^{-1}}{1-q^{-1}}$ on the level of rational functions then leads us to considering equivalence classes of cones ``modulo lines'', see Section~\ref{sec:polyhedral-omega-motivation}. This is a routine practice in the polyhedral geometry literature \cite{BarvinokIntegerPoints} which is used for great effect (for example in Brion's famous theorem, see Theorem~\ref{thm:brion} below) and should therefore be regarded as ``a feature not a bug''.

In this article, we have to be careful, however, since, as the above example shows, the Omega operator is not well-defined on equivalence classes modulo lines. Nonetheless, as we will see in the next section, polyhedral methods can still be applied to evaluate the Omega operator in an entirely symbolic manner, as long as we work with the right representatives of these equivalence classes modulo lines. To choose representatives consistently, we define any cone $C$ as \emph{forward} if all generators of $C$ are lexicographically larger than the origin, i.e., if their first non-zero entry is positive. The notion of forward cones appears in \cite{Lawrence} and plays a prominent role in the Lawrence-Varchenko decomposition (Theorem~\ref{thm:lawrence-varchenko}) that we use below. For the Omega operator to be well defined, we require that all series appearing in the expression on which Omega acts are supported in the half-open half-space of vectors that are lexicographically larger or equal than the origin. In general, one can define forward with respect to an arbitrary half-open half-space, but the above definition will suffice for our purposes. By working with expansions of rational functions that are forward in this geometric sense, we can apply the Omega operator without recourse to a full analytic treatment. In the example given above, the series $-\sum_{n=1}^{\infty} q^{-n}\lambda^{-n}$ is \emph{not} a forward expansion of $\frac{1}{1-\lambda q}$. Note that Xin's iterated Laurent series field construction corresponds precisely to this geometric idea of working with linear combinations of cones that are all forward in the lexicographic sense, again see \cite{MonforteKauers}. In Section~\ref{sec:polyhedral-omega-motivation}, we will discuss both the method of working with cones modulo lines as well as the notion of forward cones in detail.

\subsection{Applying Omega vs.~Solving LDS}
\label{sec:omega-vs-lds}

MacMahon invented the $\omeg$ operator for solving linear Diophantine systems. In this framework, $\omeg$ is applied to rational functions that arise from the MacMahon lifting. However, once $\omeg$ is defined, the question arises how to evaluate $\omeg$ when applied to a general rational function $r$ whose denominator factors into binomials. It turns out that this seemingly more general problem can be reduced to solving linear Diophantine systems, and, in particular, to applying $\omeg$ to rational functions as produced by the MacMahon lifting.

Let $r\in\KK(x_1,\ldots,x_d)$ be a rational function whose denominator factors into binomials. Let $\omeg$ eliminate the last $k$ of the variables $x_i$. Due to linearity of $\omeg$, we can, without loss of generality, restrict our attention to rational functions of the form
\[
  r(x) = \frac{x^b}{(1-\alpha_1 x^{a_1})\cdot\ldots\cdot (1-\alpha_n x^{a_n})}
\]
where $b,a_1,\ldots,a_n\in\ZZ^d\setminus\{0\}$, $\alpha_1,\ldots,\alpha_n\in\KK$ and the vectors $b,a_1,\ldots,a_n$ are forward with respect to the given ordering of the variables $x_i$. Now, replace $r$ with another rational function $s\in\QQ(z_1,\ldots,z_n,x_1,\ldots,x_d)$ defined as
\[
  s(z,x) = \frac{x^b}{(1-z_1 x^{a_1})\cdot\ldots\cdot (1-z_n x^{a_n})}.
\]
This is the rational function constructed by the MacMahon lifting from the system $Ax\geq -b$, where $A$ is the matrix with the $a_1$ as columns. It corresponds to the simplicial cone with generator matrix $\msmat{\Id_n \\A}$ and apex $\msmat{0\\-b}$. Therefore the methods discussed in this paper can be applied to compute a rational function expression for $\omeg(s)$. 

To obtain $\omeg(r)$ from $\omeg(s)$, all that is necessary is to substitute $z_i=\alpha_i$ for all $i$ in $\omeg(s)$. This substitution is possible by construction, as long as the rational function expression obtained for $\omeg(s)$ is brought into normal form first. However, as we will see below, the rational function expressions we usually obtain are sums of rational functions and $z=\alpha$ may well be a pole of individual summands in this expression. However, there are methods to do this substitution efficiently, and obtain $\omeg(r)$ from $\omeg(s)$, without bringing $\omeg(s)$ into normal form first. 

There are a number of references that discuss such methods in depth, so we will not concern ourselves further with these issues in this paper and refer the interested reader to the relevant literature: Substitutions into sums of rational functions of the above form are an essential part of Barvinok's algorithm for counting lattice points in polytopes \cite{Barvinok1993,Barvinok1994,BarvinokIntegerPoints,DeLoera2012}. In \cite{BarvinokWoods}, Barvinok and Woods extend these techniques and develop a set of algorithms for performing a number of operations on rational function expression in the above form. In particular, they give an algorithm that computes the Hadamard product of two rational functions that runs in polynomial time if the number of factors in the denominators is fixed \cite[Lemma~3.4]{BarvinokWoods}. This immediately gives an algorithm for computing $\omeg(r)$ for any $r$ of the form given above. This algorithm runs in polynomial time, if $d$ and $n$ are fixed. Computing Hadamard products with this method has been implemented, e.g., in the \texttt{barvinok} package \cite{Verdoolaege2007,Verdoolaege2011}. In \cite{Xin2012}, Xin describes another algorithm that reduces the computation of $\omeg(r)$ to solving linear Diophantine systems.

\subsection{Xin's Partial Fraction Decomposition}
\label{sec:xin-pfd}

\emph{This subsection makes use of material introduced in Section~\ref{sec:polyhedral-omega-motivation}, including symbolic cones, $\prim$, and the Lawrence--Varchenko decomposition. Readers unfamiliar with these notions should skip Section~\ref{sec:xin-pfd} on a first reading and return after they have finished Section~\ref{sec:polyhedral-omega-motivation}.}

In \cite{XinThesis,Xin2004,Xin2012}, Xin proposes several partition analysis algorithms based on iterated partial fraction decomposition (PFD) in the multivariate rational function field. While \cite{Xin2004} focuses on PFD of rational functions with two factors in the denominator, \cite{Xin2012} applies PFD directly to rational functions with $n$ factors in the denominator. 

As it turns out, partial fraction decomposition has a clear interpretation in the language of polyhedral geometry: Given a simplicial cone $C\subset\RR^d$, PFD writes $C$ as an inclusion-exclusions of simplicial cones $C_1,\ldots,C_d$, such that, for each $C_i$, all but one of the generators of $C_i$ lie in the $x_d=0$ coordinate hyperplane. This decomposition is closely related to a Brion/Lawrence-Varchenko decomposition of a section of $C$. 

For example, consider the simplicial cone $C\subset\RR^3$ with apex $0$ and generators $(1,0,1)$, $(0,1,1)$ and $(2,1,2)$. The corresponding rational function is 
\begin{eqnarray*}
  \rho_C = \frac{1}{(1 - xz) (1 - yz) (1 - x^2 y z^2)}.
\end{eqnarray*}
Computing the PFD, we find
\begin{eqnarray*}
  \rho_C &=& 
    \frac{1}{(1 - x y^{-1}) (1 - x^2 y^{-1}) (1 - y z)} 
    - \frac{x y^{-1}}{(1 - x y^{-1}) (1 - y) (1 - x z)} 
    + \frac{x + x^2 + x^3 z + x^2 y z}{(1 - x^2 y^{-1}) (1 - y) (1 - x^2 y z^2)}.
\end{eqnarray*}
On the level of symbolic cones, this corresponds precisely to the decomposition
\begin{eqnarray}
\label{eqn:xin-pdf-cones}
  & & [C] =  
    + \left[\CCC_\RR^{(0,0,0)} \left({\msmat{ 1 & 2 & 0 \\ -1 & -1 & 1 \\ 0 & 0 & 1}}\right)\right]
    - \left[\CCC_\RR^{(1,0,0)} \left({\msmat{ 1 & 0 & 1 \\ -1 & 1 & 0 \\ 0 & 0 & 1}}\right)\right] 
    + \left[\CCC_\RR^{(1,1,0)} \left({\msmat{ 2 & 0 & 2 \\ -1 & 1 & 1 \\ 0 & 0 & 2}}\right)\right].
\end{eqnarray}

Geometrically, this decomposition can be obtained as follows. The intersection of $C$ with the hyperplane $\mset{(x,y,z)}{z=2}$ is the triangle $\Delta$ with vertices $(2,0,2)$, $(0,2,2)$ and $(2,1,2)$. The height $z=2$ is lowest possible such that the triangle has all integer vertices. $\Delta$ has Lawrence--Varchenko decomposition
\begin{eqnarray*}
  [\Delta] = 
    \left[\CCC_\RR^{(0,0)} \left(\msmat{ 1 & 2 \\ -1 & -1 \\ 0 & 0 } ; \msmat{ 0 \\ 2 \\ 2}\right)\right]
    - \left[\CCC_\RR^{(1,0)} \left(\msmat{ 1 & 0 \\ -1 & 1 \\ 0 & 0 } ; \msmat{ 2 \\ 0 \\ 2}\right)\right]
    + \left[\CCC_\RR^{(1,1)} \left(\msmat{ 2 & 0 \\ -1 & 1 \\ 0 & 0 } ; \msmat{ 2 \\ 1 \\ 2}\right)\right].
\end{eqnarray*}
Each of these three 2-dimensional cones, we now translate to the origin and take the Minkowski sum with the ray through its former apex. This means simply replacing $\CCC_\RR^{(o_1,o_2)}((v_1,v_2);q)$ with $\CCC_\RR^{(o_1,o_2,0)}((v_1,\allowbreak v_2,\allowbreak\prim(q)))$. Thus, we get precisely the decomposition given in (\ref{eqn:xin-pdf-cones}). We have now obtained a (signed) decomposition $[C] = \epsilon_1[C_1] + \ldots + \epsilon_d[C_d]$ such that each $C_i$ has only one generator with $z=x_d\not= 0$. Such $C_i$ have the desirable property that the intersections $C_i \cap \mset{x}{x_d\geq 0}$ are much easier to describe. On the rational function level, Xin \cite{Xin2012} makes use of this fact by using the PFD of a rational function $\rho=\rho_1+\ldots+\rho_d$ as a starting point and then giving a recursive procedure in the spirit of Euclidean algorithm formulas for computing the $\omeg(\rho_i)$. We will come back to Xin's algorithm in Section~\ref{sec:related-work}. By contrast, the algorithm we describe in Section~\ref{sec:polyhedral-omega-motivation} makes direct use of the Law\-rence--Varchenko decomposition: We give explicit formulas on the level of symbolic cones that provide a Lawrence--Varchenko decomposition of the polyhedron $C\cap \mset{x}{x_d\geq 0}$ itself.


\section{\polyomega -- the New Algorithm}
\label{sec:polyhedral-omega-motivation}

In this section, we introduce our new algorithm, \polyomega, which is a synthesis of the partition analysis approach to solving linear Diophantine systems and methods from polyhedral geometry. 

The \polyomega algorithm solves the rdfLDS problem, as defined in the introduction. Given a matrix $A\in\ZZ^{m\times d}$ and a vector $b\in\ZZ^m$, it computes a rational function expression $\rho$ in $d$ variables $z_1,\ldots,z_d$ whose multivariate Laurent series expansion represents\footnote{The coefficient of $z^x:=z_1^{x_1}\cdot\ldots\cdot z_d^{x_d}$ in the Laurent expansion of $\rho$ is 1 if $x$ is a solution of $A x \geqslant b$, $x\geqslant 0$ and it is 0 otherwise. In the following, we will sometimes abuse terminology and refer to the coordinates of solution vectors and the corresponding variables of the generating function interchangeably.} the set of all non-negative integer solutions $x\in\ZZ^d_{\geqslant0}$ of $Ax\geqslant b$. As motivated in Section~\ref{sec:partition-analysis}, this rational function expression corresponds to a representation of the set $\mset{x\in\ZZ^d_{\geqslant 0}}{Ax \geqslant b}$ as a signed linear combination of simplicial cones. Such a representation in terms of what we call \emph{symbolic cones}, defined in Section~\ref{sec:elimination} below, is central to the algorithm and can serve as a useful output of the algorithm in its own right. 

Without loss of generality, we will deal only with systems $Ax\geqslant b$ of inequalities in this paper. However, it is straightforward to extend the algorithm presented here to handle mixed systems containing both inequalities and equations directly. In particular, we have added this feature to our implementation of Polyhedral Omega to which refer the interested reader for details \cite{PolyhedralOmegaCode}.

The Polyhedral Omega algorithm consists of the following steps:
\begin{enumerate}
\item \emph{MacMahon Lifting.} Given a matrix $A\in\ZZ^{m\times d}$ and a vector $b\in\ZZ^m$, we use the MacMahon lifting to compute a unimodular simplicial cone $C\subset\RR^{d+m}$ such that 
\[
   \mset{x\in\RR_{\geqslant0}^d}{Ax \geqslant b} = \underbrace{\Omega_\geqslant(\Omega_\geqslant(\cdots(\Omega_\geqslant}_{\text{$m$ times}} (C) ))) = \Omega_\geqslant^m(C) ,
\]
i.e., the desired set of solutions can be obtained by applying the Omega operator which eliminates the last coordinate $m$ times to the cone $C$.
\item \emph{Iterative Elimination of the Last Coordinate using Symbolic Cones.} This is the main step of the algorithm, in which we compute $\Omega_\geqslant^m(C)$ iteratively, by eliminating one variable at a time. Each application of $\omeg$ corresponds to intersecting a polyhedron $P\subset\RR^n$  with the half-space $H^+_n := \mset{x\in\RR^n}{x_n \geqslant 0}$ of points with non-negative last coordinate and then projecting-away that last coordinate. The crucial property of this procedure is that, throughout, polyhedra are represented as linear combinations of simplicial symbolic cones $C$ which are stored as triples $(V,q,o)$ of a generator matrix $V$, an apex $q$ and an openness-vector $o$. In the end, we obtain a representation 
\begin{eqnarray*}
  \Omega_\geqslant^m(C) = \sum \alpha_i C_i
\end{eqnarray*}
in terms of symbolic cones, much more compact than a representation in terms of rational functions.
\item \emph{Conversion to Rational Functions.} The representation of the solution in terms of symbolic cones is then converted to a rational function representation. For this, two different methods can be used that each have their own advantages and that can also be used in tandem.
\begin{enumerate}
\item One option is to explicitly enumerate the lattice points in the fundamental parallelepiped for each of the $C_i$ and then apply equation (\ref{eqn:Ehrhart}) to obtain a rational function.
\item Another option is to use Barvinok decompositions to represent each $C_i$ as a short signed sum of unimodular cones, which can be immediately converted to rational functions.
\end{enumerate}
\item \emph{Rational Functions in Normal Form.} Each of the alternatives in the previous step gives a rational function expression for the solution. However, this rational function expression will not in general be in normal form. Standard algebraic procedures may be used to bring the rational functions on a common denominator and to simplify the resulting expression. Note that conversion to normal form may increase the size of the output exponentially.
\end{enumerate}

Depending on the application, it may be a good idea to work with the intermediate representations computed over the course of the algorithm and skip the later steps. As already mentioned, the representation in terms of symbolic cones is in most cases significantly more concise than the rational function representations. Especially if the output of the algorithm is to be examined by hand, we recommend taking a look at the symbolic cone representation first. Among the rational function representations, Barvinok decompositions can be exponentially shorter than expressions obtained from enumerating fundamental parallelepipeds. On the other hand, the enumeration of the fundamental parallelepipeds can be described in terms of explicit formulas and expressions with fewer distinct factors in the denominators of the rational functions. Computing the normal form of a rational function can lead to useful cancellations, but is computationally expensive and can also increase the size of the output exponentially. 

We now go through each of the steps of the algorithm in detail.

\subsection{MacMahon Lifting} 
\label{sec:macmahon-lifting}

Just as MacMahon did, we begin by using the MacMahon lifting, which we already met in Section~\ref{sec:partition-analysis}, to transform the representation of the problem. The MacMahon lifting is implemented by the function $\MacMahon$ defined in Algorithm~\ref{alg:macmahon}.

Given $A\in\ZZ^{m\times d}$ and $b\in\ZZ^m$, we construct a matrix $V\in\ZZ^{(d+m)\times d}$ and a vector $q\in\ZZ^{d+m}$ by
\[
  V = \mmatrix{\Id_d \\ A} \text{ and } q= \mvec{0 \\ -b}.
\]
Geometrically, we view this construction as defining a cone $C = \CCC(V;q)$ with apex $q$ and the columns of $V$ as generators. As already mentioned, the crucial property of $C$ is that 
\[
\Omega_{\geqslant}^m(C) := \underbrace{\Omega_\geqslant(\Omega_\geqslant(\cdots(\Omega_\geqslant}_{\text{$m$ times}} (C) )))=\mset{z\in\RR_{\geqslant0}^d}{Az\geqslant b},
\]
where for any polyhedron $P\in\RR^n$,
\[
  \Omega_\geqslant(P) = \pi(H^+_n \cap P)
\]
and $\pi:\RR^{n}\rar \RR^{n-1}$ is the projection that forgets the last coordinate. Moreover, this construction of $C$ is such that for any polyhedron $P$ obtained throughout the process of applying $\Omega_\geqslant$, the projection $\pi$ maps $P$ bijectively onto its image and moreover preserves the lattice structure of $P$. More precisely, let $X$ denote the affine hull of $\mset{x}{Ax\geqslant b, x\geqslant 0}$ and let $X^j=\pi^j(X)$ denote the result of projecting-away the last $j$ coordinates. Then $\pi:\RR^{n-j+1}\rar\RR^{n-j}$ induces a bijection between $X^{j-1}$ and $X^j$. It follows that $\pi$ also induces bijections between $P$ and $\pi(P)$ for any $P\subset X^{j-1}$. This property of the MacMahon lifting will be crucial for our construction below. In fact, we have the even stronger property that $\pi$ gives a bijection between $\ZZ^{n-j+1}\cap X^{j-1}$ and $\ZZ^{n-j}\cap X^j$. This latter property is essential when working with rational functions, but it is irrelevant for our purposes, as we will be working with symbolic cones instead.

\begin{algorithm}
\Input{A matrix $A\in\ZZ^{m\times d}$ and a vector $b\in\ZZ^m$.}
\Output{A symbolic cone $C=(V,q,o)$ with $V\in\ZZ^{(d+m)\times d}$, $q\in\ZZ^{d+m}$ and $o\in\{0,1\}^d$ such that $\Omega_{\geqslant}^m(C)=\mset{z\in\RR^d_{\geqslant 0}}{Az\geqslant b}$.}
\Fn(){\MacMahon{$A,b$}}{
  $V = \mmatrix{\operatorname{Id}_d \\ A}$ \;
  $q = \mvec{0 \\ -b}$ \;
  $o = 0 \in \{0,1\}^d$ \;
  \Return $(V,q,o)$
}
\caption{MacMahon Lifting\label{alg:macmahon}}
\end{algorithm}

\subsection{Iterative Elimination of the Last Coordinate using Symbolic Cones} 
\label{sec:elimination}

A \emph{symbolic cone} $C$ is simply a triple $(V,q,o)$ where $V$ is a matrix of generators, $q$ is the apex of the cone and $o$ is a vector indicating which faces of $C$ are open. This is the form in which we are going to store cones throughout the algorithm. The triple $(V,q,o)$ represents the set $\CCC^o_\RR(V;q)$, just as defined in Section~\ref{sec:partition-analysis}. A \emph{linear combination of symbolic cones} is a formal sum $\sum \alpha_i C_i$, i.e., a list of pairs $(\alpha_i,C_i)$ where each $C_i$ is a symbolic cone given as a triple. The meaning of these linear combinations is best understood in terms of indicator functions. For any set $S\subset\RR^n$ we define the \emph{indicator function} $[S]:\RR^n\rar\RR$ by
\[
  [S](x)
    =
    \choice{ 
      1 & \text{ if } x\in S, \\
      0 & \text{ if } x\not\in S.
    }
\]
Indicator functions form a vector space. In particular, addition of indicator functions is defined pointwise, as is multiplication with a scalar. For a symbolic cone $C=(V,q,o)$ we will write $[C]$ to denote the indicator function $[\CCC_\RR^o(V;q)]$ of the associated simplicial cone. Thus, the formal sum $\sum \alpha_i C_i$ represents the indicator function $\sum \alpha_i [C_i]$. Later on we will also talk about equivalence classes of such indicator functions, but it is important to stress that by default $[S]$ stands for the actual indicator function and not an equivalence class thereof.

In the second step of the algorithm, we compute $\Omega_\geqslant^k(C)$ by iteratively applying a function $\ElimLC$, summarized in Algorithm~\ref{alg:eliminatelastcoordinate}, that takes a linear combination $\sum \alpha_i C_i$ of symbolic cones and returns a representation of $\Omega_\geqslant(\sum \alpha_i [C_i])$ as a linear combination of symbolic cones. 

The action of the Omega operator on indicator functions is induced by its action on sets. There is, however, an important subtlety. The intersection with the half-space $H_n^+$ is straightforward to define for indicator functions: For any indicator function $f$,
\[
  (H_n^+\cap f)(x) := \choice{
    f(x) & \text{ if } x\in H_n^+, \\
    0 & \text{ if } x\not\in H_n^+.
  }
\]
To define the projection, we make use of the property of the MacMahon lifting that for any set $S$ appearing throughout the algorithm and any $(x_1,\ldots,x_{n-1})\in \pi(S)$ there exists a \emph{unique} $x_n$ such that $(x_1,\ldots,x_n)\in S$ and the same holds for indicator functions. This allows us to define $\pi(f)$ for an indicator function $f$ by
\[
  \pi(f)(x_1,\ldots,x_{n-1}) = \choice{
    f(x_1,\ldots,x_{n-1},x_n) & \text{if }\exists x_n: f(x_1,\ldots,x_{n-1},x_n) \not=0 \\
    0 & \text{otherwise. }
  }
\]
Since, if it exists, the $x_n$ in the above definition is uniquely determined, it follows that $\pi(f)$ is well-defined. Actually, all indicator functions $f$ to which we apply $\pi$ moreover satisfy $f(x_1,\ldots,x_{n-1},x_n)=1$, however, we will not make use of this fact. With these preparations we can now define 
\[
  \Omega_\geqslant(f):=\pi(H_n^+\cap f).
\]
Since the Omega operator is linear, we have the important consequence that 
\[
  \Omega_\geqslant\left(\sum \alpha_i [C_i]\right) = \sum \alpha_i \Omega_\geqslant([C_i]).
\]
Therefore it suffices to define $\ElimLC$ on symbolic cones $C_i$ -- we do not need to handle general polyhedra or indicator functions directly.

To compute $\Omega_\geqslant(C)$ for a simplicial symbolic cone $C$, we decompose $H^+_n\cap C$ into a signed sum of cones. The most well-known decomposition of this type is the Brion decomposition which states, roughly, that ``modulo lines'' a polyhedron is equal to the sum of the tangent cones at its vertices. Let us make this statement precise. For a given polyhedron $P$ and a point $v$, the \emph{tangent cone} of $P$ at $v$ is
\[
  \tcone(P,v) = \mset{v+u}{v+\delta u \in P \text{ for all sufficiently small $\delta>0$ }}.
\]
Similarly, the \emph{feasible cone} of $P$ at $v$ is 
\[
  \fcone(P,v) = \mset{u}{v+\delta u \in P \text{ for all sufficiently small $\delta>0$ }},
\]
i.e., $\tcone(P,v) = v + \fcone(P,v)$. We call the tangent cones of $P$ at its vertices $v$ the \emph{vertex cones} of $P$. For each vertex cone $C_i=\tcone(P,v_i)$ of $P$ let $\rho_{C_i}$ denote a rational function expression for the generating function $\Phi_{C_i}$ of lattice points in $C_i$. For simplicial cones, $\rho_{C_i}$ can be obtained via (\ref{eqn:Ehrhart-general}) as we have seen in Section~\ref{sec:partition-analysis}. For non-simplicial cones, such expressions can be obtained via triangulation, for example, as in Figure~\ref{fig:fourthmm}. Fortunately, it turns out that we will only have to deal with the simplicial case in our algorithm -- we never have to triangulate. Given the above notation, Brion's theorem states the following.

\begin{theorem}[Brion \cite{Brion1988}]
\label{thm:brion}
For any polyhedron $P$ with vertices $v_1,\ldots,v_N$,
\[
  \Phi_P(x) = \sum_{i=1}^N \rho_{\tcone(P,v_i)}(x).
\]
\end{theorem}

Geometrically, i.e., on the level of linear combinations of indicator functions, this identity can be read as
\[
  [P] \equiv \sum_{i=1}^N [\tcone(P,v_i)] \text{ modulo lines }.
\]
The phrase ``modulo lines'' means that this identity holds true up to a finite sum of indicator functions of polyhedra which contain lines. A line is a set of the form $a+\RR b$ for some vectors $a,b$. See \cite{BarvinokIntegerPoints} for details, in particular Theorem~6.4 therein.

For our purposes, working modulo lines is not ideal, since the Omega operator is not well-defined modulo lines, as we already discussed in Section~\ref{sec:xin-convergence}. Another example that shows this, even in the simple case of one-dimensional intervals\footnote{Here, we simplify notation and write $[a,b]$ instead of $[[a,b]]$ for the indicator function of the interval $[a,b]$.}, is this: 
\[
  \Omega_\geqslant((-\infty,1]) = [0,1] \not\equiv -(1,\infty) = -\Omega_\geqslant((1,\infty))
\]
even though $(-\infty,1] \equiv -(1,\infty)$ modulo lines. We therefore have to work with a decomposition that holds exactly, not just modulo lines. 

One option is the Lawrence-Varchenko decomposition, which makes use of the notion of forward cones \cite{Lawrence}. In general, ``forward'' can be defined with respect to an arbitrary half-open half-space. For our purposes, the following more particular definition will suffice, however: A vector $v$ is forward if its first non-zero entry is positive and a symbolic cone is \emph{forward} if all of its generators $v_{i}$ are forward. Note that the cone obtained from the MacMahon lifting is forward by construction.

For any non-zero vector $v$, we have that $v$ is forward if and only if $-v$ is not forward. Therefore, we can turn a simplicial cone into a forward simplicial cone by reversing the direction of all ``backward'' generators -- as long as we adjust the sign and which faces of the new cone are open. The resulting forward cone will be equivalent modulo lines to the original cone. Here we make use of the fact that for any vector $v\not=0$, the one-dimensional ray $\CCC^0(v)$ satisfies $[\CCC^0(v)]\equiv -[\CCC^1(-v)]$ modulo lines.

In general, let $C=(V,q,o)$ denote a symbolic cone, with $V\in\ZZ^{d\times k}$, $q\in\QQ^d$ and $o\in\{0,1\}^k$. Let $\bwd(C)$ denote the number of columns of $V$ that are backward, i.e., not forward, and define $\sgn(C):=(-1)^{\bwd(C)}$. Let $V'$ denote the matrix obtained from $V$ by multiplying all backward columns by $-1$. Let $o'$ denote the vector obtained from $o$ by letting $o'_j=1-o_j$ if the $j$-th column has been reversed and $o'_j = o_j$ otherwise. Define $\Flip(C):=(V',q,o')$, as summarized in Algorithm~\ref{alg:flip}. Then we have that $\Flip(C)$ is forward and moreover
\[
  [\Flip(C)] \equiv \sgn(C) [C] \text{ modulo lines}.
\]
Applying this operation to Brion's theorem, we can obtain the following identity of indicator functions, which is essentially the theorem of Lawrence-Varchenko, even though their theorem was originally stated in terms of generating functions.

\begin{algorithm}
\Input{A symbolic cone $C=(V,q,o)$ with $V\in\ZZ^{d\times k}$, $q\in\QQ^d$ and $o\in\{0,1\}^k$.}
\Output{A pair $(s,C')$ of a symbolic cone $C'$ and a sign $s$ such that $[C]\equiv s[C']$ modulo lines and $s=\sgn(C)$.}
\Fn(){\Flip{$C$}}{
  $f(i) = $ the $i$-th column of $V$ is backward \;
  $\sigma(i) = 1$ \If $f(i)$ \Else $-1$ \;
  $s = \prod_{i=1}^k \sigma(i)$ \;
  $V' = $ the matrix obtained from $V$ by multiplying the $i$-th column by $\sigma(i)$\;
  $o' = $ the vector defined by $o_i' = $ \Xor{$o_i,f(i)$} for $i = 1,\ldots,k$ \;
  $C' = (V',q,o')$ \;
  \Return $(s,C')$
}
\caption{Flip\label{alg:flip}}
\end{algorithm}

\begin{theorem}[Lawrence-Varchenko \cite{Lawrence,Varchenko1987}]
\label{thm:lawrence-varchenko}
For any polyhedron $P$ with vertices $v_1,\ldots,v_N$,
\[
  [P] = \sum_{i=1}^N \sgn(\tcone(P,v_i)) [\Flip(\tcone(P,v_i))]
\]
\end{theorem}

The crucial point here is that this identity holds \emph{exactly}, without taking equivalence classes modulo lines. The advantage of this exact identity over Brion's result is that we can use it safely in conjunction with the Omega operator. Since we are not taking equivalence classes, the problem of well-definedness does not even arise. Working with rational function representations necessitates working with equivalence classes, a complication we can avoid by working on the level of symbolic cones instead.

The key observation behind the Polyhedral Omega algorithm is now that for the simplicial cones $C$ appearing through the course of the algorithm, we can give \emph{explicit formulas} for a Lawerence-Varchenko decomposition of the polyhedron $C \cap H^+_n$ into simplicial cones. This elimination of the last variable is performed by the function $\ElimLC$ given in Algorithm~\ref{alg:eliminatelastcoordinate}. To see the correctness of this algorithm, we proceed just as we did above: We first derive explicit formulas that determine a Brion decomposition and then we convert this Brion decomposition into a Lawrence-Varchenko decomposition by flipping all cones forward. Since the resulting decomposition is an exact identity of indicator functions, the Omega operator can safely be applied to the resulting decomposition in the next iteration.

\begin{algorithm}
\Input{A symbolic cone $C=(V,q,o)$ with $V\in\ZZ^{n\times k}$, $q\in\QQ^n$ and $o\in\{0,1\}^k$, such that the projection $\pi$ that forgets the last coordinate induces a bijection between $\aff(C)$ and $\RR^{n-1}$.}
\Output{A linear combination $\sum_i \alpha_i C_i$ of symbolic cones, represented as a dictionary mapping cones $C_i=(V_i,q_i,o_i)$ to multiplicities $\alpha_i$, such that $\Omega_\geqslant([C]) = \sum_i \alpha_i [C_i]$.}
\Fn(){\ElimLC{$C$}}{
  $v_j = $ the $j$-th column of $V$ \;
  $w_j = q - \frac{q_n}{v_{j,n}} v_j$ \;
  $J^+ = \mset{j\in[k]}{v_{j,n} > 0}$ \;
  $J^- = \mset{j\in[k]}{v_{j,n} < 0}$ \;
  $o'(j) = (o'(j)_i)_{i=1,\ldots,k}$ 
    where 
      $o'(j)_i = o_i$ \If $i\not=j$ \Else $0$
  $T^j = \mmatrix{ 
  -v_{j,n} &        &           &    &           &        & \\
           & \ddots &           &    &           &        & \\
           &        & -v_{j,n}  &    &           &        & \\
   v_{1,n} & \cdots & v_{j-1,n} & -1 & v_{j+1,n} & \cdots & v_{k,n} \\
           &        &           &    & -v_{j,n}  &        & \\
           &        &           &    &           & \ddots & \\
           &        &           &    &           &        & -v_{j,n}
  }$ \;
  $G^j = \prim(\pi(\sg(q_n) \cdot V \cdot T^j))$ \;
  $B^+ = [ (G^j,\pi(w_j),o'(j)) $ \For $ j\in J^+ ]$ \;
  $B^- = [ (G^j,\pi(w_j),o'(j)) $ \For $ j\in J^- ] \cup [ (\prim(\pi(V)),\pi(q),o) ]$ \;  
  $B = B^- $ \If $q_n \geqslant 0$ \Else $B^+$ \;
  $D = $ the dictionary that maps $\flip(C')$ to $\sgn(C')$ for all $ C' \in B$ \;
  \Return $D$
}
\caption{Eliminate Last Coordinate\label{alg:eliminatelastcoordinate}}
\end{algorithm}

To derive these formulas, we have to distinguish three cases: Whether the apex $q$ of $C$ lies above the hyperplane $H_n=\mset{x}{x_n = 0}$, on the hyperplane or below the hyperplane.

Note that our original matrix $A$ has $m$ rows and $d$ columns. After the MacMahon lifting, we are then dealing with a $d$-dimensional cone $C$ in $\RR^{d+m}$. Through successive eliminations of the last coordinate, we will eventually obtain a linear combination of $d$-dimensional cones in $\RR^d$. Throughout this process we will denote the index of the last coordinate with $n=d+m,d+m-1,\ldots,d$ and we denote the number of generators by $k=d$.

\subsubsection*{Case A: Apex above the hyperplane.}

Let $C$ be a simplicial cone in $\RR^n$ with apex $q$, openness $o$ and generators $v_1,\ldots,v_k$ such that $q_n > 0$. This situation is illustrated in Figure~\ref{fig:brion-apex-above-hyperplane}. Since $q$ is contained in the half-space $H^+_n$, the polyhedron $C\cap H^+_n$ has $q$ as one of its vertices. The other vertices of $C\cap H^+_n$ are in one-to-one correspondence with the intersections $w_j$ of the extreme rays $R_j$ of $C$ with the hyperplane $H_n$. Since $q_n > 0$, the extreme ray $R_j$ generated by $v_j$ intersects $H_n$ if and only if $v_{j,n} < 0$, i.e., the generator $v_j$ points ``down'' with respect to the last coordinate. Let 
\[
J^- = \mset{j\in[k]}{v_{j,n} < 0}
\]
denote the set of all indices $j$ such that the corresponding generator $v_j$ points down. The intersection $w_j=H_n\cap R_j$ of the intersection hyperplane with such a ``downward'' ray $R_j$, $j\in J^-$ can be computed by
\[
  w_j = q - \frac{q_n}{v_{j,n}}v_j.
\]

\newcommand{\brionsize}{0.5}
\begin{figure}[t]
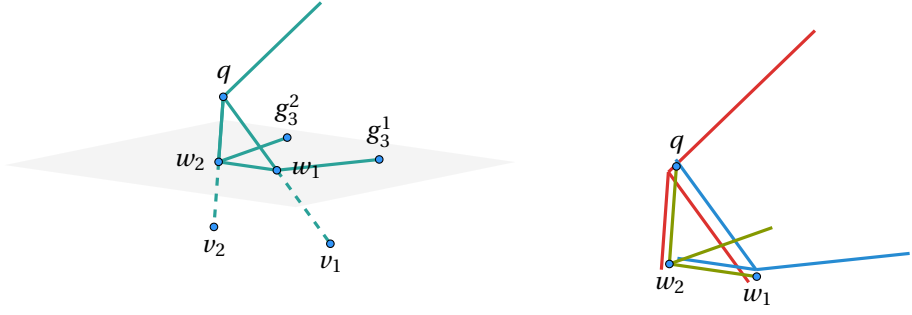


  \begin{subfigure}{0.45\textwidth}
    \centering
    \input{Graphics/brion-apex-above-hyperplane-1.tex}
    \caption{Intersecting a cone with the half-space $H^+_n$ with apex above the hyperplane $H_n$. }
  \end{subfigure} 
  \renewcommand{\brionsize}{0.75}  
  \begin{subfigure}{0.45\textwidth}
    \centering
    \input{Graphics/brion-apex-above-hyperplane-2.tex}
    \caption{ Brion decomposition of the intersection.}
  \end{subfigure}   
  \caption{Intersecting a cone with the half-space $H^+_n$ when the apex is above the hyperplane $H_n$. }
  \label{fig:brion-apex-above-hyperplane}
\end{figure}

Applying Brion's theorem, we find that, modulo lines, $C\cap H^+_n$ can be written as a sum of vertex cones at the vertices $q$ and $w_j$ for $j\in J^-$. The vertex cone at $q$ is identical to $C$, in particular it has the same generators. The generators $g^j_1,\ldots,g^j_k$ of the vertex cone $C_j:=\tcone(w_j,C\cap H^+_n)$ fall into three classes:
\begin{enumerate}
\item One of the generators of $C_j$, let us call it $g^j_j$, is just the reverse of the generator $v_j$ of $C$, i.e., 
\[
g^j_j = -v_j.
\] 
In other words, $g^j_j$ points from the apex $w_j$ of $C_j$ to the vertex $q$ of $C\cap H^+_n$.
\item \label{enum:case-a-vertex-vertex} For each $i\in J^-$ with $i\not = j$, there is a generator $g^j_i$ of $C_j$ that points to another vertex $w_i$ of $C\cap H^+_n$. These generators $g^j_i$ are given by
\begin{eqnarray*}
  g^j_i &=& \frac{v_{j,n}v_{i,n}}{q_n}(w_i - w_j)
  \;\;=\;\; 
   {v_{i,n}}v_j - {v_{j,n}}v_i.
\end{eqnarray*}
\item Finally, for each $i\in [k]\setminus J^-$, there is a generator $g^j_i$ that lies in the intersection of $H_n$ with the two-dimensional face of $C$ that is generated by $v_j$ and $v_i$. Note that $v_j$ points down while $v_i$ points up or is horizontal with respect to the last coordinate. We can therefore find a suitable $g^j_i$ by taking a linear combination of $v_j$ and $v_i$ with non-negative coefficients, such that the last coordinate of $g^j_i$ is zero. Thus, we obtain
\begin{eqnarray*}
  g^j_i &=& {v_{i,n}}v_j - {v_{j,n}}v_i.
\end{eqnarray*}
If $v_i$ is horizontal, i.e., $v_{i,n}=0$, then $g^j_i$ is just a multiple of $v_i$. 
\end{enumerate}
Note that even though, intuitively, the generators of the second and third kind arise from a different construction, the formulas defining $g^j_i$ in both cases are identical: $g^j_i$ is a linear combination of $v_j$ and $v_i$ such that the last coordinate of $g^j_i$ is zero.

It is important to note that the cones $\tcone(v,C\cap H^+_n)$, for $v$ a vertex of $C\cap H^+_n$, are all simplicial. If $v=q$, then this is true by assumption. If $v=w_j$ for some $j$, then we can observe this by noting that the matrix $G^j$, which has the generators $g^j_i$ of $C_j$ as columns, arises from the matrix $V$, which has the generators $v_i$ of $C$ as columns, via the transformation $G^j =  V \cdot T$ where $T$ is a $k\times k$ matrix of the form
\[
  T = \mmatrix{ 
  -v_{j,n} &        &           &    &           &        & \\
           & \ddots &           &    &           &        & \\
           &        & -v_{j,n}  &    &           &        & \\
   v_{1,n} & \cdots & v_{j-1,n} & -1 & v_{j+1,n} & \cdots & v_{k,n} \\
           &        &           &    & -v_{j,n}  &        & \\
           &        &           &    &           & \ddots & \\
           &        &           &    &           &        & -v_{j,n}
  }
\]
which is of full rank as $v_{j,n}<0$.

We have now determined apex and generators for each of the cones $C_j$. Moreover, we observe that the half-space $H^+_n$ is closed. Therefore the facet of $C_j$ corresponding to generator $g^j_j$ is closed, while all other facets inherit their openness from $C$. The openness vector $o'(j)=(o'(j)_1,\ldots,o'(j)_k)$ of $C_j$ is thus given by $o'(j)_j = 0$ and $o'(j)_i = o_i$ for $i\not=j$.

We thus have found an explicit Brion decomposition
\[
  [C\cap H^+_n] \equiv [C] + \sum_{j\in J^-} [C_j] \text{ modulo lines}. 
\]
To turn this into an explicit Lawrence-Varchenko decomposition of $\Omega_{\geqslant}(C)$, all we have to do is to a) flip all cones forward, by applying the map $\Flip$ defined above, and b) project away the last coordinate, by applying the projection $\pi : (x_1,\ldots,x_n) \mapsto (x_1,\ldots,x_{n-1})$. Both of these operations can be carried out easily on the level of symbolic cones. Here it is important to recall two important properties: On the one hand, $\Flip$ may introduce a sign in front of each cone and may change the openness of a cone by toggling the openness of each flipped generator. On the other hand, the affine hull of $\aff(C)$ of $C$ is such that $\pi|_{\aff(C)}$ is a bijection onto its image, so that in particular $\pi|_{C_j}$ is one-to-one for each cone $C_j$ in the decomposition. With these observations we obtain a decomposition
\[
  [\Omega_{\geqslant}(C)] = [\pi(C \cap H^+_n)] = [\pi(C)] + \sum_{j\in J^-} \sgn(C_j) [\pi(\Flip(C_j))]
\]
which holds exactly, i.e., not just modulo lines. Note that $C$ does not need to be flipped since we know that $C$ is forward, by construction.

\subsubsection*{Case B: Apex below the hyperplane.}

The case where the apex $q$ of cone $C$ lies strictly below the hyperplane $H_n$, i.e., $q_n < 0$, is similar to the previous one. The difference is that here, $q$ is not a vertex of $C\cap H^+_n$ as it is cut off by the intersection, as illustrated in Figure~\ref{fig:brion-apex-below-hyperplane}. The vertices $w_j$ of $C\cap H^+_n$ are thus in one-to-one correspondence with generators of $v_j$ of $C$ that point ``up'', i.e., that satisfy $v_{j,n} > 0$. Let 
\[
  J^+ = \mset{j\in[k]}{v_{j,n} > 0}
\]
and let $w_j:=H_n\cap R_j$ denote the intersection of $H_n$ with the ray $R_j=q+\RR_{\geqslant 0}v_j$ generated by $v_j$ for each $j\in J^+$. Then $w_j$ can be computed via
\[
  w_j = q - \frac{q_n}{v_{j,n}}v_j.
\]

\renewcommand{\brionsize}{0.5}
\begin{figure}[t]
  \begin{subfigure}{0.45\textwidth}
    \centering
    \input{Graphics/brion-apex-below-hyperplane-1.tex}
    \caption{Intersecting a cone with the half-space $H^+_n$ when the apex is below the hyperplane $H_n$.}
  \end{subfigure} 
  \renewcommand{\brionsize}{0.75}    
  \begin{subfigure}{0.45\textwidth}
    \centering
    \input{Graphics/brion-apex-below-hyperplane-2.tex}
    \caption{Brion decomposition of the intersection.}
  \end{subfigure}   
   \caption{Intersecting a cone with the half-space $H^+_n$ when the apex is below the hyperplane $H_n$.}
  \label{fig:brion-apex-below-hyperplane}
\end{figure}

Let $C_j=\tcone(w_j,C\cap H^+_n)$ denote the vertex cone at $w_j$. Again, we can think of the generators $g^j_1,\ldots,g^j_k$ of $C_j$ as belonging to three different types.
\begin{enumerate}
\item One generator, which we call $g^j_j$, is just the corresponding generator of $C$, i.e.
\[
  g^j_j = v_j.
\]
\item For every $i\in J^+$, $i\not=j$, there is one generator $g^j_i$ which points from $w_j$ to the vertex $w_i$ of $C\cap H^+_n$, i.e.
\begin{eqnarray*}
  g^j_i &=& -\frac{v_{j,n}v_{i,n}}{q_n} (w_i - w_j)
  \;\;=\;\; 
   {v_{j,n}}v_i - {v_{i,n}}v_j.
\end{eqnarray*}
Note that the sign is opposite from what we had in the previous case to ensure that the factor $-\frac{v_{j,n}v_{i,n}}{q_n}$ is positive, i.e., $g^j_i$ is a positive multiple of $w_i - w_j$.
\item For every $i \in [k] \setminus J^+$, there is one generator $g^j_i$ which lies in the intersection of $H_n$ with the face of $C$ generated by $v_j$ and $v_i$. Again, we take $g^j_i$ to be a non-negative linear combination of $v_i$ and $v_j$, namely
\[
  g^j_i = {v_{j,n}}v_i - {v_{i,n}}v_j.
\]
If $v_i$ is horizontal, i.e., $v_{i,n}=0$, then $g^j_i$ is just a multiple of $v_i$. Again, the terms in this expression have opposite signs compared to the previous case, in order to guarantee that the coefficients $v_{j,n}$ and $-{v_{i,n}}$ are both non-negative.
\end{enumerate}
Just as in the previous case, the expressions for generators of the second and third kind are identical. Also, we can set up a transformation matrix $T'$ to express the generator matrix $G^j$ of $C_j$ as a transformation of the generator matrix $V$ of $C$, via $G^j = V \cdot T'$. In this case, $T'$ is given by
\[
  T' = \mmatrix{ 
  v_{j,n}   &        &            &    &            &        & \\
            & \ddots &            &    &            &        & \\
            &        & v_{j,n}    &    &            &        & \\
   -v_{1,n} & \cdots & -v_{j-1,n} & 1  & -v_{j+1,n} & \cdots & -v_{k,n} \\
            &        &            &    & v_{j,n}    &        & \\
            &        &            &    &            & \ddots & \\
            &        &            &    &            &        & v_{j,n}
  }.
\]
First of all, we note that this has again full rank since $v_{j,n}>0$, which guarantees that the cones $C_j$ are simplicial. But more importantly, $T'=-T$, which means that both cases A and B can be treated with one common transformation
\[
  G^j = \sg(q_n)\cdot V\cdot T.
\]
where
\[
  \sg(x) = 1 \text{ if } x\geq 0 \text{ else } -1.
\]
The definition $\sg(0)=1$ will come in handy in case C, as we will see below.

We now have computed generators and apices for the vertex cones $C_j$. Again, the half-space $H^+_n$ is closed, whence the facet of $C_j$ corresponding to generator $g^j_j$ is closed, while all other facets inherit their openness from $C$. Therefore we now have an explicit Brion decomposition
\[
  [C\cap H^+_n] \equiv \sum_{j\in J^+} [C_j] \text{ modulo lines}.
\]
We proceed just as in the previous case and apply $\Flip$ and $\pi$ to obtain a Lawrence-Varchenko decomposition
\[
  [\Omega_{\geqslant}(C)] = [\pi(C\cap H^+_n)] = \sum_{j\in J^+} \sgn(C_j) [\pi(\Flip(C_j))]
\]
that holds exactly, not just modulo lines.

\subsubsection*{Case C: Apex on the hyperplane.}

\renewcommand{\brionsize}{0.5}
\begin{figure}[t]
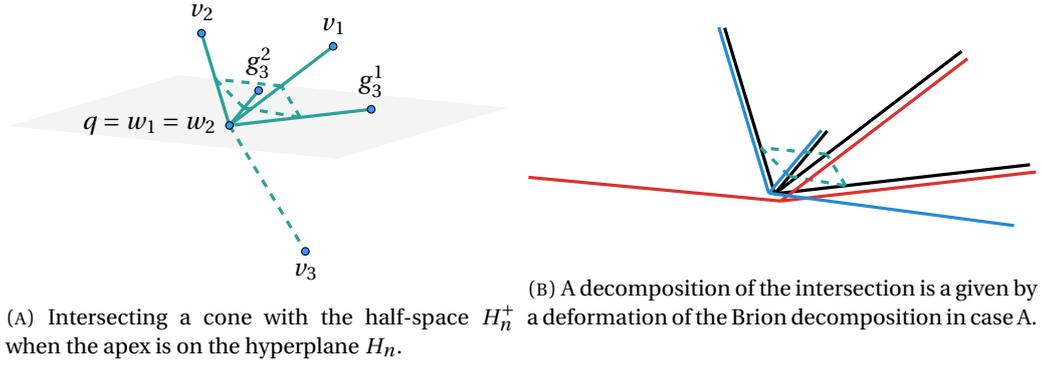

  \begin{subfigure}{0.45\textwidth}
    \centering
    \input{Graphics/brion-apex-on-hyperplane-1.tex}
    \caption{Intersecting a cone with the half-space $H^+_n$ when the apex is on the hyperplane $H_n$.}
  \end{subfigure} 
  \renewcommand{\brionsize}{0.9}    
  \begin{subfigure}{0.45\textwidth}
    \centering
    \input{Graphics/brion-apex-on-hyperplane-2.tex}
    \caption{A decomposition of the intersection is a given by a deformation of the Brion decomposition in case A.}
  \end{subfigure}   
  \caption{Intersecting a cone with the half-space $H^+_n$ when the apex is on the hyperplane $H_n$.}
  \label{fig:brion-apex-on-hyperplane}
\end{figure}

At first glance, the case where the apex $q$ of cone $C$ lies on the hyperplane $H_n$, i.e., $q_n = 0$, appears special. The reason is that in this case $C\cap H^+_n$ can be a non-simplicial cone, as shown in Figure~\ref{fig:brion-apex-on-hyperplane}(a), and that the only vertex cone of $C\cap H^+_n$ is $C\cap H^+_n$ itself. Thus a simple Brion decomposition would take us out of the regime of formal sums of simplicial cones. Of course we could solve this issue by triangulating $C\cap H^+_n$. But triangulation algorithms are a complex subject in their own right \cite{DeLoera2010,Lee2004,Pfeifle2003}, at odds with the simple approach based on explicit decomposition formulas we have applied thus far.

Fortunately, it turns out that the same formula we used to handle case A can also be applied in case C to obtain a decomposition into simplicial cones. It just so happens that all simplicial cones that appear have the same apex $q$. This phenome\-non can be explained with a deformation argument: For a small $\epsilon>0$, the apex of $\epsilon e_n + C$ lies above the hyperplane $H_n$ so that case $A$ applies and we obtain a decomposition
\[
  [(\epsilon e_n + C) \cap H^+_n] = [\epsilon e_n + C] + \sum_{j\in J} [C_j(\epsilon)].
\]
The generators of the cones $(\epsilon e_n + C)$ and $C_j(\epsilon)$ are independent of $\epsilon$. The apices $(\epsilon e_n+q)$ and $w_j(\epsilon)$ of the cones $(\epsilon e_n + C)$ and $C_j(\epsilon)$, on the other hand, do depend on $\epsilon$. However, this dependence is continuous so that $(\epsilon e_n+q) \rar q$ and $w_j(\epsilon) \rar q$ as $\epsilon \rar 0$. In the limit, we thus obtain the desired decomposition of $C\cap H^+_n$, as illustrated in Figure~\ref{fig:brion-apex-on-hyperplane}(b). That this works in general is the content of Liu's theorem.

\begin{theorem}[Liu \cite{Liu}]
\label{thm:liu}
Let $A\in\RR^{m\times n}$, $b:(-1,1)\rar\RR^m$ a continuous function and $P(t)=\mset{x}{Ax \geqslant b(t)}$. Suppose that for each $0\not=t\in(-1,1)$ the polyhedron $P(t)$ is non-empty and has precisely $\ell$ vertices $w_{t,1},\ldots,\allowbreak w_{t,\ell}$. Suppose further that for each $j=1,\ldots,\ell$ there exists a fixed cone $C_j$ such that
\[
  \fcone(w_{t,j},P(t)) = C_j \text{ for all }0\not=t\in(-1,1),
\]
i.e., the apex of each tangent cone $\tcone(w_{t,j},P(t))$ may change depending on $t$, but the feasible cone $\fcone(w_{t,j},P(t))$ does not. Then, for each $j$, the vertex $w_{t,j}$ converges to some vertex of $P(0)$ as $t\rar 0$. Let $v$ be a vertex of $P(0)$ and let $J_v$ be the set of all $j$ such that $w_{t,j}$ converges to $v$. Then
\[
  [\tcone(v,P(0))] \equiv \sum_{j\in J_v} [v + C_j] \text{ modulo lines}.
\]
\end{theorem}

To apply this theorem, we first observe that there exists a matrix $A$ and a continuous function $b$ defined on the half-open interval $[0,1)$ such that 
\[
  (\epsilon e_n + C) \cap H^+_n = \mset{x}{Ax\geqslant b(\epsilon)}.
\]
To extend $b$ continuously to $(-1,1)$ we simply let $b(-\epsilon):=b(\epsilon)$. By construction, $(\epsilon e_n + C) \cap H^+_n$ is non-empty for $\epsilon>0$. Moreover $(\epsilon e_n + C) \cap H^+_n$ has the same combinatorial type  for each $\epsilon>0$. In particular, the number of vertices of $(\epsilon e_n + C) \cap H^+_n$ and the feasible cones at these vertices do not change for $\epsilon>0$. For $\epsilon=0$ we obtain the polyhedron $C\cap H^+_n$ that we are interested in. $C\cap H^+_n$ has just one vertex, namely $q$, to which all the vertices of $(\epsilon e_n + C) \cap H^+_n$ converge. Therefore, the conclusion of Theorem~\ref{thm:liu} gives us the desired decomposition
\[
  [C\cap H^+_n] \equiv [C] + \sum_{j\in J^-} [C_j] \text{ modulo lines}
\]
where each $C_j$ is given by precisely the same formulas as in case $A$. In particular, the formula for $w_j$ gives $q$ as the apex of each of the $C_j$, as expected. We can view this decomposition as giving, implicitly, a triangulation of the dual cone of $C \cap H^+_n$. Just as in the previous two cases, we obtain an exact decomposition of $C \cap H^+_n$ (not just modulo lines) by flipping all cones on the right-hand side forward. By projection we obtain
\[
  [\Omega_{\geqslant}(C)] = [\pi(C\cap H^+_n)] = [\pi(C)] + \sum_{j\in J^-} \sgn(C_j) [\pi(\Flip(C_j))].
\]
Note that our convention $\sg(0)=1$ ensures that $G^j = \sg(q_n)\cdot V\cdot T$ holds for all three cases.

\subsubsection*{Summary and Optimizations}

We have now seen that $\ElimLC$ correctly applies $\omeg$ to a linear combination of symbolic cones of the form produced by $\MacMahon$. The function $\Elim$, defined in Algorithm~\ref{alg:eliminatecoordinates}, iterates this elimination of the last variable $m$ times. As the output of $\Elim$ we thus obtain a list of simplicial symbolic cones $C_1,\ldots,C_N$ and multiplicities $\alpha_1,\ldots,\alpha_N$ such that
\begin{eqnarray}
\label{eqn:symbolic-cone-decomposition}
  [\mset{z\in\RR^d}{Az\geqslant b}] = \sum_{i=1}^N \alpha_i [C_i]
\end{eqnarray}
where each symbolic cone $C_i$ is given in terms of a triple $(V,o,q)$ where $V$ is an integer matrix of generators, $o$ is a vector indicating which faces of the cone are open and $q$ is a rational vector giving the apex of the cone. Note that each $C_i$ is a $d$-dimensional cone in $\RR^d$. The above identity of indicator functions holds exactly, not just modulo lines. 

\begin{algorithm}
\Input{A symbolic cone $C=(V,q,o)$ with $V\in\ZZ^{d+m\times d}$, $q\in\QQ^{d+m}$ and $o\in\{0,1\}^d$ such that projection that forgets the last $m$ coordinates induces a bijection between $\aff(C)$ and $\RR^d$.}
\Output{A linear combination $\sum_i \alpha_i C_i$ of symbolic cones, represented as a dictionary mapping cones $C_i=(V_i,q_i,o_i)$ to multiplicities $\alpha_i$, such that $\Omega^m_\geqslant([C]) = \sum_i \alpha_i [C_i]$.}
\Fn(){\Elim{$C$}}{
  \ForBlock{$i = 1,\ldots,m$}{
    $C \leftarrow \Map{\ElimLC,C}$
  }
  \Return $C$
}
\caption{Eliminate Coordinates\label{alg:eliminatecoordinates}}
\end{algorithm}

The elimination algorithm is easy to implement, since all cones are stored in terms of integer matrices and all transformations are given explicitly. However, it is important to highlight two implementation aspects for performance reasons. 

First, it is important to ensure that the generators for a given cone $C_i$ are stored in primitive form, as defined in Section~\ref{sec:partition-analysis}. Algorithm~\ref{alg:prim} shows how to compute $\prim(v)$ for integer vectors. Replacing a column $v$ of $V$ with $\prim(v)$ does not change the cone $C_i$ the generator matrix $V$ defines, but it may significantly reduce the encoding size of the matrix $V$. Therefore, after each round of elimination, all columns $v$ of $V$ should be replaced by $\prim(v)$, respectively, since the transformations applied in each round may introduce non-primitive columns.

\begin{algorithm}
\Input{A non-zero integer vector $v\in\ZZ^n$.}
\Output{The shortest integer vector that is a positive multiple of $v$.}
\Fn(){\prim{$v$}}{
  $g = \gcd(v_1,\ldots,v_n)$ \;
  \Return $\frac{v}{|g|}$
}
\caption{Make a vector primitive \label{alg:prim}}
\end{algorithm}

Second, one should note that after each round of elimination, a given cone $C_i$ may appear multiple times in the sum (\ref{eqn:symbolic-cone-decomposition}), with multiplicities of opposite signs. It is therefore crucial to collect terms after each elimination and sum up the multiplicities of all instances of a given cone $C_i$, as a significant amount of cancellation may occur in some cases. For this purpose, the generators of the cones $C$ should be stored in (lexicographically) sorted order, since when generators are primitive and sorted, every rational simplicial cone has a \emph{unique} representation as a symbolic cone. Also, it may be expedient to store the sum (\ref{eqn:symbolic-cone-decomposition}) as a dictionary mapping symbolic cones $C_i$ to multiplicities $\alpha_i$ throughout the computation, see the definition of $\Map$ (Algorithm~\ref{alg:map}).

\begin{algorithm}
\Input{A function $f$ which maps a symbolic cones to a linear combination of cones, and a linear combination $\sum_i \alpha_i C_i$ of symbolic cones, represented as a dictionary $D$ mapping symbolic cones $C_i$ to multiplicities $\alpha_i$.}
\Output{A dictionary representing the linear combination $\sum_i \alpha_i f(C_i) = \sum_{i,j} \alpha_i \beta_{i,j} C_{i,j}$ of symbolic cones where $f(C_i) = \sum_j \beta_{i,j} C_{i,j}$. Coefficients are collected by the $C_{i,j}$, i.e., if $C:=C_{i_1,j_1}=\ldots=C_{i_k,j_k}$ then the dictionary contains only the key-value-pair $(C \rar \sum_{l=1}^k \alpha_{i_l} \beta_{i_l,j_l})$.}
\Fn(){\Map{$D,f$}}{
  $L = $ an empty dictionary with default value $0$\;
  \ForBlock{$(C\rar \alpha)\in D$}{
    $D' = f(C)$ \;
    \ForBlock{$(C' \rar \beta)\in D'$}{
      $L[C'] \leftarrow L[C'] + \alpha \cdot \beta$
    }
  }
  \Return $L$
}
\caption{Map Linear Combination of Symbolic Cones\label{alg:map}}
\end{algorithm}

\subsection{Conversion to Rational Functions}
\label{sec:rat-fun-conversion}

The decomposition (\ref{eqn:symbolic-cone-decomposition}) into symbolic cones that we have so far obtained has many advantages over a rational function expression for $\Phi_P$. In particular, (\ref{eqn:symbolic-cone-decomposition}) is very compact, highly structured and has a clear geometric meaning, which makes it suitable for further processing by human and machine alike. Therefore, when applying our algorithm, we recommend investigating whether the representation (\ref{eqn:symbolic-cone-decomposition}) can be used instead of a rational function expression.

In case a rational function expression for $\Phi_P$ is indeed required, we present two different methods for computing it in this section. The first method explicitly enumerates the lattice points in the fundamental parallelepiped of a cone, using an explicit formula based on the Smith normal form of the matrix of generators. The second method uses Barvinok's recursive decomposition to express a given cone as a short signed sum of unimodular cones. The first approach has the advantage of giving the rational function explicitly, in terms of a simple formula, rather than appealing to a complex recursive algorithm. The second approach has the advantage that it produces much shorter formulas. In fact, if the dimension is fixed, Barvinok's algorithm is guaranteed to produce polynomial size expressions, whereas the number of points in the fundamental parallelepipeds involved may very well be exponential. However, the representation (\ref{eqn:symbolic-cone-decomposition}) in terms of symbolic cones is shorter than either of the two rational function expressions. Moreover, it turns out that in practice, conversion to rational functions is usually the bottleneck of the algorithm, taking much longer than the computation of (\ref{eqn:symbolic-cone-decomposition}) on a large class of inputs, no matter which of the two conversion methods is used, see Section~\ref{sec:complexity}. Both approaches have in common that they convert (\ref{eqn:symbolic-cone-decomposition}) into a rational function identity one simplicial cone at a time. If $\rho_{C_i}(z)$ is a rational function expression for $\Phi_{C_i}(z)$ for any $i$, then
\begin{eqnarray*}
\Phi_P(z) = \sum_{i=1}^N \alpha_i \rho_{C_i}(z).
\end{eqnarray*}
Therefore, we can restrict our attention to computing rational function representations $\rho_C$ for a given simplicial symbolic cone $C$.

\subsubsection*{Fundamental Parallelepiped Enumeration}

As we have already seen in Section~\ref{sec:partition-analysis},
\[
  \rho_C(z) = \frac{\sum_{v\in\ZZ^n\cap\Pi^o(v_1,\ldots,v_d;q)} z^v }{(1-z^{v_1})\cdot\ldots\cdot(1-z^{v_d})}
\]
where the $v_i$ are the generators of $C$ and $\Pi^o(v_1,\ldots,v_d;q)$ is the fundamental parallelepiped of $C$ with given openness $o$ and apex $q$. The number of lattice points in the fundamental parallelepiped $\Pi$ is equal to the determinant of the matrix $V$ of generators (provided V is square) and this determinant is exponential \emph{in the encoding length of $V$} even if the dimensions of $V$ are fixed. Therefore the polynomial in the numerator of the above expression can be exponential in size, even in fixed dimension.

Nonetheless, this representation of $\rho_C$ is attractive because the numerator is very simple to compute. Contrary to what one might think at first glance, enumerating the lattice points in the fundamental parallelepiped does not require solving another linear Diophantine system. Instead, this set of points can be generated directly, given the Smith normal form of the matrix $V$. This technique is well-known \cite{Bruns2012,Koppe2008}, but the method is very instructive and the formula we use is slightly different from \cite{Bruns2012,Koppe2008} so we provide its derivation here. To simplify the exposition, we will assume that the apex $q=0$ is at the origin and that $o=0$, i.e., the cone $C$ is closed.

Every integer matrix $V\in\ZZ^{m\times n}$ has a \emph{Smith normal form} which we define as a triple of integer matrices $U,S,W$ with $V=USW$ such that the following properties hold: First, $U\in\ZZ^{m\times m}$ and $W\in\ZZ^{n \times n}$ are \emph{unimodular}, i.e., they are invertible over the integers or, equivalently, $\det U = \det W = \pm 1$. Second, $S\in\ZZ^{m \times n}$ is a diagonal matrix with diagonal entries $s_1,\ldots,s_{\min{(m,n)}}$ such that if $k=\rank V$ then $s_1,\ldots,s_k>0$ and $s_{k+1},\ldots,s_{\min(m,n)}=0$ and, moreover, $s_1 | s_2$, $s_2|s_3$, $\ldots$, $s_{k-1} | s_k$. The matrix $S$ is uniquely determined, but the transformations $U$ and $W$ are not. Nonetheless we will simply say that $USW=V$ is ``the'' Smith normal form of $V$. Note that if $V$ is square then $\det V = s_1 \cdot s_2 \cdot \ldots \cdot s_k$. Since in our situation $C$ is always a simplicial cone, $V$ will always have full rank so that $k$ is the number of generators of $C$. For now, we will make the additional assumption that $C$ is full-dimensional which means that $V$ is square and $k=m=n$.

\begin{figure}[t]
\center{\includegraphics[width=12cm]{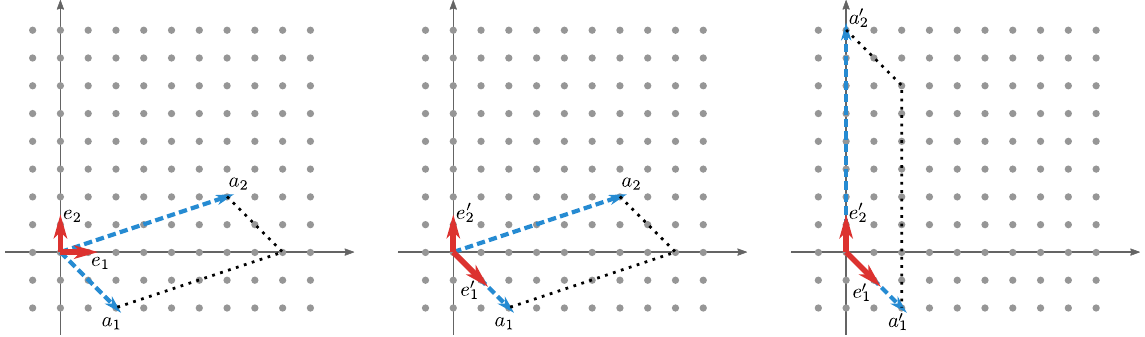}}
\caption[]{
  \label{fig:fundpar-example-smith} 
  The Smith normal form can be viewed as aligning the bases of $\ZZ^n$ and a sublattice $L$ so that the fundamental parallelepiped of the sublattice is a rectangle with respect to the chosen bases. In this example, the initial basis of $L$ is given by the columns of $V=\msmat{2 & 6 \\ -2 & 2}$. The Smith normal form is $V=USW$ with $U=\msmat{1&0\\-1&1}$, $W=\msmat{1&3\\0&1}$ and $S=\msmat{2&0\\0&8}$. Note that $U^{-1}=\msmat{1&0\\1&1}$ and $W^{-1}=\msmat{1&-3\\0&1}$.
}
\end{figure}

One way to interpret the Smith normal form is the following. Let $L$ denote the lattice spanned by the columns of $V$. $L$ is a sublattice of the integer lattice $\ZZ^n$. In this setting, the transformations $U$ and $W$ can be viewed as changing bases on both $L$ and $\ZZ^n$ so that with respect to the new bases the fundamental parallelepiped of $L$ is just a rectangle with side-lengths given by the diagonal elements $s_i$.  Figure~\ref{fig:fundpar-example-smith} illustrates this idea with an example. Suppose we start out with the standard basis $e_1,\ldots,e_n$ of $\ZZ^n$ and with a given basis $v_1,\ldots,v_n\in \ZZ^n$ of $L$. The $v_i$ are the columns of $V$ or, in our application, the generators of a simplicial cone $C$. The matrix $U$ now transforms the basis of $\ZZ^n$ into a new basis $e_1',\ldots,e_n'$ where $e_i'=\sum_{j=1}^n U_{j,i}e_j$. Simultaneously, $W$ transforms the basis of $L$ into a new basis $v_1',\ldots,v_n'$ where $v_i' = \sum_{j=1}^n W^{-1}_{j,i} v_i$. $S$ now gives the coordinates of the $v_i'$ relative to the basis $e_1',\ldots,e_n'$. The fact that $S$ is diagonal now means that $v_i'$ is simply a multiple of $e_i'$, or, in other words, that, with respect to the basis $e_1',\ldots,e_n'$, the fundamental parallelepiped $\Pi(v_1',\ldots,v_n')$ of the basis $v_1',\ldots,v_n'$ of $L$ is simply a rectangle with sidelengths given by the $s_i$. Thus, the set of lattice points in $\Pi(v_1',\ldots,v_n')$ is extremely simple: it is (a linear transformation of) a Cartesian product
\begin{eqnarray*}
\ZZ^n\cap\Pi(v_1',\ldots,v_n') &=& \mset{\sum_{i=1}^n\mu_ie_i'}{0\leq\mu_i<s_i,\mu_i\in\ZZ} \;\;\; = \;\;\; U \left( [0,s_1)_\ZZ \times \ldots \times [0,s_n)_\ZZ \right) 
\end{eqnarray*}
where the intervals $[0,s_i)_\ZZ$ are taken to denote the integer values in the given range.

\begin{figure}[t]
\center{\includegraphics[width=7cm]{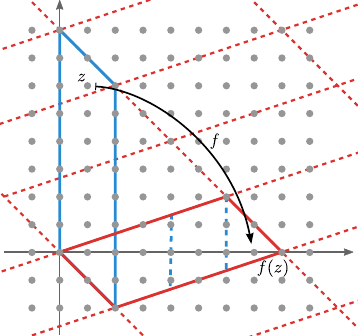}}
\caption[]{
  \label{fig:fundpar-example-rounding} 
  Transforming $\ZZ^2\cap\Pi(\mmatrix{2&0\\-2&8})$ into $\ZZ^2\cap\Pi(\mmatrix{2&6\\-2&2})$.
}
\end{figure}

We now understand $\Pi(v_1',\ldots,v_n')$ very well, but what is the relationship between $\Pi(v_1',\ldots,v_n')$ and the fundamental parallelepiped $\Pi(v_1,\ldots,v_n)$ that we set out to enumerate? It turns out we can easily transform the set of lattice points in one into the set of lattice points in the other through linear transformations and simple modular arithmetic, as shown in Figure~\ref{fig:fundpar-example-rounding}.  Since the matrices $U$ and $W$ are unimodular, it is clear that both $\Pi(v_1',\ldots,v_n')$ and $\Pi(v_1,\ldots,v_n)$ contain the same number of lattice points. Moreover, for both of the two fundamental parallelepipeds $\Pi$ it is true that for every $x\in\ZZ^n$ there exists a unique point $y\in\ZZ^n\cap\Pi$ such that has $x-y\in L$. Therefore, the function $f:\ZZ^n\rar\ZZ^n\cap\Pi(v_1,\ldots,v_n)$ that maps a lattice point $x$ to the corresponding $y\in\ZZ^n\cap\Pi(v_1,\ldots,v_n)$ induces a bijection
\[
  f|_{\Pi(v_1',\ldots,v_n')}:\ZZ^n\cap\Pi(v_1',\ldots,v_n')\rar\ZZ^n\cap\Pi(v_1,\ldots,v_n)
\]
between the set that we can describe and the set we want to describe. As it turns out, $f$ has a very simple explicit description: to compute $f(x)$ we simply take coordinates with respect to the basis $v_1,\ldots,v_n$, then we take the vector of fractional parts, and then undo the coordinate transformation. More precisely, for any $\alpha\in\RR$ let $\fract(\alpha)=\alpha-\floor{\alpha}$ denote the fractional part and for any vector $x\in\RR^n$ let $\fract(x)=(\fract(x_i))_{i=1,\ldots,n}$ denote the vector of fractional parts. Then
\[
  f(x) = V\fract(V^{-1}(x)).
\]


Putting these two observations together we obtain 
\begin{eqnarray*}
  \ZZ^n\cap\Pi(v_1,\ldots,v_n) &=& V\fract( V^{-1}( U \left( [0,s_1)_\ZZ \times \ldots \times [0,s_n)_\ZZ \right)) \\
  &=& V\fract(W^{-1}S^{-1} \left( [0,s_1)_\ZZ \times \ldots \times [0,s_n)_\ZZ \right))
\end{eqnarray*}
Note that $S^{-1}$ is the diagonal matrix with diagonal entries $\frac{1}{s_i}$. All other matrices are integer. For computational efficiency it is expedient to work with integer values throughout. The common denominator of all values in $S^{-1}$ is simply $s_n$. Let $s_i':=\frac{s_n}{s_i}\in\ZZ$ and let $S'$ denote the diagonal matrix with entries $s_1',\ldots,s_n'$. Then we can rewrite the previous identity to
\begin{eqnarray}
\label{eqn:simple-snf-set}
\ZZ^n\cap\Pi(v_1,\ldots,v_n) &=& \frac{1}{s_n} V \left(  W^{-1} S' \left( [0,s_1)_\ZZ \times \ldots \times [0,s_n)_\ZZ \right) \mmod s_n \right)
\end{eqnarray}
where the modulus is taken componentwise, i.e., $x \mmod \alpha := (x_i \mmod \alpha)_{i=1,\ldots,n}$ for a vector $x=(x_i)_{i=1,\ldots,n}$. The computation (\ref{eqn:simple-snf-set}) can be performed entirely in the integers. In particular, the final division by $s_n$ is guaranteed to produce only integer values.

Translating (\ref{eqn:simple-snf-set}) into the language of generating functions, we get the following explicit rational function formula $\rho_C(z)$ for the generating function of a symbolic cone $C$:
\begin{eqnarray}
\label{eqn:simple-snf-rat}
\rho_C(z) = 
  \frac{
    \sum_{k_1=0}^{s_1-1} \cdots \sum_{k_n=0}^{s_n-1} 
    z^{ \frac{1}{s_n} V(W^{-1} (s_1'k_1,\ldots,s_n'k_n)^\top \mmod s_n ) } 
  }{
    (1-z^{v_1})\cdot\ldots\cdot(1-z^{v_n})
  }
\end{eqnarray}

The formulas (\ref{eqn:simple-snf-set}) and (\ref{eqn:simple-snf-rat}) give a very concise and simple solution to the problem of listing lattice points in fundamental parallelepipeds and thus for converting cones to rational function expressions. The drawback of the rational function representation (\ref{eqn:simple-snf-rat}) is that the polynomial in the numerator may be exponential in the encoding size of $C$, even if the dimension is fixed. Nonetheless, computing (\ref{eqn:simple-snf-rat}) is much more efficient than solving a general linear Diophantine system, in the following sense: the Smith normal form (SNF) of $V$ can be computed in polynomial time, even if the dimension is not fixed, and once the SNF of $V$ is known, (\ref{eqn:simple-snf-set}) provides an explicit parametrization of the lattice points in the fundamental parallelepiped. In particular, this gives an output-sensitive polynomial time algorithm for listing the points in the fundamental parallelepiped.

As already mentioned, this method is well-known. The expression (\ref{eqn:simple-snf-set}) is essentially the same as \cite[Lemma~11]{Koppe2008}, the main difference being that (\ref{eqn:simple-snf-set}) avoids rational arithmetic for performance reasons. \cite{Bruns2012} gives a similar method based on the Hermite normal form.

The approach presented above can be extended to accommodate general symbolic cones $C$. In particular, it is possible to handle cones with open facets, cones with rational apices and non-full-dimensional cones. The fully general statement of (\ref{eqn:simple-snf-set}) and (\ref{eqn:simple-snf-rat}) is given in the following theorem. To deal with open facets we introduce the notation 
\[
a \momod{o} b = 
  \choice{ 
    a \mmod b & \text{if } a \mmod b \not= 0  \\
    o\cdot b & \text{if } a \mmod b = 0}
\] for $a,b\in\ZZ$ and $o\in\{0,1\}$, and let $x \momod{o} b := (x_i \momod{o_i} b)_{i=1,\ldots,n}$ be defined componentwise for vectors $x\in\ZZ^n$ and $o\in\{0,1\}^n$. Similarly, $\floor{x}:=(\floor{x_i})_i$ is defined componentwise.

\begin{theorem}
\label{thm:fundamental-parallelepiped-enumeration}
Let $C$ be a $k$-dimensional simplicial symbolic cone in $\RR^d$, given by a generator matrix $V\in\ZZ^{d\times k}$, an apex $q\in\QQ^d$ and a vector $o\in\{0,1\}^k$ that encodes which facets of $C$ are open. Let $USW=V$ denote the Smith normal form of $C$, so that in particular $U\in\ZZ^{n\times n},W\in\ZZ^{k\times k}$ are unimodular and $S\in\ZZ^{n\times k}$ is diagonal. Denote the diagonal entries of $S\in\ZZ^{n\times k}$ by $s_1,\ldots,s_k\in\ZZ$ and define $s_i'=\frac{s_k}{s_i}\in\ZZ$. Let $S'$ denote the $(k\times n)$-matrix with diagonal entries $s_i'$. 

If there is no lattice point in the affine hull of $C$, then $\ZZ^d\cap\Pi^o(V)=\emptyset$ and $\Phi_C(z)=0$. Otherwise, let $p$ be such a lattice point in the affine hull of $C$.\footnote{Both this decision and the computation of $p$, if it exists, can be readily performed using the Smith normal form.} Abbreviate $\hat{q}=U^{-1}(q-p)\in\QQ^n$ and define $\tilde{q}^{\text{int}}=\floor{-W^{-1}S'\hat{q}}\in\ZZ^k$ and $\tilde{q}^{\text{frac}}=-W^{-1}S'\hat{q}-\tilde{q}^{\text{int}}\in\QQ^k$. Let $o'$ be defined by $o'_j=o_j$ if $\tilde{q}^{\text{frac}}_j = 0$ and $o'_j=0$ otherwise. 

Then the lattice points in the fundamental parallelepiped of $C$ are given by
\begin{eqnarray*}
\Pi^o(V;q) &=& \mset{ \; y(x) }{ x\in \{0,\ldots,s_1-1\}\times\ldots\times\{0,\ldots,s_k-1\} \; } \text{ where } \\
y(x) &=& 
  \frac{1}{s_k} V( (W^{-1}S'x + \tilde{q}^{\text{int}}) \momod{o'} s_k ) 
  + \frac{1}{s_k} V \tilde{q}^{\text{frac}} 
  + q 
\end{eqnarray*}
and the generating function $\Phi_C(z)$ of $C$ is given by the rational function expression
\begin{eqnarray*}
\Phi_C(z) &=& 
\frac{
  \sum_{x_1=1}^{s_1-1} \cdots \sum_{x_k=1}^{s_k-1} z^{
    \frac{1}{s_k} V( (W^{-1}S'(x_1,\ldots,x_k)^\top + \tilde{q}^{\text{int}}) \momod{o'} s_k ) 
   + \frac{1}{s_k} V \tilde{q}^{\text{frac}} 
   + q 
  }
}{
  (1-z^{v_1})\cdot \ldots \cdot (1-z^{v_k})
}. 
\end{eqnarray*}
\end{theorem}

Theorem~\ref{thm:fundamental-parallelepiped-enumeration} gives rise to the function $\EnumFundPar$, defined in Algorithm~\ref{alg:enumfundpar}, for enumerating fundamental parallelepipeds of symbolic cones.

\begin{algorithm}
\Input{A symbolic cone $C=(V,q,o)$ with $V\in\ZZ^{d\times k}$, $q\in\QQ^d$ and $o\in\{0,1\}^k$.}
\Output{The list of all lattice points in the fundamental parallelepiped of $C$, in no particular order.}
\Fn(){\EnumFundPar{$C$}}{
  \IfBlock{$\aff(C)$ does not contain a lattice point}{
    \Return []
  }
  \ElseBlock{
    $p =$ a lattice point in $\aff(C)$ \;
    $U,S,W = \SNF{V}$ \;
    $s_1,\ldots,s_k =$ diagonal elements of $S$ \;
    $s'_i = \frac{s_k}{s_i}$ for all $i=1,\ldots,k$ \;
    $S' =$ the $(k\times n)$-matrix with diagonal entries $s'_i$ \;
    $\hat{q} = U^{-1}(q-p)$ \;
    $\tilde{q}^{\text{int}} = \floor{-W^{-1}S'\hat{q}}$ \;
    $\tilde{q}^{\text{frac}} = -W^{-1}S'\hat{q} - \tilde{q}^{\text{int}}$ \;
    $o'=(o'_j)_{j=1,\ldots,k}$ where $o'_j = o_j$ \If $\tilde{q}^\text{frac}_j=0$ \Else $0$ \;
    $P = \mset{(x_1,\ldots,x_k)}{x_i=0,\ldots,s_i-1 \text{ for all $i=1,\ldots,k$}}$ \;
    $L = [ \frac{1}{s_k} \left( V( (W^{-1}S'x + \tilde{q}^{\text{int}}) \momod{o'} s_k ) + V \tilde{q}^{\text{frac}} + s_kq \right)$ \For $ x\in P ] $ \;
    \Return $L$  
  }
}
\caption{Enumerate Fundamental Parallelepiped\label{alg:enumfundpar}}
\end{algorithm}

\subsubsection*{Barvinok decomposition}

An alternative way to convert a simplicial symbolic cone $C$ to a rational function is to construct a signed decomposition of $C$ into unimodular cones $C_i$. A cone $C_i$ is unimodular if it contains just a single lattice point in its fundamental parallelepiped, which implies that the rational function representation (\ref{eqn:Ehrhart-general}) has just a single monomial in its numerator.

\begin{figure}[t]
\input{Graphics/barvinok-sketch}
\caption[]{
  \label{fig:barvinok} Barvinok decomposition.
   $ [\CCC( \msmat{1&1\\0&a} )] =  [\CCC( \msmat{1&0\\0&1} )] - [\CCC^{(1,0)}( \msmat{0&1\\1&a} )]$
}
\end{figure}

It is crucial that we are looking for a \emph{signed} decomposition of $C$ into unimodular cones $C_i$, i.e., for a way of representing $C$ as an inclusion-exclusion of the $C_i$. While it is always possible to write $C$ as a union of unimodular cones $C_i$, the number of such cones $C_i$ required may well be exponential in the encoding size of $C$, as the left-hand side of Figure~\ref{fig:barvinok} shows: The cone generated by $(1,0)$ and $(1,a)$ has encoding size $\log a$. While it can be written as a union of unimodular cones $C_i$ generated by $(1,i)$ and $(1,i+1)$ for $i=0,\ldots,a-1$, this positive decomposition requires $a$ cones, which is exponential in $\log a$. Signed decompositions, on the other hand, allow us to write $C$ as a difference of just two unimodular cones, as shown on the right in Figure~\ref{fig:barvinok}. If $C_1$ is generated by $(1,0)$ and $(0,1)$ and $C_2$ is generated by $(0,1)$ and $(1,a)$ with openness-vector $o=(1,0)$, then $[C] = [C_1] - [C_2]$.

Barvinok's theorem \cite{Barvinok1993}, which is one of the landmark achievements in the field in the last decades, states that this works in any dimension: Every simplicial cone $C$ can be written as a signed sum of half-open unimodular cones $C_i$ such that the number of cones $C_i$ is bounded by a polynomial in the encoding size of $C$, provided that the dimension of $C$ is fixed. That is, the number of cones will grow exponentially with the dimension of the matrix $V$ of generators of $C$, but once the dimensions of $V$ are fixed, the number of cones depends only polynomially on the encoding length of the numbers in $V$. Such a decomposition is called a \emph{Barvinok decomposition}.

The construction of Barvinok decompositions is out of scope of this article. We refer the interested reader to the textbooks \cite{BarvinokIntegerPoints,DeLoera2012} for details. Briefly, the basic idea is the following. Given a simplicial $d$-dimensional cone $C$ with generators $v_1,\ldots,v_d$, use the LLL algorithm to find a vector $w$ that is ``short'' with respect to a certain basis. Then decompose $C$ into cones $C_1,\ldots,C_d$ where $C_i$ is generated by $v_1,\ldots,v_{i-1},w,v_{i+1},\ldots,v_k$. The vector $w$ may lie outside of $C$, just as in the example in Figure~\ref{fig:barvinok} where $w=(0,1)$. If $w$ lies outside of $C$, then some of the $C_i$ will have a negative sign in the decomposition (depending on whether the determinants $\det C$ and $\det C_i$ have the same sign or not). Also, some of the $C_i$ will be half-open. Which faces of the $C_i$ are open can be determined via an entirely combinatorial case-analysis, as given in \cite{Koppe2008}. The algorithm then proceeds by recursion, constructing a decomposition of each of the $C_i$ in turn.

The parameter that guarantees termination of this recurrence is the number of lattice points in the fundamental parallelepiped of $C$, which we call the \emph{index} $\ind(C)$ of $C$. In a each recursive step, a given cone $C$ is split into $d$ new cones $C_i$, until unimodular cones are obtained. This gives rise to a $d$-ary recursion tree $T$. The key property is that $w$ is ``short'' enough so that the index of all $d$ new cones is significantly smaller than the index of $C$. In fact, the index decreases so rapidly that the depth $\ell$ of $T$ is guaranteed to be doubly logarithmic in the index of $C$, more precisely $\ell\in\OOO(d\cdot \ln(\ln(\ind(C))))$. The double logarithm is crucial, since $\ln(\ind(C))$ is polynomial in the encoding size of $C$. The number of terms in the final decomposition, i.e., the number of leaves of $T$, is then bounded by $\ln(\ind(C))^{\OOO(d\cdot \ln(d))}$. If $d$ is fixed, the number of terms is thus bounded by a polynomial in the encoding size of $C$. See \cite[p.~140-141]{BarvinokIntegerPoints} for details.


In summary, Barvinok's algorithm provides us with the following result.

\begin{theorem}[Barvinok \cite{Barvinok1993}, see also \cite{Barvinok1999a,DeLoera2004a,Koppe2008}]
\label{thm:barvinok}
Let $C$ be a simplicial symbolic cone in dimension $d$ with encoding size $s$. Then there exists an $N\in\ZZnn$, unimodular $d$-dimensional half-open symbolic cones $C_1,\ldots,C_N$ with the same apex as $C$ and signs $\alpha_1,\ldots,\alpha_N\in\{-1,1\}$ such that
\[
  [C] = \sum_{i=1}^N \alpha_i [C_i]
\]
and, for any fixed $d$, $N$ is bounded by a polynomial in $s$. Moreover, let $v_{i,1},\ldots,v_{i,d}$ denote the generators of the $C_i$ and let $u_i$ denote the unique lattice point in the fundamental parallelepiped of $C_i$. Then
\[
  \Phi_C(z) = \sum_{i=1}^N \alpha_i \frac{z^{u_i}}{(1-z^{v_{i,1}}) \cdot\ldots\cdot (1-z^{v_{i,d}})}
\]
and, again, the encoding size of this rational function is bounded by a polynomial in $s$, provided that $d$ is fixed.
\end{theorem}

Comparing Theorem~\ref{thm:fundamental-parallelepiped-enumeration} with Theorem~\ref{thm:barvinok} we find that Barvinok's decomposition has the distinct advantage that the size of the rational function expressions it produces can be exponentially smaller than the expression obtained through explicit fundamental parallelepiped enumeration. This difference can be particularly pronounced when the number of variables of the linear Diophantine system is small, but the coefficients in the system are very large and give rise to cones in the decomposition that have large indices. (However, in Section~\ref{sec:related-work} we are going to see an example of a class of linear Diophantine systems with large coefficients where the indices of the cones produced in the previous step are still small.) The drawbacks of Barvinok decompositions are that they are determined by a non-trivial recursive algorithm, as opposed to the explicit formula of Theorem~\ref{thm:fundamental-parallelepiped-enumeration}, and that the number of distinct factors $(1-z^v)$ appearing in the denominators of the final rational function expression is much larger using the Barvinok approach than using fundamental parallelepiped enumeration.

It is important to remark on the dimension $d$ that drives the exponential growth of Barvinok's decomposition. In the very first step of our algorithm, we apply the MacMahon lifting and transform a linear system with $d$ variables and $m$ constraints into a $d$-dimensional cone in $(d+m)$-dimensional space. However, the $m$ extra dimensions are all eliminated during the second phase of the algorithm! Therefore the cones $C$ for which we apply Barvinok's algorithm all have dimension (at most) $d$ and lie in $d$-dimensional space again. Thus, the MacMahon-style iterative elimination approach we employ does not lead to an exponential growth of the Barvinok decomposition, because we work with symbolic cones throughout the elimination and convert to rational function only after all extra variables have been eliminated.

In practice, both approaches should be used in conjunction, as is also done in software packages like \latte~\cite{DeLoera2004a}. If a symbolic cone $C$ has a small index, fundamental parallelepiped enumeration is applied to directly obtain the rational function. If the index of $C$ is large, Barvinok's algorithm is used to recursively decompose $C$. However, instead of stopping the recurrence when $\ind(C_i)=1$, i.e., when the cones $C_i$ become unimodular, the recurrence is stopped when the indices become small enough so that fundamental parallelepipeds can be explicitly enumerated efficiently. The threshold when the index becomes small enough is implementation dependent.

\subsection{Rational Function in Normal Form}

No matter which mechanism for conversion to rational functions is used, the result is a rational function expression for $\Phi_P$, not a rational function in normal form. In a majority of applications, the desired output will be a rational function expression, as opposed to a rational function in normal form. In particular, if the linear Diophantine system has finitely many solutions, the normal form of the rational function will be a polynomial that explicitly lists all solutions, in which case it is often preferable to stick to a more compact expression. Even if the solution set is infinite, bringing the rational function in normal form -- by bringing all summands on a common denominator, expanding and canceling terms -- will typically destroy a lot of the structure of the rational function expression and may vastly increase its size. In fact, it is easy to construct families of cases where the size of the normal form is exponential in the size of the output produced by either of the conversion mechanisms presented above.

That said, there are certainly cases where computing the normal form can involve helpful cancellations that indeed simplify the rational function expressions. In cases where this behavior can be expected, it may be useful to bring the rational function in normal form or to at least perform some simplifications. Normal form algorithms are widely available in current computer algebra systems and can readily be used for this purpose. We will therefore discuss this subject no further. We merely remark that in our implementation using the \sage computer algebra system, using \sage's built-in mechanism for bringing the rational function in normal form often takes orders of magnitude longer than all the previous steps of the algorithm combined, due to the size of the expressions involved.


\section{Computational Complexity}
\label{sec:complexity}

In this section, we give a detailed complexity analysis of Polyhedral Omega, providing bounds on the running time and space requirements of the algorithm. Polyhedral Omega is the first partition analysis algorithm for which this has been done. A key reason for this is that the geometric point of view enables us to prove stronger invariants than a naive complexity analysis would allow. In particular, we prove a structure theorem that implies strong bounds on the number of symbolic cones that can appear during elimination and on their encoding size.

\subsection{Structure Theorem}

In order to prove a structure theorem for the symbolic cones that can appear during elimination inductively, we need to set up a suitable invariant that is preserved by $\ElimLC$.

For brevity, we introduce some notation for matrix manipulations. Let $\Id_{i\times j}$ denote the rectangular $(i\times j)$-matrix with ones on the diagonal and zeros everywhere else. Let $\sm{i}{j}{k}{l}{A}\in\ZZ^{(j-i+1)\times(l-k+1)}$ denote the submatrix of a matrix $A=(a_{\mu,\nu})_{\mu=1,\ldots,m; \nu=1,\ldots,n}$ consisting of rows $i,\ldots,j$ and columns $k,\ldots,l$, i.e., for $\mu=1,\ldots,j-i+1$ and $\nu=1,\ldots,l-k+1$ the entry $s_{\mu,\nu}$ of the submatrix satisfies $s_{\mu,\nu}=a_{i+\mu-1,k+\nu-1}$. If $b\in\ZZ^{m\times 1}$ is a vector, we abbreviate $\sv{i}{j}{b}=\sm{i}{j}{1}{1}{b}$. In what follows we will assume a dense representation of matrices and vectors in order to simplify the presentation. We note, though, that in practice most of the matrices are sparse. We also assume constant time for accessing any element of a matrix. Generally, the estimates we obtain in this section are coarse. In particular, we work with naive bounds on the complexity of integer and matrix multiplication. A finer analysis would thus yield tighter bounds.

Any simplicial cone can be described both in terms of its generators, as we have done in our symbolic cone representation, and in terms of a system of equations and inequalities. For the purpose of characterizing the cones that can appear during one run of $\Elim$, it will be convenient to work with inequality descriptions of symbolic cones. Given matrices $A'\in\ZZ^{m\times (d+m)}$, $A''\in\ZZ^{d\times (d+m)}$ and vectors $b'\in\ZZ^m$  $b''\in\ZZ^{d}$ such that
\[
  K = \mset{x\in\RR^{d+m}}{A' x = b' \text{ and } A'' x \geq b''}
\]
is a $d$-dimensional simplicial cone in $\RR^{d+m}$, there is a unique set of $d$ primitive integer vectors that generate $K$. Let $V$ denote the matrix with these generators as columns, let $q\in\QQ^n$ denote the unique solution of $A'q=b'$ and $A''q=b''$ and let $C=(V,q,0)$ denote the corresponding closed symbolic cone. Then $K = [C]$ and we define $C$ as the \emph{cone determined by the system $A'x=b',A''x\geq b''$}. Going one step further, we define $\Flip(C)$ as the \emph{forward cone determined by the system $A'x=b',A''x\geq b''$}. For scalars $\alpha,\beta$ and vectors $u,v$ we define
\[
  \alpha \geq^0 \beta  \Leftrightarrow  \alpha \geq \beta, \;\;\;\;
  \alpha \geq^1 \beta  \Leftrightarrow  \alpha < \beta, \;\; \text{ and } \;\;
  u \geq^o v \Leftrightarrow u_i \geq^{o_i} v_i \text{ for all }i.
\]
Note that a superscript $1$ reverses both the direction and the strictness of the inequality. Given this notation, we can observe that for $\Flip(C)=(\hat{V},\hat{q},\hat{o})$ we have
\begin{align*}
  [\Flip(C)] = \mset{x\in\RR^d}{A' x = b' \text{ and } A'' x \geq^{\hat{o}} b''}
\end{align*}
where, without loss of generality, we assume that the columns $v_i$ of $\hat{V}$ are indexed such that the generator $v_i$ lies on the extreme ray of the cone that is given by letting all the constraints $A'' x \geq^{\hat{o}} b''$ hold at equality, except for the $i$-th one.

So, the openness vector $o$ describes both, which facets of the cone are open and, correspondingly, which generators have been reversed by $\Flip$ so that the cone points forward. One important observation is that no matter how a cone has been flipped in the past, the forward cone is uniquely determined by the original system. More precisely,
\begin{align}
  \label{eqn:flip-identity}
  \Flip(\mset{x}{A'x=b', A'' x \geq^o b''}) = \Flip(\mset{x}{A'x=b', A'' x \geq b''})
\end{align}
for every vector $o$.

With these preparations we can characterize all cones that can appear throughout the elimination process.

\begin{theorem}
\label{thm:characterization-of-eliminated-cones}
Let $C=(V,q,o)$ be a symbolic cone with $V\in\ZZ^{(d+m)\times d}$, $q\in\RR^{d+m}$ and $o=0\in\{0,1\}^d$ such that the columns of $V$ are primitive and linearly independent and such that $C$ is forward. Moreover, let $C$ have an inequality description of the form
  \begin{align*}
    [C] = \mset{ x\in\RR^{d+m} }{ \mmatrix{ A & {-\Id_m} } x = b \text{ and } \mmatrix{I_d & 0}x \geq 0  }
  \end{align*}
for some $A\in\ZZ^{m\times d}$ and $b\in\ZZ^m$. Then, any symbolic cone $C'=(V',q',o')$ produced at iteration $i=0,\ldots,m$ of $\Elim(C)$ is a forward cone determined by a system of the form
  \begin{align*}
    [C'] &= \Flip(\mset{x\in\RR^{d+m-i}}{ 
      A' x = b',
      A'' x \geq b''
    }) = \mset{x\in\RR^{d+m-i}}{ 
      A' x = b',
      A'' x \geq^{o'} b''
    }
  \end{align*}
where $A' = \sm{1}{m-i}{1}{d+m-i}{ \mmatrix{ A & {-\Id_m} } } \in\ZZ^{(m-i)\times(d+m-i)}$, $b' =\sv{1}{m-i}{b}\in\ZZ^d$  and $A'' x \geq b''$ is a subsystem consisting of $d$ rows of 
\begin{align*}
  \mmatrix{ \Id_{d\times(d+m-i)} \\ \sm{m-i+1}{m}{1}{d+m-i}{ \mmatrix{ A & {-\Id_m} } } } x \geq \mvec{0_d\\ \sv{m-i+1}{m}{b} }
\end{align*}
where $0_d\in\RR^d$ denotes the zero vector. In particular, $A''\in\ZZ^{d\times (d+m-i) }$ and $b''\in\ZZ^{d}$.
\end{theorem}

The above theorem can be generalized to input cones $C$ that have open faces, but this is not necessary for our purposes.

\begin{proof}
The proof proceeds by induction over $i$. For $i=0$, i.e., before $\ElimLC$ is applied for the first time, the claim is trivially true: The claim simply asserts that $\msmat{I_d & 0}x \geq 0, \msmat{ A & {-\Id_m} } x = b$ is an inequality description of $C$ which is true by assumption.

So we assume that the claim holds for some $i$ and show that it then holds for $i+1$. Let $C'=(V',q',o')$ be a $d$-dimensional cone produced at iteration $i$ and let
\begin{align*}
  [C'] = \mset{x\in\RR^{d+m-i}}{A'x=b', A''x \geq^{o'} b''}
\end{align*}
be the system determining $C'$, where $A',b',A'',b''$ are as described in the statement of the theorem and the columns of $V'$ are indexed to fit the order of the rows of $A''$ as described above. 

At iteration $i+1$ we intersect $C'$ with the half-space given by $x_{d+m-i}\geq 0$. By assumption on $A',A''$, the variable $x_{d+m-i}$ has a non-zero coefficient only in the last row of $A'$ and can therefore be seen as a slack variable. We rewrite
\begin{align*}
 && 
  x_{d+m-i} &\geq 0  
 &&\wedge& 
  a_1x_1 + \ldots + a_dx_d - x_{d+m-i} &= b'_{m-i} \\
 \Leftrightarrow &&
  x_{d+m-i} &= a_1x_1 + \ldots + a_dx_d - b'_{m-i}  
&&\wedge& 
  a_1x_1 + \ldots + a_dx_d &\geq b'_{m-i}
\end{align*}
where the $a_i$ are the first $d$ entries of the last row of $A'$. Projecting away the coordinate $x_{d+m-i}$, we find that the polyhedron $P=\pi([C']\cap H_{d+m-i}^+)\subset\RR^{d+m-i-1}$ is given by the inequality description
\begin{align*}
P &= \mset{x\in\RR^{d+m-i-1}}{ \hat{A}'x = \hat{b}', \hat{A}''x \geq^{\hat{o}} \hat{b}'' } 
\end{align*}
where $\hat{A}' = \sm{1}{m-i-1}{1}{d+m-i-1}{A'} = \sm{1}{m-i-1}{1}{d+m-i-1}{\msmat{ A & {-\Id_m} }}$, $\hat{b}' = \sv{1}{m-i-1}{b'} = \sv{1}{m-i-1}{b}$ and $\hat{A}''x \geq^{\hat{o}} \hat{b}''$ is the system $A''x\geq^{o} b''$ with the additional constraint $a_1x_1 + \ldots +a_dx_d \geq b'_{m-i}$ added.

Consider first cases A and B from Section~\ref{sec:elimination}. As we observed, the cones produced by $\ElimLC$ are forward flips of the vertex cones of $P$. The vertex cones of a $d$-dimensional polyhedron in $(d+m-i-1)$-dimensional space are given by a system of $m-i-1$ equations and $d$ inequalities. In our case, the $m-i-1$ equations are precisely $\hat{A}'x=\hat{b'}$ and the $d$ inequalities are some subsystem of $d$ inequalities chosen from $\hat{A}''x\geq^{\hat{o}} \hat{b}''$. As shown in Section~\ref{sec:elimination}, case C is a deformation of case A. In particular, the cones produced in case C arise in the same way as in case A and hence are determined by the same type of system.

We have now seen that any cone $\hat{C}$ produced at iteration $i+1$ is of the form
\begin{align}
  [\hat{C}] &= \Flip(\mset{x\in\RR^{d+m-i-1}}{\hat{A}'x = \hat{b}', \tilde{A}''x \geq^{\tilde{o}} \tilde{b}''}) 
  \label{eqn:characterization-of-eliminated-cones}
  = \Flip(\mset{x\in\RR^{d+m-i-1}}{\hat{A}'x = \hat{b}', \tilde{A}''x \geq \tilde{b}''})
\end{align}
where $\tilde{A}''x \geq^{\tilde{o}} \tilde{b}''$ is a subsystem consisting of $d$ rows of $\hat{A}''x \geq^{\hat{o}} \hat{b}''$ and the second equality follows from (\ref{eqn:flip-identity}). The system (\ref{eqn:characterization-of-eliminated-cones}) is precisely of the form asserted in the theorem, which completes the induction proof.
\end{proof}

In particular, Theorem~\ref{thm:characterization-of-eliminated-cones} applies to the symbolic cones $C$ given by the MacMahon lifting.

\begin{lemma}
\label{lem:macmahon-h-description}
Let $A\in\ZZ^{m\times d}$ and $b\in\ZZ^m$, then the cone $C=(V,q,0)$ produced by $\MacMahon(A,b)$ has the inequality description
\[
    [C] = \mset{x\in\RR^{d+m}}{  \mmatrix{A & -I_m}x = b, \mmatrix{I_d & 0}x \geq 0}.
\]
\end{lemma}

\begin{proof}
Recall $V=\msmat{\Id_d \\ A}$ and $q=\msmat{0_d \\ -b}$. The elements of $C$ have the form $q+V\alpha$ for $\alpha\in\RR^d_{\geq 0}$. We compute $\msmat{\Id_d & 0}(q+V\alpha) = \alpha \geq 0$ and $\msmat{A & -I_m}(q+V\alpha) = b$ which completes the proof.
\end{proof}

To use Theorem~\ref{thm:characterization-of-eliminated-cones} for our complexity analysis, we need to bound the size of a symbolic cone in terms of the size of its inequality description.

\begin{lemma}
\label{lem:encoding-size-h-to-v-description}
Let $C$ be a $d$-dimensional forward symbolic cone in $\RR^{d+m-i}$ with primitive generators that is determined by the system
\[
    [C] = \Flip(\mset{x\in\RR^{d+m-i}}{A'x = b', A''x \geq b''}).
\]
Let the encoding size of all the entries in $A',b',A'',b''$ be bounded above by $s$. Then the encoding size of $C$ is bounded above by $2(d+1)(d+m-i)^2\log(d+m-i) s + d$ with the size of each entry of $V$ and $q$ bounded by 
\[
  L_i := 2(d+m-i)\log(d+m-i)s.
\]
\end{lemma}

While $L_i$ is a function of $d,m,i,s$, we will in the following write just $L_i$ whenever the values of $d,m,s$ are understood from context.

\begin{proof}
Let $C=(V,q,o)$. The columns of $V$ are of the form $\prim(v)$ where the vectors $v$ are solutions to $Av = e_j$ for $A=\msmat{A' \\ A''}\in\ZZ^{(d+m-i)^2}$. Similarly, $q$ is the solution to $Aq=b$ where $b=\msmat{b'\\b''}$. All the entries of the $v$ and $q$ are therefore quotients of determinants $\det(A_k)$ and $\det(A)$, where the matrices $A_k$ are obtained by replacing a column of $A$ by $e_j$ or $b$, as appropriate.

By Hadamard's inequality, $|\det(A)| \leq \prod_{j} ||a_j||_2 \leq \prod_{j} ||a_j||_1$ where the $a_j$ are the columns of $A$. Therefore, the encoding size of $\det(A)$ is at most $(d+m-i)\log(d+m-i)s$. The same holds for the determinants $\det(A_k)$. The encoding size of the quotient $\frac{\det(A_k)}{\det(A)}$ is therefore at most $2(d+m-i)\log(d+m-i)s$, which is thus also an upper bound on the encoding size of the entries of the vectors $v$ and $q$. 

Note that the vectors $v$ are rational, while $\prim(v)$ are always integer. In particular, $\prim(v)$ may actually be longer than $v$. However, the entries of $\prim(v)$ are factors of the integers $\det(A_k)$, whence their encoding size is bounded above by $(d+m-i)\log(d+m-i)s$. Therefore, $2(d+m-i)\log(d+m-i)s$ gives a uniform bound on the encoding size of the entries of $V$ and $q$ and the total the encoding size of $C$ is at most $2(d+1)(d+m-i)^2\log(d+m-i)s + d$.
\end{proof}

Putting these results together we obtain bounds on the number and encoding size of the symbolic cones produced during elimination. These bounds are much tighter than a naive analysis of the algorithm would imply, due to the geometric insight used. In particular, the encoding sizes of all cones produced during the iterative procedure can be bounded in terms of the encoding size of the original system. Here we use that rational simplicial cones have a unique representation as a symbolic cone with primitive generators (which are stored in sorted order).

\begin{corollary}
\label{cor:cone-sizes}
Let $C$ be as in Theorem~\ref{thm:characterization-of-eliminated-cones}. Then the following hold:
\begin{enumerate}
    \item The number of symbolic cones produced after iteration $i=0,\ldots,m$ of $\Elim(C)$ is at most $\binom{d+i}{d}$.
    \item For each cone $C'=(V',q',o')$ produced after iteration $i$ of $\Elim(C)$, the encoding size of the entries of $V'$ and $q'$ is bounded by $L_i=2(d+m-i)\log(d+m-i)s$, where $s$ is a bound on the encoding size of each entry in the inequality description of $C$.
    \item In particular, for $C=\MacMahon(A,b)$, the above holds if $s$ is a bound on the encoding size of each entry in $A,b$.
\end{enumerate}
\end{corollary}

\begin{proof}
(1) The cones produced at iteration $i$ each correspond to an inequality description as given in Theorem~\ref{thm:characterization-of-eliminated-cones}. Since there are only $\binom{d+i}{d}$ ways to select $d$ rows from a matrix with $d+i$ rows, there can be at most $\binom{d+i}{d}$ distinct cones.

(2) At step $i$, we know by Theorem~\ref{thm:characterization-of-eliminated-cones} there exists an inequality description of $C'$ such that the coefficients in this inequality description are bounded by $s$. By Lemma~\ref{lem:encoding-size-h-to-v-description} the unique primitive symbolic cone representation $C'=(V',q',o')$ has the property that the entries of $V',q'$ are bounded by $2(d+m-i)\log(d+m-i)s$.

(3) Follows from the (2) and Lemma~\ref{lem:macmahon-h-description}.
\end{proof}

\subsection{Complexity of Lifting and Elimination}

We can now give a complexity analysis of the algorithms in Section~\ref{sec:macmahon-lifting} and \ref{sec:elimination}.

\begin{lemma}[Complexity of MacMahon Lifting, Algorithm~\ref{alg:macmahon}]
\label{lem:complexity-macmahon}
On input $A\in\ZZ^{m\times d}$, $b\in\ZZ^m$ with entries of encoding size at most $s$, $\MacMahon$ runs in \Ob{d m s + d^2} time using \Ob{d m s + d^2} space.
\end{lemma}

\begin{proof}
The MacMahon lifting is trivial and in essence only requires the time and space to read the input and write down the result.
\end{proof}

\begin{lemma}[Complexity of Flip, Algorithm~\ref{alg:flip}]
\label{lem:complexity-flip}
Given a symbolic cone $C=(V,q,o)$ with $V\in\ZZ^{n\times k}$, $q\in\QQ^n$ and $o\in\{0,1\}^k$, computing $\Flip(C)$ and $\sgn(C)$ takes time \Ob{k n} time and no extra space.
\end{lemma}
\begin{proof}
The flip of a symbolic cone requires for each generator to run through its entries until we find the first non-zero entry. If it is negative, the multiply the rest of the entries and the sign of the cone by $-1$ and perform a $\operatorname{xor}$ operation on $o$. This requires \Ob{k n} time and no extra space.
\end{proof}

\begin{lemma}[Complexity of Prim, Algorithm~\ref{alg:prim}]
\label{lem:complexity-prim}
Given a matrix $V\in\ZZ^{n\times k}$ with entries of encoding size at most $s$, $\prim$ runs in time $\Ob{kns^2}$ and $\Ob{kns}$ space.
\end{lemma}

\begin{proof}
To make each column primitive, we compute the greatest common divisor of all the entries in the column and divide. For every column, the greatest common divisor computation and the division take time \Ob{ns^2} each.
\end{proof}

\begin{lemma}[Complexity of Eliminate Last Coordinate, Algorithm~\ref{alg:eliminatelastcoordinate}]
\label{lem:complexityElimlC}
Let $i=1,\ldots,m$. Let $C=(V,q,o)$ denote a symbolic cone satisfying the conditions given in Theorem~\ref{thm:characterization-of-eliminated-cones} for one of the cones $C'$ produced at iteration $i-1$ of $\Elim$ when applied to an input system $A,b$ in which each entry has encoding size bounded above by $s$. In particular, $C$ is a $d$-dimensional cone in $(d+m-i+1)$-dimensional space. Then, computing $\ElimLC(C)$ requires at most time 
\[
\Ob{d^2(d+m-i)^3 \log(d)\log(d+m-i)^2 s^2 }
\]
and space $\Ob{d(d+m-i+1)(d+1)L_{i-1}}$.
\end{lemma}

\begin{proof}
By Theorem~\ref{thm:characterization-of-eliminated-cones} and Corollary~\ref{cor:cone-sizes}, we know that the encoding sizes of the entries in $V$ and $q$ are bounded above by $L_{i-1}$.

Computing $\ElimLC(C)$, we first construct the index set $J$ in \Ob{d} time, represented as a binary vector using $d$ bits.
Then we construct the vectors $w_j$ for the appropriate $j$'s. 
For each entry of $w_j$ we perform 6 multiplications, 1 subtraction, 1 greatest common divisor computation and 2 exact divisions, 
where all the initial quantities had encoding length bounded by $L_{i-1}$. 
The time complexity for the computation of each entry is \Ob{L_{i-1}^2} and the space needed is $6 L_{i-1}+1$.
Thus, for the computation of $w_j$ we need \Ob{(d+m-i+1)L_{i-1}^2} time and $ (d+m-i+1) (6 L_{i-1}+1)$ space.

Computing an entry of the matrix $VT^j$ amounts to taking a linear combination of two elements from $V$ with coefficients
from $T^j$. The time complexity for each entry is \Ob{L_{i-1}^2} and the encoding size of an entry of of the linear combination is bounded
by $2L_{i-1}+1$. 
This is repeated for all $d(d+m-i+1)$ entries of the matrix $VT^j$, resulting in time complexity \Ob{d(d+m-i+1)L_{i-1}^2}.

For each of the columns of the new matrix, we forget the last coordinate, which costs nothing and does not affect the bounds in the encoding size. Then we use $\prim$ to make the columns of this truncated $(d+m-i)\times d$-matrix primitive, which takes time at most \Ob{d(d+m-i)L_{i-1}^2} by Lemma~\ref{lem:complexity-prim}. The resulting matrix is the generator matrix of a symbolic cone after the $i$-th elimination step, and thus by Theorem~\ref{thm:characterization-of-eliminated-cones} and Corollary~\ref{cor:cone-sizes} the encoding size of its entries are bounded by $L_i$. Flipping the resulting cone forward takes an additional time of at most $\Ob{d(d+m-i)}$ by Lemma~\ref{lem:complexity-flip}. In total, computing one symbolic cone takes time
\begin{align*}
\Ob{(2d+1)(d+m-i+1)L_{i-1}^2+2d+m-i+2} 
\leq \Ob{d(d+m-i+1)L_{i-1}^2}.
\end{align*}

In order to make comparison of cones faster in the next step, we sort the generators of the symbolic cone lexicographically. This requires $d \log(d)$ comparisons of vectors of encoding length $(d+m-i)L_i$.

Finally, inserting the at most $d$ resulting symbolic cones in a dictionary mapping a symbolic cone to its multiplicity costs \Ob{d\log(d)(d+1)(d+m-i)L_i}, since the comparison of two symbolic cones costs \Ob{(d+1)(d+m-i)L_i}.

The total complexity of eliminating the last coordinate is thus 
\begin{align*}
\Ob{d^2(d+m-i) \left( L_{i-1}^2 + \log(d) L_i \right) } 
&\leq \Ob{d^2(d+m-i) \log(d) L_{i-1}^2} \\
&\leq \Ob{d^2(d+m-i)^3 \log(d)\log(d+m-i)^2 s^2 }
\end{align*}
where, in the first step, we use $L_{i-1}\geq L_i$ and, in the last step, we substitute $L_{i-1}=2(d+m-i)\log(d+m-i)s$.
\end{proof}

\begin{lemma}[Complexity of Eliminate Coordinates, Algorithm~\ref{alg:eliminatecoordinates}]
\label{lem:complexity-elim}
Let $C=(V,q,o)$ denote a symbolic cone satisfying the conditions given in Theorem~\ref{thm:characterization-of-eliminated-cones}. 
Then, computing \Elim{$C$} requires at most time 
\[
\Ob{ s^2 \binom{d+m-1}{d} (d+m)^3 d^2 m^2 \log(d)^2 \log(d+m)^2 }.
\]
and space $\Ob{ s \binom{d+m}{d} (d+m)^2 d \log(d+m) }$.
\end{lemma}

\begin{proof}
At elimination step $i-1$, we obtain $\binom{d+i-1}{d}$ symbolic cones. For each of them we compute a dictionary containing $d$ symbolic cones using Algorithm~\ref{alg:eliminatelastcoordinate} in \Ob{d^2(d+m-i)^3 \log(d)\log(d+m-i)^2 s^2 } time. Then we insert each cone appearing in these dictionaries into a new dictionary that will contain at most $\binom{d+i}{d}$ cones. Comparing two cones produced in elimination step $i$ costs \Ob{d(d+m-i) L_i}. Using appropriate data structures, the complexity of inserting an element in a dictionary of length $\binom{d+i}{d}$ is $\log\left(\binom{d+i}{d}\right)\in \Ob{i \log(d)}$ comparisons. Thus, inserting a cone in the dictionary costs $\OO(i d \log(d) (d+m-i) L_i)$. We have to insert in the dictionary the $d \binom{d+i-1}{d}$ symbolic cones obtained after eliminating the $i$-th coordinate in the $\binom{d+i-1}{d}$ symbolic cones obtained in the previous step. In total, the complexity of obtaining the dictionary at elimination step $i$ is
\begin{align*}
 \Ob{ \binom{d+i-1}{d} i d^2 \log(d) (d+m-i) L_i} 
 & \leq \Ob{ \binom{d+i-1}{d} i d^2 (d+m-i)^2 \log(d) \log(d+m-i)  s } . 
\end{align*}
The cost of computing \ElimLC{$C$} for each of the $\binom{d+i-1}{d}$ symbolic cones is $\OO(d^2(d+m-i)^3 {\log(d)} \allowbreak {\log(d+m-i)^2} s^2 )$ from Lemma~\ref{lem:complexityElimlC}. Thus, elimination step $i$ has a total cost of 
\begin{align*}
&\mathbb{O} \left(  \binom{d+i-1}{d} \left( i d^2(d+m-i)^2 \log(d) \log(d+m-i) s + d^2(d+m-i)^3 \log(d)\log(d+m-i)^2 s^2 \right) \right) \\
\leq \;\; &
\mathbb{O} \left( \binom{d+i-1}{d} i d^2 (d+m-i)^3 \log(d)^2 \log(d+m-i)^2 s^2 \right)
\end{align*}
Iterating for $i=1,2,\ldots,m$ results in time complexity 
\begin{align*}
\Ob{  \binom{d+m-1}{d} (d+m)^3 d^2 m^2 \log(d)^2 \log(d+m)^2 s^2 }.
\end{align*}
Regarding the space used we observe that in each iteration we need to store at most $\binom{d+m}{d}$ cones and that each cone has encoding size at most $\Ob{d(d+m)L_0}$, which yields the desired bound.
\end{proof}

\subsection{Complexity of Rational Function Conversion}

Next, we analyze how long it takes to explicitly enumerate the lattice points in the fundamental parallelepiped of a symbolic cone $C=(V,q,o)$, using the Smith normal form as described in $\EnumFundPar$ (Algorithm~\ref{alg:enumfundpar}). The total number of lattice points we have to enumerate is $D=|\det(V)|$ which can be exponential in the encoding length of $V$. As the size of the output can be exponential in the size of the input, it makes sense to give an output-sensitive analysis of the algorithm. In particular, the following lemma states the \emph{preprocessing time} of $\EnumFundPar$, i.e., the time it takes to perform a one-off calculation at the beginning of the algorithm, and then the \emph{delay}, which is the time it takes to generate the next lattice point in the list of lattice points in the fundamental parallelepiped. It turns out that the preprocessing time and the delay are polynomial in the input size and independent of the output size.

\begin{lemma}[Complexity of Enumerate Fundamental Parallelepiped]
\label{lem:complexity-fundamental-parallelepiped-enumeration}
Let $C=(V,q,o)$ denote a $d$-dimensional symbolic cone in $\QQ^d$ such that the encoding size of each entry in $V$ and $q$ is bounded by $s$. Listing the $D=|\det(V)|$ lattice points in the fundamental parallelepiped of $C$ using $\EnumFundPar$ as defined in Algorithm~\ref{alg:enumfundpar} takes preprocessing time $\Ob{d^9(\log(d)+s)^2}$ and has delay $\Ob{d^4(\log(d)+s)^2}$ resulting in a total running time of $$\Ob{d^9(\log(d)+s)^2 + D d^4(\log(d)+s)^2 }.$$
The space required during preprocessing and generation of individual lattice points is polynomial in $s$ and $d$ and independent of the total number $D$ of generated lattice points.
\end{lemma}

\begin{proof}
The preprocessing phase of $\EnumFundPar$ consists of computing the Smith normal form and then computing $\hat{q}$,  $W^{-1}S'\hat{q}$, $\tilde{q}^{\mathrm{int}}$, $\tilde{q}^{\mathrm{frac}}$ and $-V\tilde{q}^{\mathrm{frac}}+s_kq$. Since we assume that $C$ is full-dimensional, $p$ does not have to be computed, but can be assumed to be $p=0$. These operations involving $q$ are relatively expensive as we need to work with rational vectors. Bounds on the running time required for computing the Smith normal form $S$, $U^{-1}$ and $W^{-1}$, as well as bounds on the encoding size of the entries in these matrices are taken from \cite{Storjohann1996,Storjohann2000}. In total, this gives a bound of $\Ob{d^9(\log(d)+s)^2}$ on the running time of the preprocessing step.

To generate the list of all lattice points, one lattice point at a time, we do not first generate $x\in P$ and then compute $W^{-1}S'x$, as this would be inefficient. Instead, we observe that the lattice points in $P$ can be listed in such a way that two successive entries in the list differ by a unit vector. Thus we can list the lattice points $W^{-1}S'P$ by starting at $0$ and then adding or subtracting one column of $W^{-1}S'$ in each step. After the modulus has been taken, multiplication with $V$ is required, though. It is also important to note that the vectors $V( W^{-1}S'x \momod{o'} s_k )$ and $-V\tilde{q}^{\mathrm{frac}}+s_kq$ are both integer and that the final division by $s_k$ is guaranteed to produce an integer vector. Since the final result is guaranteed to lie in the fundamental parallelepiped of $C$, the encoding size of every entry in every single output vector is bounded by $s+d$. In total, this gives a bound of $\Ob{d^4(\log(d)+s)^2}$ on the time taken to generate the next point in the list.

While generating lattice points, only one point in $W^{-1}S'P$ has to be stored at a time. Therefore the space required is independent of the total number of points generated.
\end{proof}

To analyze the complexity of computing a Barvinok decomposition, we cite the following result due to Barvinok \cite{Barvinok1993,Barvinok1994,Barvinok1999a}, see \cite{BarvinokIntegerPoints,DeLoera2012} for details.

\begin{lemma}[Complexity of Barvinok Decomposition \cite{Barvinok1994}]
\label{lem:complexity-barvinok}
Let $C=(V,q,o)$ be a $d$-dimensional symbolic cone in $\QQ^d$ and let $s$ be a bound on the encoding size of all entries in $V$ and $q$. Let $D=\det(V)$ denote the number of lattice points in the fundamental parallelepiped of $C$. Then a list of the symbolic cones $C_1,\ldots,C_N$ as given in Theorem~\ref{thm:barvinok} can be computed in time $t(d,s)$ which is a polynomial in $s$ for fixed $d$. Moreover, the number of cones satisfies
\[
    N \leq d\cdot (\log D)^{\frac{\log d}{\log \frac{d}{d-1}}} \approx d\cdot (\log D)^{(d-\frac{1}{2})\log d} 
\]
which is a polynomial in $s$ for fixed $d$, since $\log D \leq d s + \frac{d}{2} \log(d)$.
\end{lemma}

Unfortunately, an explicit formula for the bound $t(d,s)$ is not available in the literature. The derivation of such a bound, while straightforward in principle, is out of the scope of this article.

\begin{proof}
The time bound, originally obtained in \cite{Barvinok1994}, is given in \cite[Theorem~7.1.10]{DeLoera2012}. We use the bound on the depth of the decomposition tree from \cite[Lemma~7.1.9]{DeLoera2012} to get
\begin{eqnarray*}
    N \leq d^{\floor{1 + \frac{\log \log D}{\log \frac{d}{d-1}}}} 
      \leq d(d^{\log \log D})^{\frac{1}{\log \frac{d}{d-1}}}
      \leq d(\log D)^{\frac{\log d}{\log \frac{d}{d-1}}} 
      \approx d (\log D)^{d\log d}
\end{eqnarray*}
as an upper bound on the number of unimodular cones in the final decomposition.
\end{proof}

The complexity of both fundamental parallelepiped enumeration and Barvinok decomposition depend on the number $D=|\det(V)|$ of lattice points in the fundamental parallelepiped of a given symbolic cone $C=(V,q,o)$.

\begin{lemma}
\label{lem:fundamental-parallelepiped-bound}
Let $a\in\ZZ_{\geq0}$ be an upper bound on the absolute value of the entries of the original linear Diophantine system $Ax\geq b$. Then for each cone $C=(V,q,o)$ output by $\Elim(\MacMahon(A,b))$ the number of lattice points in the fundamental parallelepiped of $C$ is bounded above by 
\[
  |\det(V)| \leq (da)^{d^2}.
\]
\end{lemma}

\begin{proof}
Let $C=(V,q,o)$ be a symbolic cone output by $\Elim(\MacMahon(A,b))$. By Theorem~\ref{thm:characterization-of-eliminated-cones} and the proof of Lemma~\ref{lem:encoding-size-h-to-v-description}, $V$ is of the form $V=\prim(\tilde{A}^{-1})$ where $\prim(M)$ denotes the matrix obtained by applying $\prim$ to the columns of $M$ and $\tilde{A}$ is a $d\times d$ integer matrix with entries that are bounded in absolute value by $a$. The inverse $\tilde{A}^{-1}$ has rational entries with common denominator $\det(\tilde{A})$. Thus $\det(\tilde{A})\tilde{A}^{-1}$ is an integer matrix and each column of $\det(\tilde{A})\tilde{A}^{-1}$ is a positive multiple of the corresponding column of $\prim(\tilde{A}^{-1})$. So
\[
  |\det(\prim(\tilde{A}^{-1}))| \leq |\det(\det(\tilde{A})\tilde{A}^{-1})| = |\det(\tilde{A})|^{d} \cdot |\det(\tilde{A}^{-1})| = |\det(\tilde{A})|^{d-1}.
\]
Again, using the Hadamard bound we obtain $|\det(\tilde{A})| \leq (da)^d$. Combining these inequalities we get $|\det(V)|\leq (da)^{d^2}$ as desired.
\end{proof}

\subsection{Summary of Complexity Analysis}

The results of the above complexity analysis are summarized in the following theorem.

\begin{theorem}
\label{thm:complexity-summary}
Let $A\in\ZZ^{m \times d}$ and $b\in\ZZ^m$ such that the encoding size of all entries in $A$ and $b$ is bounded above by $s$. Then the different variants of the algorithm described in Section~\ref{sec:polyhedral-omega-motivation} have the following running times.
\begin{enumerate}

\item \label{itm:compsum-elim} Computing a list of symbolic cones $C_1,\ldots,C_n$ with multiplicities $\alpha_1,\ldots,\alpha_n$ such that $$[\mset{x}{Ax\geq b, x\geq 0}] = \sum_i \alpha_i [C_i]$$ using $\MacMahon$ and $\Elim$ takes time at most 
\[
  \Ob{ s^2 \binom{d+m-1}{d} (d+m)^3 d^2 m^2 \log(d)^2 \log(d+m)^2 }.
\]
This is a polynomial in the parameters, provided at least one of $d$ or $m$ are fixed. The number of points $n_i$ in the fundamental parallelepiped of cone $C_i$ is bounded above by $n_i \leq  (d2^s)^{d^2}$ and the total number of cones is bounded above by $n \leq \binom{d+m}{d}$.

\item \label{itm:compsum-ratfun} After the list of symbolic cones has been computed, there are two ways to obtain a rational function representation of the desired generating function $\Phi_{\mset{x}{Ax\geq b, x\geq 0}}(z)$.
  \begin{enumerate}

  \item \label{itm:compsum-fundpar} Using fundamental parallelepiped enumeration, a rational function expression of the form
  \[
    \Phi_{\mset{x}{Ax\geq b, x\geq 0}}(z) = \sum_{i=1}^n \alpha_i \frac{\sum_{j=1}^{n_i} z^{u_{i,j}}}{\prod_{j=1}^d (1-z^{v_{i,j}}) }
  \]
  can be emitted iteratively as follows: The $\alpha_i$ and $v_{i,j}$ are given directly by the list of symbolic cones. After preprocessing time $\Ob{nd^{11}\log(d)^2 s^2}$ the at most $\sum_j n_j$ vectors $u_{i,j}$ can be listed one at a time with delay $\Ob{d^6\log(d)^2s^2}$. In total, the time taken for computing the entire expression is thus at most
  \[
    \Ob{  \binom{d+m}{d}  d^{d^2+11}  2^{sd^2}  s^2  \log(d)^2 }.
  \]

  \item \label{itm:compsum-barvinok} Using the Barvinok decomposition, a rational function expression of the form
  \[
    \Phi_{\mset{x}{Ax\geq b, x\geq 0}}(z) = \sum_{i=1}^N \epsilon_i \frac{z^{t_i}}{\prod_{j=1}^d (1-z^{w_{i,j}}) }
  \]
  can be computed in time at most
  \[
    \Ob{ \binom{d+m}{d} t(d,2d\log(d)s) }
  \]
  where $t(d,s)$ is the expression from Lemma~\ref{lem:complexity-barvinok} which is a polynomial in $s$ if $d$ is fixed. The number $N$ of summands is at most
  \begin{eqnarray*}
    N &\leq& \binom{d+m}{d} d (d^2  \log(d)  s  )^{\frac{\log d}{\log\frac{d}{d-1} } } \approx \binom{d+m}{d} d (d^2  \log(d) s )^{ (d-1/2) \log d }.
  \end{eqnarray*}
  \end{enumerate}

\end{enumerate}
\end{theorem}

All the bounds stated in this theorem are rather coarse and could be improved by a refined analysis. Note also that, by construction, $N\geq n$, i.e., the Barvinok decomposition will produce always at least as many summands in the outer sum. Yet, typically $N$ will be exponentially shorter than $\sum_i n_i$. In (\ref{itm:compsum-fundpar}), the exponents $v_{i,j}$ will be the generators of the symbolic cones produced in (\ref{itm:compsum-elim}) and thus edge directions of the arrangement of hyperplanes given by the original system. In (\ref{itm:compsum-barvinok}), the set of exponents $w_{i,j}$ will typically be much larger, since the  $w_{i,j}$ arise from the $v_{i,j}$ through a recursive arithmetic procedure.

\begin{proof}
(\ref{itm:compsum-elim}) follows from Corollary~\ref{cor:cone-sizes} and Lemmas~\ref{lem:complexity-macmahon}, \ref{lem:complexity-elim}, \ref{lem:fundamental-parallelepiped-bound}. For (\ref{itm:compsum-fundpar}), we combine Corollary~\ref{cor:cone-sizes} and Lemma~\ref{lem:complexity-fundamental-parallelepiped-enumeration}. For (\ref{itm:compsum-barvinok}), we combine Lemmas~\ref{lem:complexity-barvinok} and \ref{lem:fundamental-parallelepiped-bound}.
\end{proof}

We conclude this complexity analysis with a couple of final remarks. (\ref{itm:compsum-elim}) runs in polynomial time if either $d$ or $m$ is fixed. (\ref{itm:compsum-barvinok}) runs in polynomial time if $d$ is fixed. (\ref{itm:compsum-fundpar}) does not run in polynomial time if $d$ is fixed, due to the term $2^s$ which arises from enumerating fundamental parallelepipeds which may be exponentially large.  This means that using (\ref{itm:compsum-elim}) and (\ref{itm:compsum-barvinok}) we obtain an algorithm that solves the rfsLDS problem in polynomial time, if the dimension of the problem is fixed. Thus, the Polyhedral Omega algorithm lies in the same complexity class as the best known polyhedral algorithms for solving rfsLDS. None of the previous partition analysis algorithms is known to have this property, see Section~\ref{sec:comparison-first-poly-time}.

A key bottleneck is the number of symbolic cones output by $\Elim$ for which we have the upper bound $\binom{d+m}{d}$. This upper bound is polynomial if either $d$ or $m$ are fixed, but can be exponential otherwise, e.g., $\binom{d+d}{d} \approx \frac{4^d}{\sqrt{\pi d}}$. A bottleneck of this type is inherent in all algorithms that have to touch each vertex of the polyhedron defined by the inequality system at least once (e.g., when using Brion decompositions): By McMullen's Upper Bound Theorem \cite{McMullen1970}, we know that the maximal number of vertices that a polytope with a given number of facets can have is realized by duals of cyclic polytopes. In our context, this yields inequality systems with $$\binom{d+m-\ceil{\frac{d}{2}}}{\floor{\frac{d}{2}}} + \binom{d+m-1-\ceil{\frac{d-1}{2}}}{\floor{\frac{d-1}{2}}}$$ vertices \cite{Ziegler}. It is important to remark that the actual number of symbolic cones is typically much lower than $\binom{d+m}{d}$, but can be larger than the number of vertices, even in the case of simple polyhedra. See Section~\ref{sec:comparison-polyhedral} for further discussion and open research questions related to this phenomenon.


\section{Comparison with Related Work}
\label{sec:related-work}

In this section, we compare Polyhedral Omega with previous algorithms from both partition analysis and polyhedral geometry.

\subsection{Speedup via Symbolic Cones}
\label{sec:comparison-symbolic-cones}

Partition analysis algorithms are characterized by the iterative application of a collection of explicitly given formulas. However, iteration can appear in the various algorithms on many different levels and in many different forms, which can have important performance implications. 

\begin{table}
\begin{tabular}{ l | l | l | l | l | l }
                  &  Elim   & RF      & Gn   & P   & G2 \\
\hline
Elliott/MacMahon  &  R      &         &      & R   & R  \\
Omega2            &  R      &         &      & R   & L  \\
Ell               &  R      &         &      & R   & L  \\
CTEuclid          &  R      &         & R    &     &    \\
PolyOmega-FP      &  R      & L       &      &     &    \\
PolyOmega-BD      &  R      & R       &      &     &
\end{tabular}
\caption{\label{tab:iteration}Iteration in partition analysis algorithms. See text for abbreviations.}
\end{table}

Table~\ref{tab:iteration} gives an overview of the iterative structure of the classic algorithm by Elliott/MacMahon, the Omega2 algorithm by Andrews--Paule--Riese \cite{PA6}, Xin's algorithms Ell \cite{Xin2004} and \cteuclid \cite{Xin2012} and the two variants of our Polyhedral Omega algorithm, using fundamental parallelepiped enumeration (PolyOmega-FP) and Barvinok decompositions (PolyOmega-BD). Iteration can take either the form of a recurrence (R) or the form of a loop (L). By a loop we mean explicit formulas involving a long summation. The defining feature of all partition analysis algorithms is the iterative elimination (Elim) of one variable at a time using a recursion that splits one rational function/cone into several rational functions/cones. In analogy with geometry, we will call the binomial factors in the denominator of a rational function \emph{generators}. Within the elimination of one variable, the algorithms have very different structure. Elliott's algorithm uses an outer recurrence over pairs of generators (P) as well as an inner recurrence, similar to the Euclidean algorithm, to apply $\omeg$ to one pair of generators (G2). A key contribution of Omega2 was to replace the inner recurrence by an explicit formula (i.e., a loop) leading to a significant speedup. Xin's Ell algorithm has a similar structure. In contrast, both Polyhedral Omega and Xin's \cteuclid algorithm do not iterate over pairs of generators, but apply $\omeg$ to all $n$ generators at once (Gn). While \cteuclid uses a recursive procedure in this phase, Polyhedral Omega does not need to use any iteration at all here, due to the use of symbolic cones. In Polyhedral Omega, the corresponding iteration can happen \emph{outside the elimination recurrence}, when converting symbolic cones to rational functions (RF). The fact that in Polyhedral Omega it is not necessary to iterate at all inside a single elimination can already provide an exponential speedup in comparison with other partition analysis methods, as we will see in a concrete example below. The conversion to rational functions can be performed after elimination is complete using either fundamental parallelepiped enumeration (i.e., an explicit summation formula) or the Barvinok decomposition (i.e., a recursive procedure).

Polyhedral Omega is the only partition analysis algorithm for which a comprehensive complexity analysis is available. Furthermore, Polyhedral Omega is the only partition analysis algorithm whose elimination phase runs in polynomial time if either the number of variables or the number of constraints is fixed, and, if Barvinok decompositions are used, it is the only partition analysis algorithm that, in total, runs in polynomial time if the dimension is fixed. These are not just better bound due to a refined analysis that exploits the geometric structure, but genuine exponential speedups, even if Barvinok decompositions are not used. To illustrate these performance gains, we will discuss two structural differences between the algorithms that lead to exponential improvements on large classes of instances.

First, it is generally desirable to handle all generators at once, instead of iterating over pairs of generators. Omega2, for example, acts on a rational function by taking a product of two generators with positive last coordinate and splitting it into a sum of two rational functions that each have one generator with positive last coordinate and one generator with zero last coordinate. Now, assume $d$ is even and $\rho$ is a rational function with $d$ generators, half of which have positive and half of which have negative last coordinate. If we recursively apply the Omega2 rule until each summand has at most one generator with positive last coordinate (which is a base case of the Omega2 algorithm), we end up with a total of $2^{d/2-1}$ summands. In contrast, Polyhedral Omega's $\ElimLC$ produces at most $d+1$ summands in the worst case.

Second, and most importantly, using symbolic cones during elimination can lead to exponentially more compact intermediate representations, which in turn implies exponential speedups during the elimination phase, compared with algorithms working with rational functions. This can save a lot of work in total, as soon as several variables need to be eliminated, as the following example shows. Let $a_1$ be any large positive integer and let $a_2$ be any positive integer coprime to $a_1$. Let $b_1,b_2$ denote positive integers such that $a_1b_2 - a_2b_1 =1$. Now consider the linear Diophantine system $S$ consisting of the two constraints $b_2x_1 - b_1x_2 \geq 0$ and $-a_2x_2+a_1x_1 \geq 1$. If we solve $S$ iteratively by, in the first step, computing a full rational function representation of the set of all Diophantine solutions $x\geq0$ to the inequality $-a_2x_2+a_1x_1 \geq 0$, we will typically obtain a rational function expression such as
\begin{eqnarray*}
\label{eqn:iteration-example}
 \rho = \frac{\sum_{i=0}^{a_1-1}x^{(i,1+\floor{\frac{a_2}{a_1}i})}}{(1-x^{(a_1,a_2)})(1-x^{(0,1)})},
\end{eqnarray*}
or a lifted version thereof, which has $a_1$ terms in the numerator. This is because the fundamental parallelepiped of the corresponding cone $\CCC_\RR(\msmat{a_1 & 0 \\ a_2 & 1})$ contains $a_1$ lattice points. Using rational function representations, the first elimination has thus increased the encoding size of the problem exponentially, which will slow down subsequent eliminations correspondingly. Using symbolic cones, on the other hand, the encoding size does not increase during elimination: Polyhedral Omega starts with the symbolic cone $$C_0=\CCC_\RR(\msmat{ 1 & 0 \\ 0 & 1 \\ b_2 & -b_1 \\ -a_2 & a_1};0).$$ The first elimination produces $$C_1=\CCC_\RR(\msmat{ a_1 & 0 \\ a_2 & 1 \\ 1 & -b_1}; \msmat{0 \\ \frac{1}{a_1} \\ -\frac{b_1}{a_1}} ).$$ As expected, the projection of $C_1$ onto the first two coordinates is the solution set of the system $x\geq 0$ and $-a_2x_2+a_1x_1 \geq 1$. Note that, from $C_0$ to $C_1$, the determinant of the cone has increased from $1$ to $a_1$ but the encoding size has not. The second elimination then produces $$C_2 = \CCC_\RR(\msmat{a_1 & b_1 \\ a_2 & b_2};\msmat{b_1 \\ b_2}),$$ which is precisely the solution set of the system $S$. The cone $C_2$ has again determinant 1, so that the fundamental parallelepiped can be enumerated in just a single iteration and the resulting rational function expression is short. Since we used symbolic cones we did not have to enumerate the fundamental parallelepiped explicitly at the intermediate stage and can get to the short end-result quickly. Rational function representations, however, produce an exponential blowup at the intermediate stage that can prevent partition analysis algorithms from reaching the short final result. 

Of course there are shorter expressions for the rational function $\rho$ given in (\ref{eqn:iteration-example}), which can be obtained, e.g., using Barvinok decomposition or methods from \cite{BreuerVonHeymann,Xin2012}. This would avoid an \emph{exponential} blowup in one elimination step. But the increase in representation size in a single elimination step would still be larger than using symbolic cones. It is thus preferable to apply Barvinok techniques only after elimination has been completed, as we do in Polyhedral Omega. Initial experiments with our implementation of Polyhedral Omega \cite{PolyhedralOmegaCode} confirm these benefits of symbolic cone representations in practice.

As the above example shows, symbolic cones are the key for obtaining Polyhedral Omega's improvements in running time. However, symbolic cones are not tied to the particulars of our algorithm. On the contrary, symbolic cones have great potential for further applications in partition analysis and beyond. One particularly promising direction for future research is to develop a symbolic cone version of partial fraction decomposition, following the polyhedral interpretation of PFD given in Section~\ref{sec:xin-pfd}. This might lead in particular to a polyhedral version of Xin's \cteuclid algorithm with improved performance when several variables are eliminated.

\subsection{First Polynomial-Time Algorithm from Partition Analysis Family}
\label{sec:comparison-first-poly-time}

Polyhedral Omega is the first algorithm from the partition analysis family for which a comprehensive complexity analysis is available (Section~\ref{sec:complexity}). Geometric insight played a crucial role in obtaining strong bound on the running time. Most importantly, by Theorem~\ref{thm:complexity-summary}, Polyhedral Omega runs in polynomial time if the dimension is fixed, provided that the Barvinok decomposition is used. It is the first partition analysis algorithm with this property.

This latter statement requires some discussion. First, it is important to draw the following careful distinction: The \emph{Barvinok decomposition} is an algorithm for computing a short signed decomposition of a simplicial cone into unimodular simplicial cones, which runs in polynomial time if the dimension is fixed. More generally, however, there is \emph{Barvinok's algorithm} for solving the rfsLDS problem. Barvinok's algorithm makes crucial use of Barvinok decompositions, but performs a whole range of other tasks as well, including but not limited to computing vertices and edge directions as well as triangulation.

As we mentioned in Section~\ref{sec:omega-vs-lds}, \cite[Section 3]{Xin2012} describes a procedure for reducing the problem of applying $\omeg$ to a rational function to the rfsLDS problem. The rfsLDS problem itself, though, is solved by using Barvinok's algorithm as a black box. In our view, the procedure from \cite[Section 3]{Xin2012} does therefore not fall into the domain of partition analysis algorithms. In contrast to \cite[Section 3]{Xin2012}, Polyhedral Omega does not use all of Barvinok's algorithm as a black box, but only Barvinok decompositions. The Barvinok decomposition itself can be described in terms of a recursive formula \cite{Koppe2008}, very much in the spirit of partition analysis. Thus, we argue, Polyhedral Omega with Barvinok decompositions (PolyOmega-BD) is the first rfsLDS algorithm that follows the partition analysis paradigm and achieves polynomial running time in fixed dimension.

The \cteuclid algorithm given in \cite[Section 4]{Xin2012} is quite separate from the algorithm from \cite[Section 3]{Xin2012} mentioned above, and does not use Barvinok decompositions at all. Xin raises the question whether \cteuclid runs in polynomial time if just a single variable is eliminated. This question remains open. However, as Xin acknowledges, even if the answer was positive, polynomiality would likely break down as soon as several variables are eliminated one after another. Here symbolic cones could be of help as explained in Section~\ref{sec:comparison-symbolic-cones}.

\subsection{Comparison with Polyhedral Methods}
\label{sec:comparison-polyhedral}

In comparison with polyhedral methods, the main benefit of Polyhedral Omega is simplicity. Polyhedral algorithms for rfsLDS typically make use of a number of non-trivial and highly-specialized algorithms for enumerating vertices and edges of a polyhedron and for triangulating a cone, since these tasks are essential for decomposing a polyhedron into simplicial cones. In contrast, Polyhedral Omega obtains such a decomposition implicitly through the iterative application of a few simple rules. This fact is the main insight transferred from partition analysis to polyhedral geometry as part of this research project. Crucially, this simple iterative approach still results in an algorithm that lies in the same complexity class as the best known polyhedral methods for solving rfsLDS: Polyhedral Omega runs in polynomial time in fixed dimension. 

\begin{figure}[t]
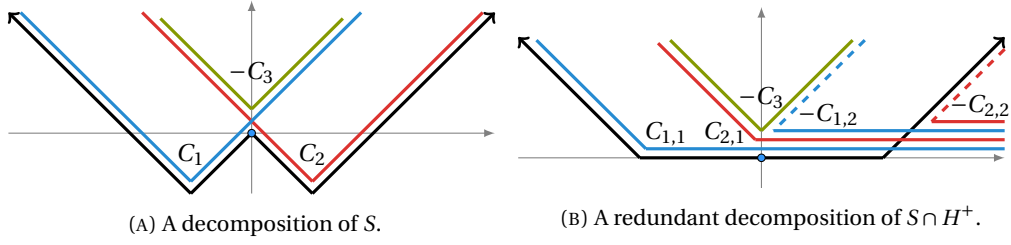


\newcommand{\zerosumsize}{0.8}
  \begin{subfigure}{0.45\textwidth}
    \centering
    \input{Graphics/zero-sum-2.tex}
    \caption{A decomposition of $S$.}
  \end{subfigure} 
  \begin{subfigure}{0.45\textwidth}
    \centering
    \input{Graphics/zero-sum-3.tex}
    \caption{A redundant decomposition of $S\cap H^+$.}
  \end{subfigure} 
  \caption{  \label{fig:example-zero-sums} This example shows how zero sums of symbolic cones can arise during elimination. Consider the area $S$ above and including the ``W''-shaped line. (A) gives a symbolic cone decomposition $[S] = [C_1] + [C_2] - [C_3]$. Intersecting each of these symbolic cones with the half-space $H^+$ above the horizontal axis gives the decomposition $[S\cap H^+] = [C_{1,1}] - [C_{1,2}] + [C_{2,1}] - [C_{2,2}] - [C_3]$ shown in (B). Note that $[\emptyset] = [C_{2,1}] - [C_3] - [C_{1,2}]$, but these cones do not cancel when collecting terms. Especially in higher dimension, situations like these can arise even if the starting set $S$ is convex and no cut goes through the apex of a cone.}
\end{figure}

However, there is one interesting regard in which there is a conceptual gap between the performance of Polyhedral Omega and polyhedral algorithms using specialized methods for vertex and edge enumeration. Applying our iterative approach, we obtain a representation 
\[
  [P] = \sum_{i=1}^N \alpha_i [C_i]
\]
of a polyhedron $P$ as a sum of (indicator functions of) symbolic cones $C_i$with multiplicities $\alpha_i\in\ZZ$. As we discussed in Section~\ref{sec:elimination} normalizing and collecting terms $[C_i]$ is crucial for performance. However, over the course of the algorithm, there may still arise groups of symbolic cones that sum to the empty set, i.e.,
\[
  \sum_{i=k}^l \alpha_i [C_i] = [\emptyset],
\]
without our algorithm being able to recognize this fact and simplify the sum accordingly. To see how such zero sums can arise, consider the example given in Figure~\ref{fig:example-zero-sums}. The underlying issue here is that on the rational function level such zero sums could be recognized by bringing the rational function expression in normal form. This calls for the development of a symbolic method for bringing symbolic cone expressions into a normal form. Such a method could speed up Polyhedral Omega significantly and is a promising topic for future research.

\section{Summary and Future Research}
\label{sec:conclusion}

In this article, we have given a detailed polyhedral interpretation of existing partition analysis algorithms solving the rfsLDS problem (Section~\ref{sec:partition-analysis}) and used these insights to develop Polyhedral Omega (Section~\ref{sec:polyhedral-omega-motivation}), a new algorithm solving rfsLDS which combines the polyhedral and the partition analysis approaches in a common framework. A key insight is the use of symbolic cones instead of rational functions, which allows us to retain the simple iterative approach based on explicit formulas that is inherent to partition analysis algorithms and still achieve a running time that had previously been attained only by polyhedral algorithms (Sections~\ref{sec:elimination} and \ref{sec:comparison-symbolic-cones}). For conversion from symbolic cones to rational functions, we present two methods (Section~\ref{sec:rat-fun-conversion}). The first, based on fundamental parallelepiped enumeration (Theorem~\ref{thm:fundamental-parallelepiped-enumeration}), has the advantage of being described by a simple closed formula, while the second, based on Barvinok decompositions (Theorem~\ref{thm:barvinok}), is a recursive algorithm that yields short rational function expressions. When using Barvinok decomposition, Polyhedral Omega runs in polynomial time in fixed dimension (Theorem~\ref{thm:complexity-summary}). Thus it lies in the same complexity class as the fastest known polyhedral algorithms for rfsLDS and it is the first algorithm from the partition analysis family to do so (Section~\ref{sec:comparison-first-poly-time}). Moreover the geometric point of view allows us to do the first comprehensive complexity analysis of any partition analysis algorithm, and obtain strong bounds on the running time from the geometric invariants we establish (Theorem~\ref{thm:characterization-of-eliminated-cones}). Compared to previous polyhedral methods, the key advantage of Polyhedral Omega is its simplicity: instead of using several sophisticated algorithmic tools from polyhedral geometry, a simple recursive algorithm based on a few explicit formulas suffices.

Our results point to several promising directions for future research.

On the practical side, the next task is to complete our implementation of Polyhedral Omega \cite{PolyhedralOmegaCode}. So far, we have implemented the elimination algorithm using symbolic cones as well as rational function conversion using fundamental parallelepiped enumeration on the \sage system. The next step is the implementation of Barvinok decompositions on this platform. This would enable a comprehensive benchmark comparison of the different rfsLDS algorithms available today. In this context, it would also be desirable to do a detailed complexity analysis of the Barvinok decomposition, giving an explicit formula for an upper bound on the running time of that algorithm. While such an analysis is straightforward in principle, no such formula has been published so far.

On the theoretical side, the most interesting direction for future research is a detailed study of symbolic cones. Symbolic cones form an interesting algebraic structure closely related to, but distinct from both, the rational function field and the algebra of polyhedra \cite{BarvinokIntegerPoints}. In particular, a good definition of a normal form of a symbolic cone expression along with an efficient algorithm for bringing a symbolic cone expression into normal form would be of great use, both in the context of Polyhedral Omega (see Section~\ref{sec:comparison-symbolic-cones}) as well as for working with rational functions in general. Moreover, symbolic cones may be of use outside the context of Brion/Lawrence--Varchenko decompositions. Especially our polyhedral interpretation of partial fraction decomposition in Section~\ref{sec:xin-pfd} calls for closer investigation. A symbolic cone variant of partial fraction decomposition could, for example, lead to a polyhedral version of the \cteuclid algorithm.

Finally, a key feature of partition analysis algorithms has always been that they lend themselves very well for use in induction proofs, due to their being based on the recursive application of explicit rules. Polyhedral Omega promises to be of great use in both manual and automatic induction proofs, since the intermediate results obtained are given as symbolic cones and therefore more compact and easy to manipulate. We are especially looking forward to combining Polyhedral Omega with automatic induction provers, as this may lead to algorithms for proving structural results about infinite families of polytopes of varying dimension and even to algorithms for manipulating infinite-dimensional polyhedra. A first milestone in this direction would be to finally realize one of the visions that motivated Andrews--Paule--Riese to revive partition analysis in the first place: a fully automatic proof of the Lecture Hall Partition Theorem \cite{PA3,BME} in full generality, with the dimension treated as a symbolic variable.

\section*{Acknowledgements}
The authors are grateful to Matthias Beck and Peter Paule for bringing them together in this research project, 
as well as for their assistance.
Moreover, the authors would like to thank Matthias K\"oppe and Fu Liu for helpful discussions.

\newpage

\bibliographystyle{abbrv}
\bibliography{polyomega}

\end{document}

%% file: colors.tex
\usepackage{color}

\definecolor{textColor}{RGB}{240 240 240}
\definecolor{bgColor}{RGB}{0 0 0}

\definecolor{exampleFillCol}{RGB}{50 150 50}
\definecolor{alertFillCol}{RGB}{170 70 70}
\definecolor{greenFillCol}{RGB}{50 150 50}
\definecolor{redFillCol}{RGB}{150 50 50}
\definecolor{blueFillCol}{RGB}{50 50 150}

\definecolor{greenColor}{RGB}{133 153 00}
\definecolor{redColor}{RGB}{220 50 47}
\definecolor{blueColor}{RGB}{38 139 210}

\definecolor{pictureFillCol}{RGB}{50 50 50}
\definecolor{pictureDarkFillCol}{RGB}{50 50 50}

\definecolor{subtitleColor}{RGB}{200 200 50}

\definecolor{alertColor}{RGB}{240 50 50}

\definecolor{edgeColor}{RGB}{42 161 152}

\definecolor{rayColor}{RGB}{42 161 152}
\definecolor{vertexColor}{RGB}{50 150 250}

\definecolor{facetColor}{RGB}{150 150 150}
\definecolor{intersectionColor}{RGB}{250 10 10}

\definecolor{solCyan}{RGB}{42 161 152}
\definecolor{solRed}{RGB}{220 50 47}
\definecolor{solGreen}{RGB}{133 153 00}
\definecolor{solBlue}{RGB}{38 139 210}

\definecolor{cyan}{RGB}{42 161 152}
\definecolor{red}{RGB}{220 50 47}
\definecolor{green}{RGB}{133 153 00}
\definecolor{blue}{RGB}{38 139 210}
\definecolor{magenta}{RGB}{211 54 130}
\definecolor{yellow}{RGB}{181 137 0}
\definecolor{orange}{RGB}{203 75 22}
\definecolor{violet}{RGB}{108 113 196}

%% file: graphics.tex
\definecolor{point}{RGB}{50 50 50}
\definecolor{alertPoint}{RGB}{50 50 50}
\definecolor{focusPoint}{RGB}{50 50 50}
\definecolor{outoffocusPoint}{RGB}{50 50 50}

\tikzstyle{highlight}=[fill=alertColor]

\newcommand{\point}[2]{
  \coordinate (#2) at #1;
}

\newcommand{\labelvertexn}[4]{
  \node [draw,circle,inner sep=1pt,fill=vertexColor,label={[anchor=#2]#3: #4}] at (#1) {};
}

\newcommand{\vertex}[1]{
  \node [draw,circle,inner sep=1pt,fill=vertexColor] at (#1) {};
}

\newcommand{\alertvertex}[1]{
  \node [draw,circle,inner sep=2pt,fill=alertColor] at (#1) {};
}

\newcommand{\edge}[2]{
  \draw [very thick,color=edgeColor, -] (#1) -- (#2);
}

\newcommand{\colorededge}[3]{
  \draw [very thick,color=#3, -] (#1) -- (#2);
}

\newcommand{\dashededge}[2]{
  \draw [very thick, dashed, color=edgeColor] (#1) -- (#2);
}

\newcommand{\coldashededge}[3]{
  \draw [dashed, very thick, color=#3] (#1) -- (#2);
}

\newcommand{\ray}[2]{
  \draw [color=rayColor, very thick, ->] (#1) -- (#2);
}

\newcommand{\coloredray}[3]{
  \draw [color=#3, very thick, ->] (#1) -- (#2);
}
\newcommand{\openray}[3]{
  \draw [dashed,color=#3, very thick, ->] (#1) -- (#2);
}

\newcommand{\openedge}[2]{
  \draw [dashed,color=edgeColor, very thick, -] (#1) -- (#2);
}

\newcommand{\alertray}[2]{
  \draw [color=alertColor, very thick, ->] (#1) -- (#2);
}

\newcommand{\trifacet}[3]{
  \fill[pattern=north west lines, pattern color=facetColor, opacity=0.5, rounded corners=0pt] (#1)--(#2)--(#3);
}

\newcommand{\rectfacetcolor}[5]{
  \fill[color=facetColor, opacity=#5, rounded corners=0pt] (#1)--(#2)--(#3)--(#4)--(#1);
}

\newcommand{\suphyperplane}[4]{
   \draw[fill,  pattern=north west lines, pattern color=alertColor, opacity=0.8, rounded corners=0pt] (#1)--(#2)--(#3)--(#4)--(#1);
}

\newcommand{\polyh}[5]{
   \draw[draw=none, fill, pattern=north west lines, pattern color=facetColor, opacity=0.8, rounded corners=0pt] (#1)--(#2)--(#3)--(#4)--(#5)--(#1);
}

\newcommand{\polyt}[4]{
   \draw[draw=none, fill, pattern=north west lines, pattern color=facetColor, opacity=0.8, rounded corners=0pt] (#1)--(#2)--(#3)--(#4)--(#1);
}

\newcommand{\zgaxis}[7]
{
    \coordinate (Origin)   at (0,0,0);
    \coordinate (XAxisMin) at (-#1,0,0);
    \coordinate (XAxisMax) at (#2,0,0);
    \coordinate (YAxisMin) at (0,-#3,0);
    \coordinate (YAxisMax) at (0,#4,0);
    \ifthenelse{\equal{#5}{}}
    {}{
      \coordinate (ZAxisMin) at (0,0,-#5);
      \coordinate (ZAxisMax) at (0,0,#6);
    }

    \draw [thin, gray,-latex] (XAxisMin) -- (XAxisMax) ;
    \draw [thin, gray,-latex] (YAxisMin) -- (YAxisMax);
    \ifthenelse{\equal{#5}{}}
    {}{
    \draw [thin, gray,-latex] (ZAxisMin) -- (ZAxisMax);
    }

    \ifthenelse{\equal{#5}{}}
    {}{    
      \coordinate (XLabel) at (#2+0.3,0,0);
      \coordinate (YLabel) at (0,#4+0.3,0);
      \coordinate (ZLabel) at (0,0,#6+0.3);
      \node at (XLabel) {$x$};
      \node at (YLabel) {$y$};
      \node at (ZLabel) { \ifthenelse{\equal{#7}{}}{$\lambda$}{#7} };   
    }
}

\newcommand{\zgaxisz}[7]
{
    \coordinate (Origin)   at (0,0,0);
    \coordinate (XAxisMin) at (-#1,0,0);
    \coordinate (XAxisMax) at (#2,0,0);
    \coordinate (YAxisMin) at (0,-#3,0);
    \coordinate (YAxisMax) at (0,#4,0);
    \ifthenelse{\equal{#5}{}}
    {}{
      \coordinate (ZAxisMin) at (0,0,-#5);
      \coordinate (ZAxisMax) at (0,0,#6);
    }

    \draw [thin, gray,-latex] (XAxisMin) -- (XAxisMax) ;
    \draw [thin, gray,-latex] (YAxisMin) -- (YAxisMax);
    \ifthenelse{\equal{#5}{}}
    {}{
    \draw [thin, gray,-latex] (ZAxisMin) -- (ZAxisMax);
    }

    \ifthenelse{\equal{#5}{}}
    {}{    
      \coordinate (XLabel) at (#2+0.3,0,0);
      \coordinate (YLabel) at (0,#4+0.3,0);
      \coordinate (ZLabel) at (0,0,#6+0.3);
      \node at (XLabel) {$x$};
      \node at (YLabel) {$y$};
      \node at (ZLabel) { \ifthenelse{\equal{#7}{}}{$z$}{#7} };   
    }
}